\documentclass[11pt]{article}

\usepackage[top=2cm,bottom=2cm,left=2cm,right=2cm]{geometry}
\usepackage{amsmath, amsthm, amssymb, amsfonts}
\usepackage{mathrsfs}
\usepackage{graphicx}
\usepackage{subfiles}
\usepackage{caption}
\usepackage{subcaption}
\usepackage{booktabs}
\usepackage{listings}
\usepackage{xcolor}
\usepackage{multirow}
\usepackage{algorithmic}
\usepackage{algorithm}
\usepackage{array}
\usepackage{float}
\usepackage[colorlinks=true, citecolor=blue, urlcolor=blue, linkcolor=blue]{hyperref}
\usepackage{cleveref}
\usepackage[]{natbib}
\bibpunct{(}{)}{;}{a}{}{,}
\usepackage{authblk}

\theoremstyle{plain}
\newtheorem{theorem}{Theorem}[section]
\newtheorem{proposition}[theorem]{Proposition}
\newtheorem{lemma}[theorem]{Lemma}
\newtheorem{corollary}[theorem]{Corollary}

\theoremstyle{definition}
\newtheorem{definition}[theorem]{Definition}

\newtheorem{remark}[theorem]{Remark}

\title{\textbf{Hub Neighbor-Degree Diagnostics for Sparse Random Graphs}}

\author[a]{Qian Hui \thanks{\href{mailto:qhui24@m.fudan.edu.cn}{qhui24@m.fudan.edu.cn}.}}
\author[a,b]{Tiandong Wang \thanks{Corresponding author, \href{mailto:td_wang@fudan.edu.cn}{td\underline{ }wang@fudan.edu.cn}.}}

\affil[a]{Shanghai Center for Mathematical Sciences, Fudan University}
\affil[b]{Shanghai Academy of Artificial Intelligence for Science}

\date{}

\begin{document}
\maketitle

\begin{abstract}
Networks with nearly identical degree distributions can place their hubs in sharply different neighborhoods. We develop a model diagnostic based on the mean degree of the neighbors of a degree-$k$ vertex. Under rank-one inhomogeneous random graphs, this statistic has degree-invariant centering and $k^{-1/2}$ fluctuations. Under non-rank-one kernels, posterior uncertainty about the root type can instead determine both centering and scale. Under linear preferential attachment, the statistic grows as $(m+\delta)\log k$. We turn these model-specific limits into goodness-of-fit tests for specified sparse-graph nulls and a weighted log-degree slope test for residual hub-neighborhood trends. Simulations evaluate null calibration, degree-distribution misspecification, and power against degree-matched preferential-attachment alternatives. Applications to high-school contact and arXiv coauthorship networks show that the method separates level misspecification from disassortative and positive residual trends. Reddit interaction networks provide a further appendix example.
\end{abstract}

\noindent\textbf{Keywords:} Neighbor degree, local weak convergence, inhomogeneous random graph, preferential attachment, sparse networks.

\section{Introduction}
\label{sec:intro}

Hubs often carry the scientific interpretation of a network. In a Reddit interaction graph, a high-degree user may broadcast to many low-activity users or interact mainly with other highly active users. In a school contact graph, a high-degree student may bridge classes or concentrate contacts within one classroom. In a coauthorship graph, a prolific author may link otherwise separate groups or collaborate inside an already dense circle. These patterns imply different mechanisms, yet they can share nearly the same degree distribution. This paper develops a formal diagnostic for deciding how a fitted sparse-graph model places its hubs.

The problem arises because two standard explanations of degree heterogeneity agree marginally but disagree locally. Rank-one inhomogeneous random graphs, including Chung-Lu and related conditionally Poissonian models, generate heterogeneous degrees from latent propensities \citep{chung2002connected,norros2006conditionally}. Preferential attachment generates them through network growth and cumulative advantage \citep{price1976general,barabasi1999emergence,bollobas2001degree,krapivsky2000connectivity}. A Chung--Lu weight tail index $a_0$ and a linear preferential attachment parameter satisfying $2+\delta/m=a_0$ produce the same limiting degree power law. Consequently, agreement in the degree tail does not determine whether observed hubs are embedded as the fitted mechanism predicts.

This ambiguity has a direct statistical consequence. A degree-based check can accept a sparse null even when the model systematically assigns the wrong neighbors to high-degree vertices. Under a rank-one IRG, the neighbors of a degree-$k$ root follow a size-biased type law whose mean does not grow with $k$. Under an assortative non-rank-one kernel, conditioning on a large degree changes the posterior law of the root type and may induce a degree-dependent mean. Under linear PA with $\delta>0$, a large degree identifies an old root, and its mix of older and younger neighbors produces logarithmic growth. Thus a degree-matched null may fail through its level, its trend over hub degrees, or both. Each failure supports a different interpretation of the underlying network.

We study this disagreement through the conditional mean neighbor degree
\[
  \bar D_k=\frac1k\sum_{u\in\mathcal N(\varnothing)}D(u),
  \qquad D_\varnothing=k,
\]
and ask whether its observed profile agrees with the profile implied by a specified sparse-graph null. The target is deliberately local and conditional. It does not attempt to recover a complete formation history from one graph. Instead, it tests a concrete implication of the fitted model after degree heterogeneity has been accounted for. When the null fails, the residual slope against $\log k$ supplies a signed diagnosis. A positive slope points in the PA direction, a negative slope indicates disassortative hub neighborhoods, and a flat slope agrees with the centered rank-one prediction.

Operationally, the input is one observed network together with a fitted sparse null, such as Chung--Lu or a covariate-adjusted variant. The output is a small set of calibrated summaries: whether hub-neighborhood levels agree with the null, whether the remaining residuals trend with degree, and which direction that trend takes. The procedure is therefore intended as a diagnostic for the fitted null, not as a universal test of all possible network mechanisms.

Several literatures motivate this target but leave the testing problem open. Methods for fitting and testing power-law degree tails establish whether the marginal degree sequence is compatible with heavy-tailed behavior \citep{clauset2009power,broido2019scale}. Tail-index and likelihood-based procedures for PA estimate features of the growth rule \citep{jeong2003measuring,wan2017fitting,wang2019hill,cirkovic2024modeling}. These methods answer questions about marginal degrees or attachment parameters. Our null comparison is instead degree matched and concerns the placement of neighbors around hubs.

Degree assortativity and average nearest-neighbor degree provide the closest descriptive antecedents \citep{newman2002assortative,newman2003mixing,pastorsatorras2001dynamical}. Existing theory analyzes neighbor-degree correlations in growing networks \citep{barrat2005rate,krot2017assortativity} and the behavior of average nearest-neighbor degree in configuration-type scale-free graphs \citep{newman2001random,yao2018average}. Those results establish that neighbor-degree profiles contain information beyond the degree distribution. They do not provide the model-specific centering and variance needed to test a fitted IRG null, nor do they separate an overall level error from a signed residual trend among degree-$k$ hubs.

Our inferential route also differs from existing network goodness-of-fit procedures. Simulation-based checks compare observed structural statistics with graphs drawn from a fitted model \citep{hunter2008goodness}, while spectral tests target broader departures from low-rank structure \citep{lei2016goodness,bickel2016hypothesis,le2017concentration}. Such methods are valuable omnibus checks, but a rejection need not identify how the model misplaces hubs. We use local weak limits \citep{benjamini2011recurrence,aldous2004objective,van2024random} to derive the null law of a degree-conditional statistic and to attach a mechanistic sign to its residual trend.

The paper makes four contributions. First, we derive conditional asymptotics for $\bar D_k$ under IRG and PA local limits. The IRG analysis separates neighbor-sampling noise from posterior uncertainty about the latent root type and thereby identifies distinct rank-one and non-rank-one regimes. The PA analysis proves that the neighbor-degree sum has order $k\log k$ and limit coefficient $m+\delta$. Second, we convert these limits into tests that separate null adequacy from mechanism diagnosis. Methods~1--3 test predicted levels across high-degree classes, whereas Method~4 tests a weighted residual slope against $\log k$. Third, simulations identify both the sensitivity gain from pooling under rank-one mean shifts and its finite-sample cost under a heterogeneous non-rank-one null. All four methods detect degree-matched PA alternatives, while Method~4 supplies directional evidence. Fourth, the applications show that aggregation and class adjustment can change the school-network diagnosis and that similar level failures can accompany positive arXiv or negative Reddit residual trends.

In this paper, Section~\ref{sec:locallimits} presents the IRG and PA local limits used for calibration. Section~\ref{sec:neighbourdegree} derives the conditional asymptotics of $\bar D_k$ and states the resulting model contrast. Section~\ref{sec:methods} develops the level and log-degree slope tests. Section~\ref{sec:simulation} evaluates the procedures under calibrated nulls and matched PA alternatives. Section~\ref{sec:realdata} analyzes the high-school and coauthorship networks. Section~\ref{sec:conclusion} concludes. The appendix reports baseline comparisons, proofs, and the Reddit analysis. Throughout, $\varnothing$ denotes the root, $D_\varnothing$ its degree, and $\mathcal N(\varnothing)$ its neighborhood. We write $\Phi$ for the standard normal distribution function and $z_{1-\alpha}$ for its $(1-\alpha)$ quantile.

\section{Local Limits of Random Graph Models}
\label{sec:locallimits}

\subsection{Local weak convergence}

Local weak convergence formalizes the neighborhood law seen from a uniformly chosen vertex in a large graph \citep{benjamini2011recurrence,aldous2004objective}. For a rooted graph $(G,o)$ and radius $r\geq 1$, let $B_r(G,o)$ be the subgraph induced by vertices within distance $r$ of $o$. A sequence $\{G_n\}$ converges \emph{locally in probability} to $(G_\infty,o_\infty)$ if, for every fixed $r$ and finite rooted graph $H$,
\[
  \frac{1}{n}\sum_{v\in V(G_n)}\mathbf{1}\!\bigl\{B_r(G_n,v)\cong H\bigr\}
  \;\xrightarrow{\mathbb{P}}\;
  \mathbb{P}\bigl(B_r(G_\infty,o_\infty)\cong H\bigr),
\]
where $\cong$ denotes rooted-graph isomorphism. Thus the empirical law of fixed-radius neighborhoods converges to the law around the limiting root.

This framework is well suited to sparse IRG and PA models because their typical fixed-radius neighborhoods converge to random trees. The branching laws of these trees retain the local mechanism needed for our hub diagnostic.

\subsection{Inhomogeneous random graphs}

The \emph{inhomogeneous random graph} $\mathrm{IRG}_n(\kappa_n)$ is constructed as follows. Each vertex $i\in[n]$ is assigned an independent type $X_i$ sampled from a probability measure $\mu$ on a type space $\mathcal{S}$. Conditional on the types, distinct unordered pairs are connected independently with probability
\[
  p_{ij} = \min\!\left\{\frac{\kappa_n(X_i,X_j)}{n},1\right\},
\]
and $A_{ii}=0$. The symmetric kernel $\kappa_n$ encodes affinity between types. The Chung--Lu model is the rank-one case \citep{chung2002connected,norros2006conditionally}
$$
\kappa(x,y)=\frac{w(x)w(y)}{\mathbb{E}[W]},
$$
where $w:\mathcal S\to(0,\infty)$ is a weight function. General kernels, including non-rank-one cases, fall under the Bollob\'{a}s--Janson--Riordan framework \citep{bollobas2007phase}. In either case, latent types generate degree heterogeneity, and vertices with the same type have stochastically identical neighborhoods.

A type-$x$ vertex has asymptotically Poisson degree with mean
\[
\lambda(x)=\int_{\mathcal{S}}\kappa(x,y)\mu(\mathrm{d}y)
\]
which may depend on $x$. The following local convergence result supplies the branching-process representation used below.
\begin{theorem}[Local convergence of $\mathrm{IRG}_n(\kappa_n)$,
  {\citet[Chap.~3]{van2024random}}]
\label{thm:IRGlimit}
Assume that $\kappa_n$ is an irreducible graphical kernel converging to a limiting kernel $\kappa$. Then $\mathrm{IRG}_n(\kappa_n)$ converges locally in probability to a unimodular multi-type marked Galton-Watson tree, in which:
\begin{itemize}
  \item[(i)] the root has type drawn from $\mu$;
  \item[(ii)] a vertex of type $x$ independently has offspring distribution $\mathrm{Poisson}(\lambda(x))$ with
  $$
  \lambda(x)=\int_{\mathcal{S}}\kappa(x,y)\,\mu(\mathrm{d}y);
  $$
  \item[(iii)] each of its offspring receives an independent type $y$ with probability proportional to $\kappa(x,y)\,\mu(\mathrm{d}y)$, i.e., from the distribution
  $$
  Q_x(\mathrm{d}y)=\frac{\kappa(x,y)\mu(\mathrm{d}y)}{\lambda(x)}.
  $$
\end{itemize}
\end{theorem}

The local limit has a direct statistical interpretation. In a large sparse IRG, the neighborhood of a typical vertex is approximated by a branching process whose offspring and neighbor-type distributions are determined by the root type. The distribution $Q_x$ governs which types are likely to be found among the neighbors of a type-$x$ vertex.

\subsection{Preferential attachment}

We use the loop-free simple version of linear preferential attachment $\mathrm{PA}_n^{(m,\delta)}$, where $m\geq1$ and $\delta>-m$. Starting from the complete graph, each new vertex selects $m$ distinct existing vertices sequentially with probabilities proportional to $D(u)+\delta$. Self-loops and parallel edges are excluded, and the fixed finite seed does not affect the local limit. Smaller $\delta$ concentrates attachment more strongly on existing hubs, whereas larger $\delta$ spreads attachment more evenly.

Unlike an IRG, PA creates degree heterogeneity through arrival order. Early vertices have more time to acquire edges and are therefore more likely to become hubs. The P\'{o}lya point tree retains this temporal asymmetry in the local limit \citep{van2024random}.

Each vertex in the P\'olya point tree has a birth time $U\in(0,1)$, with smaller values corresponding to earlier arrival. Although the root birth time is marginally uniform, conditioning on its degree yields the Beta law in Lemma~\ref{lem:posterior-U}. Its neighbors divide into $m$ older vertices chosen when the root arrived and younger vertices that later selected the root. Their different age and degree laws create the hub-neighborhood signature that is absent from a rank-one IRG.

The finite PA graph has the following local limit.
\begin{theorem}[Local convergence of $\mathrm{PA}_n^{(m,\delta)}$,
  {\citet[Chap.~5]{van2024random}}]
\label{thm:PAlimit}
Fix $m\geq 1$ and $\delta>-m$. The loop-free preferential attachment model $\mathrm{PA}_n^{(m,\delta)}$ defined above converges locally in probability to the P\'{o}lya point tree. Standard PA variants that differ only in their finite-seed, self-loop, or parallel-edge conventions have the same local limit.
\end{theorem}

The limiting P\'{o}lya point tree is simple and loop free. We use the same convention in the theory and simulations, so every neighbor degree counts the root--neighbor edge once.

\section{Asymptotics of the Mean Neighbor Degree}
\label{sec:neighbourdegree}

For a root $\varnothing$ with degree $D_\varnothing=k$, define its \emph{mean neighbor degree} by
\[
  \bar{D}_k := \frac{1}{k}\sum_{u\in\mathcal{N}(\varnothing)}D(u).
\]
We derive its conditional law as $k\to\infty$ under the IRG and PA local limits. Sections~\ref{subsec:irg-results} and~\ref{subsec:pa-results} derive the asymptotics under the IRG and PA limits, respectively. After that,  section~\ref{subsec:contrast} states the key contrast.

\subsection{Mean neighbor degree concentrations in IRG}
\label{subsec:irg-results}

We work throughout under the local limit of Theorem~\ref{thm:IRGlimit}. Let $X_\varnothing$ denote the type of the root, and define the expected degree and conditional type distribution of a neighbor:
\[
    \lambda(x):=\int_{\mathcal S}\kappa(x,y)\,\mu(\mathrm{d}y),
    \qquad
    Q_x(\mathrm{d}y):=\frac{\kappa(x,y)\,\mu(\mathrm{d}y)}{\lambda(x)}.
\]
The Galton-Watson tree structure supplies the two conditional independence facts used throughout this subsection. Given $(X_\varnothing=x,D_\varnothing=k)$: (i) the neighbor types $Y_1,\ldots,Y_k$ are independent, each with distribution $Q_x$; and (ii) given $Y_i=y_i$, the forward offspring count $\xi_i$ of neighbor $i$ is $\mathrm{Poisson}(\lambda(y_i))$, independent across $i$. We record these properties for the proofs \citep[Chap.~3]{van2024random}.
\begin{align}
    \mathbb{P}(Y_1\in A_1,\ldots,Y_k\in A_k\mid X_\varnothing=x,D_\varnothing=k)
    &=\prod_{i=1}^k Q_x(A_i), \label{eq:neighbor-type}\\
    \mathbb{P}(\xi_i=n\mid X_\varnothing=x,D_\varnothing=k,Y_i=y)
    &=e^{-\lambda(y)}\frac{\lambda(y)^n}{n!}. \label{eq:offspring}
\end{align}
Since neighbor $i$ has degree $D(i)=1+\xi_i$, the pairs $(Y_i,D(i))$ are conditionally independent given $(X_\varnothing=x,D_\varnothing=k)$. The mean neighbor degree is $\bar D_k=k^{-1}\sum_{i=1}^k D(i)$, an average of $k$ conditionally i.i.d.\ terms.

The first result computes the conditional mean and variance of a single neighbor degree.

\begin{proposition}[Conditional mean and variance of neighbor degree]
\label{prop:irg-cond-mv}
Given $(X_\varnothing=x,\,D_\varnothing=k)$, define
\[
  \rho(x)
  := \int_S\lambda(y)\,Q_x(\mathrm{d}y)
  = \frac{\displaystyle\int_S\lambda(y)\kappa(x,y)\,\mu(\mathrm{d}y)}{\lambda(x)},
  \quad
  \sigma^2(x)
  := \rho(x)+\mathrm{Var}_{Q_x}\!\left(\lambda(Y)\right).
\]
Then for each $i=1,\ldots,k$,
\[
  \mathbb{E}[D(i)\mid X_\varnothing=x,\,D_\varnothing=k]=1+\rho(x),
  \qquad
  \mathrm{Var}(D(i)\mid X_\varnothing=x,\,D_\varnothing=k)=\sigma^2(x).
\]
\end{proposition}

The quantity $\rho(x)$ is the expected offspring count of a randomly chosen neighbor of a type-$x$ root. A neighbor sampled from $Q_x$ has type $y$ and therefore Poisson expected degree $1+\lambda(y)$, giving $\mathbb{E}[D(i)\mid x,k]=1+\rho(x)$. The variance $\sigma^2(x)=\rho(x)+\mathrm{Var}_{Q_x}(\lambda(Y))$ decomposes into two sources, i.e. the Poisson noise within a given neighbor type and the between-type spread in expected degrees. Both quantities depend on the root type $x$, but not directly on the realized degree $k$. Conditioning on the number of neighbors does not change the distribution of a single neighbor's degree once the type is fixed.

Since the $k$ neighbor degrees are conditionally i.i.d. given $(X_\varnothing=x,D_\varnothing=k)$, the conditional law of large numbers follows directly.

\begin{theorem}[Conditional Law of Large Numbers (LLN) for IRG]
\label{thm:lln-irg}
For every $\varepsilon>0$ and $x\in S$ with $\sigma^2(x)<\infty$,
\[
  \mathbb{P}\!\left[\,\left|\bar{D}_k-(1+\rho(x))\right|>\varepsilon\mid X_\varnothing=x,\,D_\varnothing=k\right]
  \le \frac{\sigma^2(x)}{k\varepsilon^2}
  \longrightarrow 0
  \quad\text{as }k\to\infty.
\]
\end{theorem}

Theorem~\ref{thm:lln-irg} states that $\bar D_k$ concentrates around the constant $1+\rho(x)$ as the root degree k grows, at the Chebyshev rate $k^{-1}$. The fluctuations are $O(k^{-1/2})$, and a Gaussian approximation follows from the Lyapunov Central Limit Theorem (CLT) applied to the conditionally i.i.d. summands.

\begin{theorem}[Conditional CLT for IRG]
\label{thm:clt-irg-conditional}
Suppose the conditional $(2+\eta)$-th moment of $D(i)$ is uniformly bounded for some $\eta>0$. Then for every $x\in S$ and $t\in\mathbb{R}$,
\[
  \mathbb{P}\!\left[\sqrt{k}\,(\bar{D}_k-(1+\rho(x)))\le t\mid X_\varnothing=x,\,D_\varnothing=k\right]
  \;\longrightarrow\; \Phi\!\left(\frac{t}{\sigma(x)}\right)
  \quad\text{as }k\to\infty.
\]
\end{theorem}

Theorems~\ref{thm:lln-irg} and~\ref{thm:clt-irg-conditional} condition on the root type, which is latent in applications. To obtain a marginal CLT for $\bar D_k$, we first describe how observing $D_\varnothing=k$ updates the latent expected degree $\Lambda:=\lambda(X_\varnothing)$. This is a Bayesian conditioning step, i.e. a large observed degree is evidence for a large latent expected degree.

\begin{proposition}[Posterior of the root's expected degree]
\label{prop:posterior-type}
Let $\Lambda:=\lambda(X_\varnothing)$ and let $\nu=\mu\circ\lambda^{-1}$ be the push-forward of $\mu$ under $\lambda$. Suppose $\nu$ has a Lebesgue density $f_\nu$. Then the conditional density of $\Lambda$ given $D_\varnothing=k$ is
\[
  f_{\Lambda\mid D_\varnothing=k}(\ell)
  = \frac{e^{-\ell}\ell^k\,f_\nu(\ell)}
         {\displaystyle\int_0^\infty e^{-s}s^k\,f_\nu(s)\,\mathrm{d}s},
  \quad \ell>0.
\]
\end{proposition}

Proposition~\ref{prop:posterior-type} shows that conditioning on $D_\varnothing=k$ tilts the prior $f_\nu(\ell)$ by the Poisson likelihood $e^{-\ell}\ell^k$. As $k\to\infty$, this likelihood concentrates around $\ell=k$, so $f_{\Lambda\mid D_\varnothing=k}$ puts mass near $\ell\approx k$. This posterior concentration is the main input for the marginal results.

To connect the conditional-on-type CLT to the marginal setting, define
\[
  r(\ell):=\mathbb{E}[\rho(X_\varnothing)\mid\Lambda=\ell],
  \qquad
  \bar\rho_k:=\mathbb{E}[\rho(X_\varnothing)\mid D_\varnothing=k]
  =\int_0^\infty r(\ell)\,f_{\Lambda\mid D_\varnothing=k}(\ell)\,\mathrm{d}\ell,
\]
and decompose the centered mean neighbor degree as
\begin{equation}
  \bar{D}_k-(1+\bar\rho_k)
  = \underbrace{\bar{D}_k-(1+\rho(X_\varnothing))}_{\displaystyle=:\,\mathrm{Term~I}}
  + \underbrace{\rho(X_\varnothing)-\bar\rho_k}_{\displaystyle=:\,\mathrm{Term~II}}.
  \label{eq:decomp}
\end{equation}
Term~I is the neighbor-averaging noise after the root type has been fixed. Theorem~\ref{thm:clt-irg-conditional} shows this is $O_P(k^{-1/2})$. Term~II is the remaining posterior fluctuation of the root affinity $\rho(X_\varnothing)$ around $\bar\rho_k$. The relative size of these two terms depends on the rank and assortativity of the kernel.

\begin{theorem}[Marginal CLT for Term~I]
\label{thm:marginal-clt-I}
Let $\tilde\sigma^2(\ell):=\mathbb{E}[\sigma^2(X_\varnothing)\mid\Lambda=\ell]$ and assume, in addition to the uniform $(2+\eta)$-moment bound of Theorem~\ref{thm:clt-irg-conditional}:
\begin{enumerate}
  \item[\rm(V1)] \emph{Uniform ellipticity:} $\inf_{x\in S}\sigma^2(x)\ge\sigma_{\min}^2>0$;
  \item[\rm(V2)] \emph{Posterior variance stabilization:}
        $\sigma^2(X_\varnothing)\xrightarrow{\mathbb{P}(\,\cdot\mid D_\varnothing=k)}\tilde\sigma^2_\infty\in(0,\infty)$
        as $k\to\infty$.
\end{enumerate}
Then for every $t\in\mathbb{R}$,
\[
  \mathbb{P}\!\left[\sqrt{k}\,(\bar{D}_k-(1+\rho(X_\varnothing)))\le t\mid D_\varnothing=k\right]
  \;\longrightarrow\;\Phi\!\left(\frac{t}{\tilde\sigma_\infty}\right)
  \quad\text{as }k\to\infty.
\]
A primitive sufficient condition for \emph{(V2)} is $\tilde\sigma^2(\ell)\to\tilde\sigma^2_\infty$ together with $\mathrm{Var}\bigl(\sigma^2(X_\varnothing)\mid\Lambda=\ell\bigr)\to 0$ as $\ell\to\infty$, since the posterior of $\Lambda$ concentrates on $\{\Lambda\to\infty\}$ by Proposition~\ref{prop:posterior-type}. In the rank-one case $\sigma^2(x)\equiv\tilde\sigma^2_\infty$ is constant, so \emph{(V1)}-\emph{(V2)} hold automatically.
\end{theorem}

Theorem~\ref{thm:marginal-clt-I} controls only the neighbor-averaging component. In the \emph{non-rank-one} case, Term~II may dominate. Here the type need not be determined by $\Lambda=\lambda(X_\varnothing)$, so the affinity $\rho(X_\varnothing)$ carries fluctuation from two distinct sources. Write
\begin{equation}
  \rho(X_\varnothing):=r(\Lambda)
   +\zeta,
  \qquad \zeta:=\rho(X_\varnothing)-r(\Lambda)
  \label{eq:rho-split}
\end{equation}
where $\mathbb{E}[\zeta\mid\Lambda]=0$,
and set 
$$
    \omega^2(\ell):=\mathrm{Var}(\rho(X_\varnothing)\mid\Lambda=\ell)=\mathbb{E}[\zeta^2\mid\Lambda=\ell]
$$
for the within-level-set variance.
$r(\Lambda)-\bar\rho_k$ reflects posterior uncertainty about $\Lambda$ and, when $r$ varies, fluctuates at scale $|r'(\ell_k^*)|\sqrt{v_k}\asymp|r'(\ell_k^*)|\sqrt k$. $\zeta$ is a single, non-averaged draw of the root's residual type at scale $\omega(\ell_k^*)$, while it is not an average and need not be asymptotically Gaussian.

\begin{theorem}[Full CLT in the non-rank-one case]
\label{thm:full-clt}
Assume the conditions of Theorem~\ref{thm:marginal-clt-I}. Define $h_k(\ell):=k\log\ell-\ell+\log f_\nu(\ell)$, let $\ell_k^*$ be the maximizer of $f_{\Lambda\mid D_\varnothing=k}$ (solution to $h_k'(\ell)=0$), set $v_k:=-1/h_k''(\ell_k^*)$, and let $r$, $\zeta$, $\omega^2$ be as above.
Under conditions:
\begin{enumerate}
  \item[\rm(R1)] \emph{Laplace approximation:}
        $(\Lambda-\ell_k^*)/\sqrt{v_k}\xrightarrow{d}N(0,1)$
        under $\mathbb{P}(\,\cdot\mid D_\varnothing=k)$;
  \item[\rm(R2)] \emph{Non-negligible affinity gradient:}
        $r$ is eventually $C^2$ with
        \[
          k\,|r'(\ell_k^*)|\to\infty,
          \qquad
          \sqrt{k}\,|r''(\ell_k^*)|=o\bigl(|r'(\ell_k^*)|\bigr),
        \]
        and $\sup_{|\ell-\ell_k^*|\le\delta\sqrt{v_k}}|r''(\ell)|=O(|r''(\ell_k^*)|)$ for some $\delta>0$;
  \item[\rm(R3)] \emph{Within-level-set negligibility:}
        $\omega(\ell_k^*)=o\bigl(|r'(\ell_k^*)|\sqrt{v_k}\bigr)$.
\end{enumerate}
Then Term~II in Eq.\eqref{eq:decomp} dominates Term~I, and for every $t\in\mathbb{R}$,
\[
  \mathbb{P}\!\left[\frac{\bar D_k-(1+\bar\rho_k)}{|r'(\ell_k^*)|\sqrt{v_k}}\le t\mid D_\varnothing=k\right]
  \;\longrightarrow\;\Phi(t)
  \quad\text{as }k\to\infty.
\]
\end{theorem}

\begin{corollary}[One-dimensional type]
\label{cor:full-clt-1d}
If the latent type is one-dimensional with $\lambda$ strictly increasing, then $\Lambda=\lambda(X_\varnothing)$ determines $X_\varnothing$ up to null sets, and hence $\rho(X_\varnothing)=r(\Lambda)$ for $r=\rho\circ\lambda^{-1}$. Then $\zeta\equiv0$, $\omega\equiv0$, and \emph{(R3)} holds automatically. Theorem~\ref{thm:full-clt} then holds under \emph{(R1)} and \emph{(R2)} alone.
\end{corollary}

The conditions isolate three sources of behavior. Condition (R1) controls posterior concentration of $\Lambda$ around $\ell_k^*\sim k$ at scale $\sqrt{v_k}\asymp\sqrt{k}$. Condition (R2) keeps the affinity variation above the $k^{-1/2}$ neighbor-sampling noise while controlling the Taylor remainder. Condition (R3) makes residual type variation negligible after conditioning on $\Lambda$. It is automatic for one-dimensional type and holds more generally when $\omega(\ell)$ decays sufficiently fast. Proposition~\ref{prop:assortativity} supplies an interpretable assortativity condition for (R2). Consequently, neighbor averaging governs the rank-one scale, while posterior type fluctuation dominates in the assortative non-rank-one regime.

\begin{remark}[When (R3) fails]
\label{rem:r3-fails}
Condition (R3) can genuinely fail. When it fails, the normal approximation above is no longer the natural calibration. The full standardizing scale is the conditional standard deviation $s_k:=\mathrm{sd}(\bar D_k\mid D_\varnothing=k)$, whose square decomposes into three orthogonal pieces,
\[
  s_k^2=\underbrace{\mathbb{E}[\sigma^2(X_\varnothing)\mid D_\varnothing=k]/k}_{\text{Term I}}
       +\underbrace{\mathrm{Var}(r(\Lambda)\mid D_\varnothing=k)}_{\text{between-level-set}}
       +\underbrace{\mathbb{E}[\omega^2(\Lambda)\mid D_\varnothing=k]}_{\text{within-level-set}}
       +o(\cdot),
\]
of orders $k^{-1}$, $r'(\ell_k^*)^2 v_k$, and $\omega^2(\ell_k^*)$ respectively. When the within-level-set order is comparable to or larger than the other two, the dominant fluctuation is the single draw $\zeta$, whose limit law is that of $\rho(X_\varnothing)-r(\Lambda)$ at $\Lambda\approx\ell_k^*$ and is generally \emph{non-Gaussian}. For example, in a finite-type block kernel with community-dependent affinity, $\omega^2(\ell)\asymp \mathbb{P}(\text{minority community}\mid\Lambda=\ell)$, which decays only polynomially when the community weight tails differ; if that decay is slower than $k^{-1/2}$ the standardized statistic is asymptotically non-Gaussian. In such cases, the analytic normal calibration should be replaced by a resampling calibration of $s_k$, as in the class-block analysis of Section~\ref{subsec:highschool_examples}.
\end{remark}

We first verify (R1) from primitive regularity conditions on the latent-degree density $f_\nu$.

\begin{proposition}[Primitive sufficient conditions for (R1)]
\label{prop:r1-primitive}
Suppose $f_\nu\in C^3(0,\infty)$ with $f_\nu(\ell)>0$ for all $\ell>0$, and that
\begin{enumerate}
  \item[\rm(L1)] $Z_k:=\int_0^\infty \ell^{\,k}e^{-\ell}f_\nu(\ell)\,\mathrm{d}\ell<\infty$ for each $k\ge 1$;
  \item[\rm(L2)] $\ell^2\,|(\log f_\nu)''(\ell)|=O(1)$ as $\ell\to\infty$;
  \item[\rm(L3)] $\ell^3\,|(\log f_\nu)'''(\ell)|=O(1)$ as $\ell\to\infty$.
\end{enumerate}
Then condition \emph{(R1)} of Theorem~\ref{thm:full-clt} holds, i.e.\ $(\Lambda-\ell_k^*)/\sqrt{v_k}\xrightarrow{d}N(0,1)$ under $\mathbb{P}(\,\cdot\mid D_\varnothing=k)$ as $k\to\infty$. The proof is given in Section~\ref{subsec:proof-irg}.
\end{proposition}

\noindent
Conditions (L2)-(L3) require only that the log-prior's curvature and third derivative decay at least as fast as $\ell^{-2}$ and $\ell^{-3}$. This ensures that the Poisson likelihood dominates the prior at scale $\sqrt{k}$. Two canonical cases are:
\begin{itemize}
  \item \emph{Pareto:} $f_\nu(\ell)\propto\ell^{-(a_0+1)}$, $\ell>w_{\min}$.
        Then $(\log f_\nu)''=(a_0+1)/\ell^2$, $(\log f_\nu)'''=-2(a_0+1)/\ell^3$,
        both satisfying (L2)--(L3). The posterior $\Lambda\mid D_\varnothing=k$
        is exactly $\mathrm{Gamma}(k-a_0,1)$ on $(w_{\min},\infty)$, so
        (R1) follows directly from the Gamma CLT.
  \item \emph{Exponential:} $f_\nu(\ell)\propto e^{-\mu\ell}$.
        Then $(\log f_\nu)''=0$, (L2)--(L3) hold trivially,
        and the posterior is $\mathrm{Gamma}(k+1,1+\mu)$, again satisfying (R1) exactly.
\end{itemize}

We now turn to condition (R2), whose statistical content is degree assortativity \citep{newman2002assortative,newman2003mixing}. The function $r(\ell)=\mathbb{E}[\rho(X_\varnothing)\mid\Lambda=\ell]$ is the expected affinity of a neighbor of a root with latent degree $\ell$, so the growth of $r$ measures how strongly high-degree vertices attach to other high-degree vertices. We make this quantitative.

\begin{definition}[Assortativity index]
\label{def:assortativity}
A one-dimensional IRG kernel is \emph{assortative of index $\vartheta\in[0,1)$} if $r$ is eventually $C^2$ and regularly varying with
\[
  r'(\ell)=\ell^{-\vartheta}L(\ell),
  \qquad
  r''(\ell)=O\bigl(\ell^{-\vartheta-1}L(\ell)\bigr)
  \quad\text{as }\ell\to\infty,
\]
for some slowly varying $L>0$. Equivalently, along the tail,
\[
  \rho(x)=r(\lambda(x))
  \asymp \lambda(x)^{1-\vartheta}L(\lambda(x)).
\]
\end{definition}

The index $\vartheta$ interpolates between strong and weak degree assortativity. At $\vartheta=0$ the affinity grows linearly, $\rho(x)\asymp\lambda(x)$, the maximal-assortativity case in which a hub's neighbors are themselves hubs; this recovers the simple condition $\lim_k|r'(\ell_k^*)|=L_\infty>0$. As $\vartheta\uparrow 1$ the gradient flattens and Term~II loses its dominance over Term~I, so the boundary $\vartheta=1$ marks the transition to the rank-one (non-assortative) regime of Theorem~\ref{thm:clt-rank-one}, where $r$ is constant and Term~II vanishes.

\begin{proposition}[Assortativity implies (R2)]
\label{prop:assortativity}
If a one-dimensional IRG kernel is assortative of index $\vartheta\in[0,1)$ in the sense of Definition~\ref{def:assortativity}, and (R1) holds with $\ell_k^*\sim ck$ and $v_k\sim c^2k$ as in Proposition~\ref{prop:r1-primitive}, then condition (R2) of Theorem~\ref{thm:full-clt} is satisfied. Since the type is one-dimensional, (R3) holds vacuously (Corollary~\ref{cor:full-clt-1d}), so the non-rank-one CLT applies:
$$
    \bigl(\bar D_k-(1+\bar\rho_k)\bigr)/\bigl(|r'(\ell_k^*)|\sqrt{v_k}\bigr)\xrightarrow{d}N(0,1).
$$
The proof is given in Section~\ref{subsec:proof-irg}.
\end{proposition}

Degree-assortative kernels can realize these profiles. For example, $r(\ell)=\rho_0+\beta\ell^{1-\vartheta}$ has index $\vartheta\in[0,1)$, and $\vartheta=0$ gives proportional growth of neighbor and root expected degrees. By contrast, finite-rank and degree-corrected block kernels with bounded $\rho$ fall outside (R2) when their affinity gradients vanish too quickly. In that regime, Theorem~\ref{thm:clt-rank-one} supplies the relevant neighbor-averaging scale.

The rank-one case is the main calibration used in our tests, and it includes the Chung-Lu model.

\begin{theorem}[CLT in the rank-one case]
	\label{thm:clt-rank-one}
	Suppose the symmetric kernel is rank-one: $\kappa(x,y)=\varphi(x)\varphi(y)/\int\varphi\,\mathrm{d}\mu$ for measurable $\varphi:S\to(0,\infty)$. Then $\rho(x)\equiv\rho$ is constant across all types, Term~II in Eq.\eqref{eq:decomp} is identically zero, and for every $t\in\mathbb{R}$,
	\[
	\mathbb{P}\!\left[\sqrt{k}\,(\bar{D}_k-(1+\rho))\le t\mid D_\varnothing=k\right]
	\;\longrightarrow\;\Phi\!\left(\frac{t}{\tilde\sigma_\infty}\right)
	\quad\text{as }k\to\infty.
	\]
	For the Chung-Lu kernel $\kappa(x,y)=w(x)w(y)/\mathbb{E}[W]$,
	\[
	\rho=\frac{\mathbb{E}[W^2]}{\mathbb{E}[W]},
	\qquad
	\tilde\sigma^2_\infty
	=\frac{\mathbb{E}[W^2]}{\mathbb{E}[W]}+\frac{\mathbb{E}[W^3]}{\mathbb{E}[W]}
	-\left(\frac{\mathbb{E}[W^2]}{\mathbb{E}[W]}\right)^{\!2},
	\]
	provided $\mathbb{E}[W^3]<\infty$.
\end{theorem}

Theorem~\ref{thm:clt-rank-one} is the calibration result behind the goodness-of-fit tests in Section~\ref{sec:methods}. Under the Chung-Lu null, $\bar D_k$ is asymptotically $N(1+\rho,\,\tilde\sigma_\infty^2/k)$, with mean and variance independent of $k$ beyond the $1/\sqrt{k}$ scale factor. The mean $\rho=\mathbb{E}[W^2]/\mathbb{E}[W]$ is the size-biased mean weight. A randomly chosen neighbor is sampled proportional to its weight, so its expected weight is $\mathbb{E}[W^2]/\mathbb{E}[W]$. Both $\rho$ and $\tilde\sigma^2_\infty$ can be estimated consistently from the empirical degree sequence, as detailed in Section~\ref{sec:gof-notation}. All proofs for this subsection are collected in Section~\ref{subsec:proof-irg}.

\subsection{Logarithmic growth of the neighbor-degree sum in PA}
\label{subsec:pa-results}

Previous work derives nearest-neighbor degree correlations for growing and generalized PA models \citep{barrat2005rate,krot2017assortativity}. We instead use the P\'{o}lya point tree to obtain the conditional high-degree limit and its coefficient, which are needed for calibration. Set $\tau=3+\delta/m$ and condition on $D_\varnothing=k>m$. The root then has two classes of neighbors:
\begin{itemize}
	\item \emph{Old neighbors} (labels $1,\ldots,m$): the $m$ vertices to which the root attached at the moment it entered the network. They arrived before the root and have had additional time to accumulate their own connections.
	\item \emph{Young neighbors} (labels $m+1,\ldots,k$): the $k-m$ vertices that attached to the root after it was born, selecting it via preferential attachment. They are younger than the root.
\end{itemize}
Encode the root's age by $A_\varnothing\sim\mathrm{Uniform}(0,1)$ and set $U_\varnothing :=A_\varnothing^{1/(\tau-1)}$, where smaller ages indicate earlier arrival. The P\'olya point tree gives the degrees of old neighbor $i\in\{1,\ldots,m\}$ and young neighbor $j\in\{1,\ldots,k-m\}$ as follows \cite[Chap.~5]{van2024random}.
\begin{align}
    D(i) &= 1+m+\mathrm{Poisson}\!\left(\Gamma^O_i(B_i^{-1}-1)\right), \label{eq:deg-O}\\
    D(m+j) &= m+\mathrm{Poisson}\!\left(\Gamma^Y_j(C_j^{-1}-1)\right), \label{eq:deg-Y}
\end{align}
where $\Gamma_i^O\sim\mathrm{Gamma}(m+\delta+1,1)$, $\Gamma_j^Y\sim\mathrm{Gamma}(m+\delta,1)$, $B_i=U_{\varnothing i}^{1/(\tau-2)}U_\varnothing$ with $U_{\varnothing i}\sim\mathrm{Uniform}(0,1)$ independent of $U_\varnothing$, and $C_j\mid(D_\varnothing=k,U_\varnothing=u)\sim\mathrm{Uniform}(u,1)$, with $C_1,\ldots,C_{k-m}$ conditionally independent given $(D_\varnothing=k,U_\varnothing=u)$:
\begin{equation}\label{eq:Cj-dist}
    P(C_j\le c\mid D_\varnothing=k,U_\varnothing=u)=\frac{c-u}{1-u},\quad c\in[u,1].
\end{equation}
The Poisson terms in Eq.\eqref{eq:deg-O}--\eqref{eq:deg-Y} count connections acquired after each neighbor's birth. Earlier neighbors have more time to acquire such connections.

The central object in the PA calculation is the posterior distribution of the root's age given its degree.

\begin{lemma}[Posterior age of the root]
\label{lem:posterior-U}
For $k>m$, the transformed age $U_\varnothing$ satisfies, under $\mathbb{P}(\,\cdot\mid D_\varnothing=k)$,
\[
  U_\varnothing\mid(D_\varnothing=k)\;\sim\;\mathrm{Beta}(\alpha,\beta_k),
\]
where $\alpha:=m+2+\delta+\delta/m$ and $\beta_k:=k-m+1$.
\end{lemma}

The posterior mean $\mathbb{E}[U_\varnothing\mid D_\varnothing=k]=\alpha/(\alpha+\beta_k)$ tends to zero. Hence a high-degree root is asymptotically old, and $-\log U_\varnothing\sim\log k$ in probability. This age effect produces the $k\log k$ neighbor-degree sum.

The old and young neighbors contribute on different scales. Proposition~\ref{prop:EO} shows that old neighbors contribute only $O(k)$ in total, one order smaller than the dominant $k\log k$ term.

\begin{proposition}[Expected degree from old neighbors]
\label{prop:EO}
Assume $\delta>0$. For $k>m$,
\[
  \mathbb{E}\!\left[\sum_{i=1}^{m}D(i)\mid D_\varnothing=k\right]
  = -m\delta+m(m+\delta+1)\cdot\frac{m+\delta}{\delta}
          \cdot\frac{k+\delta+\delta/m+2}{m+\delta+\delta/m+1}.
\]
In particular, $\mathbb{E}[\sum_{i=1}^m D(i)\mid D_\varnothing=k]=O(k)$ as $k\to\infty$.
\end{proposition}

Thus old neighbors contribute at most a bounded amount per neighbor to $\bar D_k$, similar to the IRG case. The dominant term comes from the young neighbors, whose expected degrees involve the digamma function through the inverse-age integral. Specifically, each young neighbor $j$ has conditional mean degree
\[
  \mathbb{E}[D(m+j)\mid D_\varnothing=k,U_\varnothing=u]
  = m+(m+\delta)\!\left(\frac{-\log u}{1-u}-1\right),
\]
and integrating over the Beta posterior of $U_\varnothing$ produces a digamma expression.

\begin{proposition}[Expected degree from young neighbors]
\label{prop:EY}
For $k>m$,
\begin{align}
  \mathbb{E}\!\left[\sum_{j=1}^{k-m}D(m+j)\mid D_\varnothing=k\right]
  &=(m+\delta)(k-m) \notag\\
  &\quad\times
    \left\{\frac{k+\delta+\frac{\delta}{m}+2}{k-m}
    \Bigl[\psi\!\left(k+\delta+\tfrac{\delta}{m}+2\right)
         -\psi\!\left(m+\delta+\tfrac{\delta}{m}+2\right)\Bigr]
    -1\right\},
\end{align}
where $\psi=\Gamma'/\Gamma$ is the digamma function.
\end{proposition}

The digamma difference satisfies $\psi(k+c_1)-\psi(c_2)=\log k+O(1)$ for fixed constants $c_1,c_2$ (Lemma~\ref{lem:digamma-asymp} in Section~\ref{sec:proofs}), so Proposition~\ref{prop:EY} gives
\[
  \mathbb{E}\!\left[\sum_{j=1}^{k-m}D(m+j)\mid D_\varnothing=k\right]\sim(m+\delta)\,k\log k.
\]
Thus young neighbors contribute order $k\log k$, while Proposition~\ref{prop:EO} bounds the old-neighbor contribution by order $k$. Their sum yields the following law of large numbers.

\begin{theorem}[LLN for neighbor degrees under PA]
\label{thm:lln-pa}
Under the P\'{o}lya point tree with $\delta>0$, for any $\varepsilon>0$,
\[
  \mathbb{P}\!\left[\left|\,\frac{1}{k\log k}\sum_{i=1}^{k}D(i)-(m+\delta)\,\right|>\varepsilon\mid D_\varnothing=k\right]
  \;\longrightarrow\; 0
  \quad\text{as }k\to\infty.
\]
\end{theorem}

Equivalently, $\bar D_k/\log k\xrightarrow{\mathbb P}m+\delta$. The coefficient combines the $m$ initial connections with the attractiveness offset $\delta$. Section~\ref{subsec:proof-pa} contains the proof and the required digamma calculations.

\subsection{The discriminating signature}
\label{subsec:contrast}

The preceding limits separate the two mechanisms at a single high-degree root.

\begin{remark}[Key contrast]
\label{rem:contrast}
Let $k=D_\varnothing\to\infty$.
\begin{itemize}
  \item \textbf{Under a rank-one IRG (Chung--Lu) null:}
        $\bar{D}_k \xrightarrow{\mathbb{P}} 1+\rho$, a finite constant that does not depend on $k$. The fluctuations around $1+\rho$ are of order $k^{-1/2}$ (Theorem~\ref{thm:clt-rank-one}).
  \item \textbf{Under PA with $\delta>0$:} $\bar{D}_k/\log k \xrightarrow{\mathbb{P}} m+\delta$, so
        $\bar{D}_k$ grows without bound at a logarithmic rate (Theorem~\ref{thm:lln-pa}).
\end{itemize}
In particular, for any fixed $C<\infty$, $\mathbb{P}(\bar{D}_k>C\mid D_\varnothing=k)\to 1$ under PA with $\delta>0$, whereas $\mathbb{P}(\bar{D}_k>C\mid D_\varnothing=k)\to 0$ under IRG for $C>1+\rho$.
The two mechanisms are therefore \emph{asymptotically distinguishable} from the conditional distribution of $\bar{D}_k$ alone and no additional network-level information is needed.
\end{remark}

This contrast yields a level check under a specified IRG null and a slope check for the positive $\log k$ trend predicted by PA.

\section{Testing Procedures}
\label{sec:methods}

We separate null adequacy from direction. Methods~1--3 test whether high-degree neighborhoods match the level and scale of a specified sparse-graph null. Method~4 then tests the residual slope after removing that null prediction. Positive, negative, and flat slopes correspond respectively to the PA direction, disassortative hub neighborhoods, and no remaining trend at this statistic.

Method~1 targets the maximum-degree class. Methods~2 and~3 aggregate the upper degree tail, with Bonferroni favoring robust familywise control and Simes favoring sensitivity to repeated moderate departures. Method~4 answers a different question and should be interpreted jointly with the level tests.

For a vertex $v$ with $D(v)=k>0$, recall that $\bar D_k^{(v)}=k^{-1}\sum_{u\in\mathcal N(v)}D(u)$.
For each distinct degree value $k$ present in the graph, define
\[
  \mathcal V(k)=\{v\in V:D(v)=k\},\quad
  n(k)=|\mathcal V(k)|,\quad
  T_k=\frac{1}{n(k)}\sum_{v\in\mathcal V(k)}\bar D_k^{(v)},
\]
the average of $\bar D_k^{(v)}$ across all degree-$k$ vertices. Let $\mathcal K$ denote the set of distinct positive degrees. The \emph{level-test high-degree set} is
\[
  \mathcal K^+_{\mathrm{level}}=\bigl\{k\in\mathcal K:k\geq k_{\mathrm{level}}(n)\bigr\},
  \qquad M_{\mathrm{level}}=|\mathcal K^+_{\mathrm{level}}|,
\]
where, unless stated otherwise, $k_{\mathrm{level}}(n)$ is the
$\lceil 0.05\,|\mathcal K|\rceil$-th largest distinct degree, so
$\mathcal K^+_{\mathrm{level}}$ contains the top $5\%$ of distinct degree
values.  A vertex-percentile window is a different rule: it chooses
$k_{\mathrm{level}}(n)$ from the empirical distribution of the positive
vertex degrees and then retains every distinct class above that threshold.
We use that rule only for the Reddit application, where it matches the
original analysis, and identify it explicitly there.

\subsection{Null calibration for degree-class tests}
\label{sec:gof-notation}

Methods~1--3 test whether the class statistics $T_k$ follow a specified IRG null. Let $m_0(k)$ and $s_0(k)$ denote the null conditional mean and standard deviation, so that
$$
    \frac{(T_k-m_0(k))}{s_0(k)/\sqrt{n(k)}}\xrightarrow{d}N(0,1)
$$
under $H_0$.

For the rank-one Chung-Lu null, Theorem~\ref{thm:clt-rank-one} gives $m_0(k)=1+\rho$ and $s_0(k)=\tilde\sigma_\infty/\sqrt{k}$, where $\rho=\mathbb{E}[W^2]/\mathbb{E}[W]$. Thus
\[
  Z_k=\frac{T_k-(1+\rho)}{\tilde\sigma_\infty/\sqrt{k\,n(k)}}
\]
is asymptotically $N(0,1)$ with a centering and variance factor shared across degree classes. This rank-one invariance converts the composite null into a family of asymptotically pivotal statistics with only two scalar nuisance parameters. The level tests below differ only in how they aggregate this pivot family, trading off sensitivity to a single dominant class against evidence spread across the upper tail.

Because the latent weights $W$ are unobserved, the calibration uses degree-based moment estimates. In the finite simple Chung--Lu graph, conditional on the full weight sequence, $D_i$ is Poisson-binomial with mean $\lambda_i^{(n)}=\sum_{j\ne i}\min\{W_iW_j/L_n,1\}$, where $L_n=\sum_jW_j$. Under the sparse Chung--Lu condition $\max_i W_i^2/L_n=o_P(1)$, truncation at one and the exclusion of self-loops are asymptotically negligible, $\lambda_i^{(n)}=W_i\{1+o_P(1)\}$, and the conditional degree distribution is asymptotically $\mathrm{Poisson}(W_i)$. The resulting asymptotic Poisson moment relations give $\mathbb{E}[D^2]/\mathbb{E}[D]=\mathbb{E}[W^2]/\mathbb{E}[W]+1$ and $\mathbb{E}[D^3]/\mathbb{E}[D]=\mathbb{E}[W^3]/\mathbb{E}[W]+3\mathbb{E}[W^2]/\mathbb{E}[W]+1$.
Substituting into $\mu_0=1+\rho=1+\mathbb{E}[W^2]/\mathbb{E}[W]$ and $\tilde\sigma_\infty^2=\rho+\mathbb{E}[W^3]/\mathbb{E}[W]-\rho^2$ yields the corrected degree-based plug-ins:
\[
  \hat\mu_0
  = \frac{\sum_{i}D(i)^2}{\sum_{i}D(i)},
  \qquad
  \hat{\tilde\sigma}_\infty^2
  = \frac{\sum_{i}D(i)^3}{\sum_{i}D(i)}
  - \left(\frac{\sum_{i}D(i)^2}{\sum_{i}D(i)}\right)^{\!2},
\]
giving plug-in estimates $\hat m_0(k)=\hat\mu_0$ and $\hat s_0(k)=\hat{\tilde\sigma}_\infty/\sqrt{k}$. The naive formula $1+\sum D_i^2/\sum D_i$ overestimates $\mu_0$ by one unit and inflates $\hat{\tilde\sigma}_\infty^2$ by $\hat\mu_0-1$. Both errors cause the level tests to reject too often under the null.
For a non-rank-one IRG null, Theorem~\ref{thm:full-clt} gives the centering $m_0(k)=1+\bar\rho_k$ and scale $s_0(k)=|r'(\ell_k^*)|\sqrt{v_k}$, estimated by plug-in as $\hat m_0(k)=1+\hat{\bar\rho}_k$ and $\hat s_0(k)=|\hat r'(\hat\ell_k^*)|\sqrt{\hat v_k}$.

\subsection{Method 1: maximum-degree hub test}
\label{sec:method1}

Let $k^*=\max_{v\in V}D(v)$ be the maximum degree, and let $\mathcal V(k^*)$ and $n(k^*)$ be the set and count of vertices attaining it. Method~1 uses only the maximum-degree class, averaging the neighbor-degree statistic over all maximum-degree vertices:
\[
  T_1 = \frac{1}{n(k^*)}\sum_{v\in\mathcal V(k^*)}\bar D_{k^*}^{(v)}.
\]
The standardized statistic and two-sided $p$-value are
\[
  Z_1 = \frac{T_1-\hat m_0(k^*)}{\hat s_0(k^*)/\sqrt{n(k^*)}}
  \;\xrightarrow{d}\; N(0,1),
  \qquad
  p_1=2\{1-\Phi(|Z_1|)\}.
\]
The test rejects at level $\alpha$ when $|Z_1|>z_{1-\alpha/2}$. Method~1 is simple and directly interpretable, but it uses only the single most extreme degree class. It can therefore have limited power when departures from the null are spread across several high-degree classes.

\subsection{Method 2: Bonferroni test over high-degree classes}
\label{sec:method2}

For each $k\in\mathcal K^+_{\mathrm{level}}$, the standardized class statistic and two-sided $p$-value are
\[
  Z_k=\frac{T_k-\hat m_0(k)}{\hat s_0(k)/\sqrt{n(k)}},
  \qquad
  p_k=2\{1-\Phi(|Z_k|)\}.
\]
Method~2 rejects the global null if at least one class is significant after Bonferroni correction:
\[
  p_2=\min\!\left\{M_{\mathrm{level}}\min_{k\in\mathcal{K}^+_{\mathrm{level}}}p_k,\,1\right\},
\]
rejecting when $p_2\leq\alpha$. The Bonferroni correction controls the familywise Type~I error without requiring independence among the class statistics. Its cost is conservativeness when many classes are tested and the evidence is distributed rather than concentrated.

\subsection{Method 3: Simes global test over high-degree classes}
\label{sec:method3}

Method~3 uses the same class $p$-values $\{p_k:k\in\mathcal K^+_{\mathrm{level}}\}$ as Method~2, but combines them using the Simes global test \citep{simes1986improved}. Let $p_{(1)}\leq\cdots\leq p_{(M_{\mathrm{level}})}$ be the ordered $p$-values. The Simes $p$-value is
\[
  p_3=\min_{1\leq r\leq M_{\mathrm{level}}}\!\left\{\frac{M_{\mathrm{level}}}{r}\,p_{(r)}, 1\right\},
\]
rejecting when $p_3\leq\alpha$. The Simes test controls the global Type~I error exactly under independence of the class $p$-values \citep{simes1986improved}, and also under positive-dependence conditions such as PRDS \citep{sarkar1998probability,benjamini2001control}. Independence, however, does not hold here: the class statistics $\{Z_k\}_{k\in\mathcal K^+_{\mathrm{level}}}$ are all functionals of the same edge set, and hubs in distinct classes may be adjacent or share neighbors. We therefore do not use Simes as the sole analytic size guarantee. Method~2 supplies the dependence-robust Bonferroni benchmark, while Method~3 is reported as a less conservative pooling companion whose calibration is checked in the rank-one null simulation and, for the real-data cases where the analytic assumptions are doubtful, by the same fitted-null bootstrap used for the reported $p$-values.
Compared with Method~2, the Simes procedure is more sensitive when multiple high-degree classes each show moderate evidence against the null, making it useful for systematic departures spread across the upper tail.

\subsection{Method 4: log-degree slope test}
\label{sec:method4}

Methods~1--3 test whether high-degree neighbor means agree with the specified null level. Method~4 uses the same degree-class summaries for the question of whether there is a trend in $\log k$ after subtracting the null centering. Under PA with $\delta>0$, the local limit gives $\bar D_k\propto (m+\delta)\log k$, so the residual trend has positive slope.

Let $\mathcal K^+_{\mathrm{slope}} =\{k\in\mathcal K:k\geq k_{\mathrm{slope}}(n)\}$ be the slope-test window, chosen to contain enough distinct degrees to estimate a stable line while still focusing on the upper tail. 
For $k\in\mathcal K^+_{\mathrm{slope}}$, define
\[
  R_k=T_k-\hat m_0(k),\qquad x_k=\log k,\qquad
  w_k=\frac{n(k)}{\hat s_0(k)^2}.
\]
For the Chung-Lu null, $w_k=k\,n(k)/\hat{\tilde\sigma}_\infty^2$. Let
\[
  \bar x_w=\frac{\sum_{k\in\mathcal K^+_{\mathrm{slope}}}w_kx_k}
                 {\sum_{k\in\mathcal K^+_{\mathrm{slope}}}w_k},
  \qquad
  S_{xx}=\sum_{k\in\mathcal K^+_{\mathrm{slope}}}w_k(x_k-\bar x_w)^2.
\]
The weighted slope and standardized statistic are
\[
  \hat\beta
  =\frac{\sum_{k\in\mathcal K^+_{\mathrm{slope}}}
          w_k(x_k-\bar x_w)R_k}{S_{xx}},
  \qquad
  Z_4=\hat\beta\sqrt{S_{xx}}.
\]
Large positive values indicate the cumulative-advantage direction. The one-sided test rejects when $Z_4>z_{1-\alpha}$, and a two-sided version can be used when the direction of departure is not specified.

The size-control central limit theorem of Theorem~\ref{thm:method4} is established under the rank-one Chung--Lu null. The procedure extends formally to a specified non-rank-one IRG null by replacing $\hat m_0(k)$ and $\hat s_0(k)$ with the Theorem~\ref{thm:full-clt} centering $1+\bar\rho_k$ and scale, so that the residualization removes $1+\bar\rho_k$; because the analytic $N(0,1)$ null is proven only in the rank-one case, we then calibrate by the fitted-null conditional bootstrap of Remark~\ref{rem:method4-bootstrap}. The slope test uses a wider high-degree window than Methods~1--3 because a slope estimate needs leverage in $\log k$.
The lower edge of the window trades this leverage against restriction to the hub regime; there is no universally optimal choice, so we fix it a priori for each application and report it together with the number of slope classes. The simulations and the arXiv analysis use the upper $50\%$ of distinct positive degrees, whereas the contact and reply networks use the upper $25\%$; because $\hat\beta$ is inverse-variance weighted, the window affects the estimate only through which classes enter, and where a conclusion is borderline it should be checked against the alternative window.

\begin{theorem}[Method~4: size control and consistency]
\label{thm:method4}
Let $\mathcal{K}^+_{\mathrm{slope}}=\{k_1<\cdots<k_M\}$ be the set of high-degree classes with $M\to\infty$ and $S_{xx}\to\infty$ as $n\to\infty$.
\begin{enumerate}
  \item[(i)] \textup{(Size control.)}
  Work under the Chung--Lu null, conditionally on the weight sequence
  $\mathbf W=(W_1,\dots,W_n)$, so that the edges $\{A_{ij}\}_{i<j}$ are independent,
  and condition throughout on the realized hub degrees $\{D(v):D(v)\in\mathcal K^+_{\mathrm{slope}}\}$
  as elsewhere in the paper. Assume $\mathbb{E}[W^3]<\infty$ and:
  \begin{enumerate}
    \item[\rm(S1)] \emph{Plug-in consistency and design stabilization:}
      $\hat\mu_0\xrightarrow{\mathbb{P}}\mu_0$,
      $\hat{\tilde\sigma}_\infty^2\xrightarrow{\mathbb{P}}\tilde\sigma_\infty^2\in(0,\infty)$,
      and the empirical threshold $k_{\mathrm{slope}}(n)$, the class counts $\{n(k)\}$,
      and the design quantities $(\bar x_w,S_{xx})$ agree with their oracle counterparts
      (obtained by replacing $\hat{\tilde\sigma}_\infty^2$ with $\tilde\sigma_\infty^2$)
      up to a $1+o_{\mathbb P}(1)$ factor;
    \item[\rm(S2)] \emph{Per-class Berry--Ess\'een and no domination:} for each
      $k\in\mathcal K^+_{\mathrm{slope}}$,
      $\operatorname{Var}_{\mathbf W}(\bar D_k^{(v)})=\tilde\sigma_\infty^2/k\,(1+o(1))$ and
      $\mathbb E_{\mathbf W}|T_k-\mu_0|^3=O\{(k\,n(k))^{-3/2}\}$, and the weights obey the
      Lyapunov no-domination condition
      $\max_{k\in\mathcal K^+_{\mathrm{slope}}}k\,n(k)(\log k)^2=o(S_{xx})$;
    \item[\rm(S3)] \emph{Sparsity, tail window, and log-degree spread:}
      $\varrho_n:=\max_i W_i^2/L_n\to0$; the maximum degree satisfies $k_{\max}^3=o(n)$,
      where $k_{\max}:=\max_vD(v)$; the slope window is a genuine upper tail,
      $k_{\mathrm{slope}}(n)\to\infty$; and the design is regular,
      $\sum_{k\in\mathcal K^+_{\mathrm{slope}}}k\,n(k)(\log k)^2=O(S_{xx})$ and
      $|\mathcal H|^2=o(nS_{xx})$ with $\mathcal H=\{v:D(v)\in\mathcal K^+_{\mathrm{slope}}\}$.
      When $\mathbb E[W^3]<\infty$ and the degree support is unbounded these hold
      automatically, since then $k_{\max}=o_{\mathbb P}(n^{1/3})$ and
      $\varrho_n=o_{\mathbb P}(n^{-1/3})$.
  \end{enumerate}
  Then, by Lemma~\ref{lem:crosscov}, $\operatorname{Var}_{\mathbf W}(Z_4)=1+o_{\mathbb P}(1)$ and
  \[
    Z_4 \;\xrightarrow{d}\; N(0,1),
  \]
  both conditionally on $\mathbf W$ and, by bounded convergence over $\mathbf W$,
  unconditionally. The one-sided test $\{Z_4>z_{1-\alpha}\}$ therefore has asymptotic
  level~$\alpha$.

  \item[(ii)] \textup{(Consistency.)}
  Under $\mathrm{PA}_n^{(m,\delta)}$ with $\delta>0$,
  \[
    \frac{Z_4}{\sqrt{S_{xx}}} \;\xrightarrow{\mathbb{P}}\; m+\delta \;>\; 0,
  \]
  so $Z_4\xrightarrow{\mathbb{P}}+\infty$ and the power $\mathbb{P}_{\mathrm{PA}}(Z_4>z_{1-\alpha})\to 1$.
\end{enumerate}
\end{theorem}

\begin{remark}[On the conditions, and what replaces the disjointness argument]
\label{rem:method4-conditions}
Condition (S2) supplies the per-class input: under $\mathbb E[W^3]<\infty$ the third-moment
bound holds and the no-domination condition is the Lyapunov condition for the
\emph{diagonal} part of the variance, exactly as in the earlier version of this remark.
The substantive change is (S3). The class statistics $\{T_k\}$ are \emph{not} independent
across $k$: they are functionals of the same edge set, and hubs in distinct classes may
be adjacent or share neighbors, so disjointness of the vertex sets
$\mathcal V(k)$ does not make the $T_k$ independent and does not by itself justify the
additivity of variances or the Lindeberg--Feller argument. Condition (S3) instead controls
the aggregate cross-class covariance directly: Lemma~\ref{lem:crosscov} shows that the
covariance sum vanishes because its leading, separable part is annihilated by the exact
weighted-centering identity $\sum_k w_k(\log k-\bar x_w)=0$; the only non-separable residual
is direct hub--hub adjacency, bounded term by term in Appendix~\ref{app:crosscov} and
controlled by the finite-third-moment assumption $\mathbb E[W^3]<\infty$. The sparsity
parameter $\varrho_n=\max_iW_i^2/L_n\to0$ is the
same condition already required for the degree-based calibration in
Section~\ref{sec:gof-notation} and used in the simulation design of
Section~\ref{subsec:cl_power}. When any of (S1)--(S3) is in doubt, the conditional bootstrap
calibration described after the proof provides a model-based fitted-null alternative to the
analytic $N(0,1)$ null.
\end{remark}

\section{Simulation Study}\label{sec:simulation}
The simulations address three finite-sample questions. We assess size and mean-shift sensitivity under rank-one and non-rank-one IRG nulls, sensitivity to controlled degree-distribution misspecification, and power against degree-matched PA alternatives. Methods~1--3 are evaluated as level tests, while Method~4 targets the positive residual slope predicted by PA. We use significance level $\alpha=0.05$ and, unless stated otherwise, $B=1000$ replications with $n=20000$. The exact non-rank-one Bernoulli experiment uses $n=5000$ because it requires a dense finite kernel matrix.

\subsection{Rank-one null calibration and mean-shift sensitivity}
\label{subsec:cl_power}
We generate a Chung-Lu random graph with $n = 20000$ nodes. Each vertex $i \in [n]$ is assigned an independent weight
\[
  w_i = 10\, U_i^{-1/3.15},
  \qquad U_i \sim \mathrm{Uniform}(0,1).
\]
Conditional on the weight sequence $\{w_i\}_{i=1}^n$, edges are generated independently with probability
\[
  p_{ij} = \min\!\left\{\frac{w_i w_j}{\sum_{\ell=1}^n w_\ell},\, 1\right\},
  \qquad 1 \leq i < j \leq n.
\]
Self-loops are not allowed. For the Pareto exponent $3.15>2$, $\max_i w_i^2/L_n=o_P(1)$, where $L_n=\sum_iw_i$. Hence the truncation in $p_{ij}$ and the exclusion of self-loops do not change the limiting Chung--Lu calibration, although they can produce a small finite-sample downward discrepancy between a weight and its expected degree.

For each of the $B = 1000$ replicated networks, we compute the three test statistics $T_1$, $\{T_k\}_{k \in \mathcal{K}^+_{\mathrm{level}}}$, and the corresponding $p$-values. To isolate sensitivity to a class-level mean shift, we evaluate the rejection rules on the same $B$ networks after substituting $\hat{\mu}_0 + \Delta$ for $\hat{\mu}_0$ in the standardization. For Method~1,
\[
  Z_1^{(\Delta)} = \frac{T_1 - (\hat{\mu}_0 + \Delta)}{\hat{\tilde\sigma}_\infty / \sqrt{k^* n(k^*)}},
\]
and analogously for Methods~2 and~3. The empirical rejection rate at shift $\Delta$ for method $i$ is
\[
  \hat{\pi}_i(\Delta)
  = \frac{1}{B} \sum_{b=1}^B \mathbf{1}\!\left\{ T_i^{(b)} \in R_i(\hat{\mu}_0 + \Delta) \right\},
  \quad i = 1, 2, 3,
\]
evaluated over $\Delta \in \{-10, -9, \ldots, 9, 10\}$ at nominal level $\alpha = 0.05$. By construction, $\hat{\pi}_i(0)$ estimates the empirical Type~I error of Method~$i$.

\begin{figure}[htbp]
  \centering
  \begin{subfigure}[b]{0.32\textwidth}
    \centering
    \includegraphics[width=\linewidth]{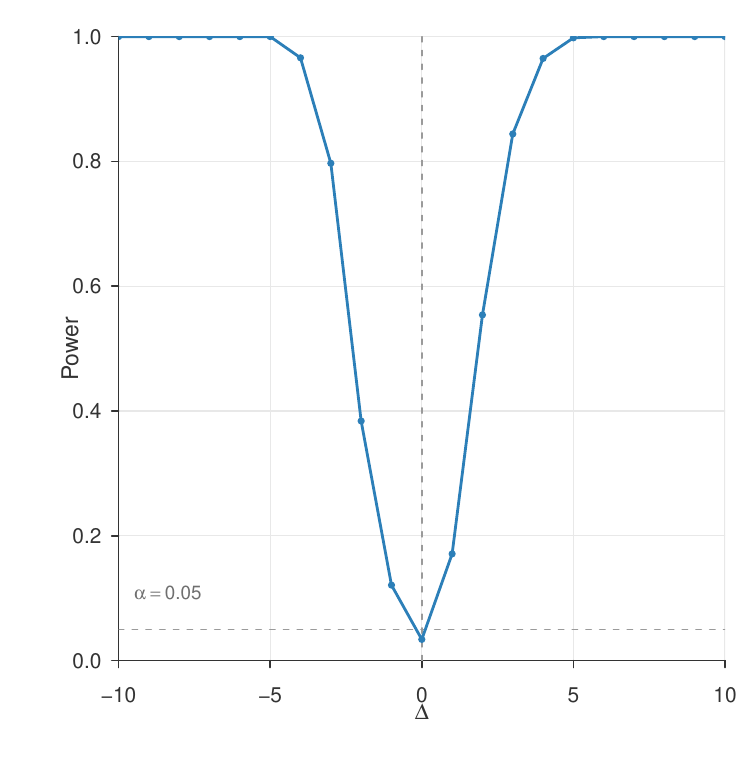}
    \caption{Method~1 (max-degree hub).}
    \label{fig:power_m1}
  \end{subfigure}
  \hfill
  \begin{subfigure}[b]{0.32\textwidth}
    \centering
    \includegraphics[width=\linewidth]{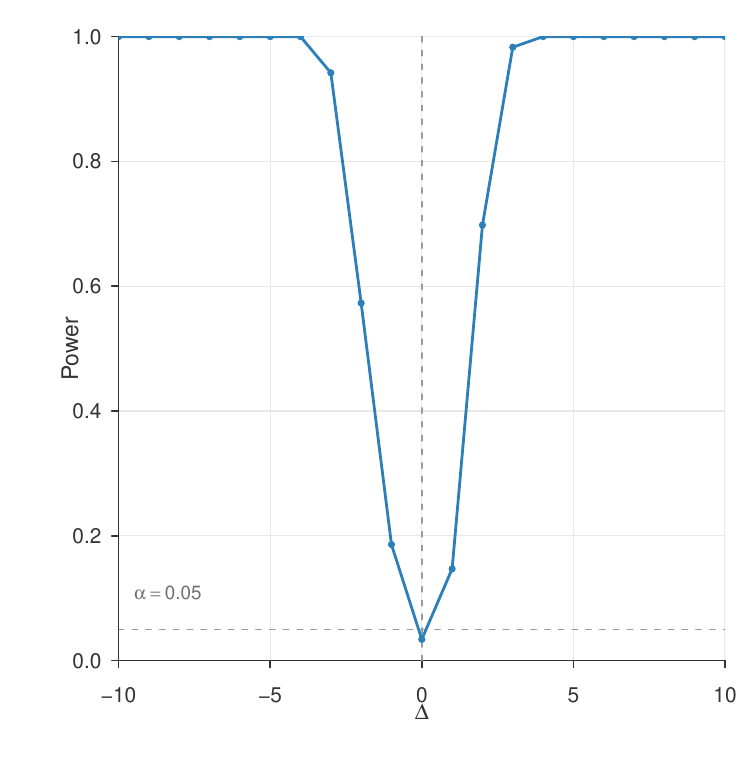}
    \caption{Method~2 (Bonferroni, top 5\%).}
    \label{fig:power_m2}
  \end{subfigure}
  \hfill
  \begin{subfigure}[b]{0.32\textwidth}
    \centering
    \includegraphics[width=\linewidth]{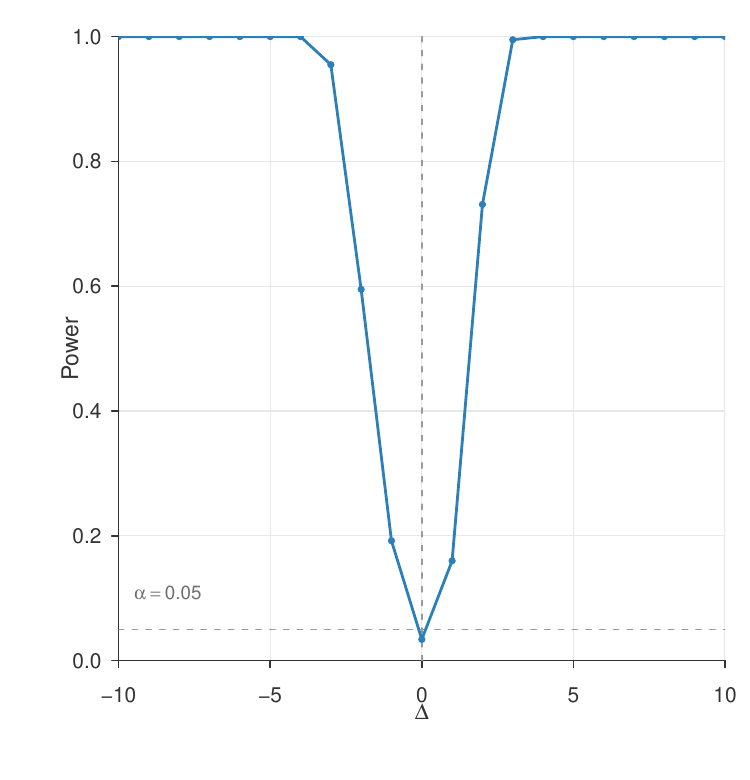}
    \caption{Method~3 (Simes, top 5\%).}
    \label{fig:power_m3}
  \end{subfigure}
  \caption{Mean-shift rejection curves $\hat{\pi}_i(\Delta)$ for Methods~1-3 as a function of the imposed centering shift $\Delta$ ($n = 20000$, $B = 1000$, $\alpha = 0.05$, Chung-Lu weight tail index $3.15$). The dashed horizontal line marks the nominal level $\alpha = 0.05$.}
  \label{fig:power_all}
\end{figure}
Figure~\ref{fig:power_all} shows mild conservativeness under the rank-one null, with an empirical rejection rate of $0.034$ for all three methods at $\Delta=0$. The rejection rate increases rapidly away from zero. At $\Delta=2$, the rates are $0.554$, $0.698$, and $0.731$ for Methods~1--3, respectively, and at $\Delta=3$ they rise to $0.844$, $0.983$, and $0.995$. The corresponding rates for negative shifts are slightly lower but have the same ordering once the departure is moderate. Thus Methods~2 and~3 are more sensitive than Method~1 to moderate common centering errors because they combine evidence across high-degree classes. Method~3 provides only a modest additional gain over Method~2.

\subsection{Specified non-rank-one null}
\label{subsec:irg_non1_power}
We next repeat the mean-shift analysis under a specified non-rank-one IRG null. The testing logic is the same as in Section~\ref{subsec:cl_power}, i.e. the three methods are applied to high-degree neighbor-degree statistics, and the rejection curve is obtained by shifting the null centering by $\Delta$. Rather than Chung-Lu quantities from Theorem~\ref{thm:clt-rank-one}, the null centering and null scale are the degree-dependent quantities in Theorem~\ref{thm:full-clt}.

We generate a two-type non-rank-one IRG. Each vertex first receives an independent type $c_i\in\{1,2\}$ with probabilities
\[
  \mathbb{P}(c_i=1)=0.3,\qquad
  \mathbb{P}(c_i=2)=0.7.
\]
Conditional on its type, vertex $i$ receives weight
\[
  w_i=10\,U_i^{-1/\alpha_{c_i}},\qquad
  U_i\sim\mathrm{Uniform}(0,1),
\]
where $\alpha_1=2.15$ and $\alpha_2=2.90$ are the type-specific tail indices. Edges are then generated independently with probabilities
\[
  p_{ij}
  = \min\!\left\{
    \frac{w_iw_j\,B_{c_ic_j}}{\sum_{\ell=1}^n w_\ell},\,1
  \right\},
  \qquad
  B=
  \begin{pmatrix}
    1.2 & 0.1\\
    0.1 & 0.9
  \end{pmatrix}.
\]
Since $B$ has rank larger than one, this model is not a Chung-Lu graph.
The null hypothesis for this experiment is
\[
  H_0:\quad
  \mathbb{E}\{\bar D_k\mid D_\varnothing=k\}=1+\bar\rho_k,
\]
with the non-rank-one centering given by Theorem~\ref{thm:full-clt}. In the simulation these quantities are estimated from the fixed type and weight sequence, giving plug-in values $\hat m_0(k)=1+\hat{\bar\rho}_k$ for each observed high-degree class. 
For finite-sample calibration we standardize by the full conditional standard deviation $s_k=\mathrm{sd}(\bar D_k\mid D_\varnothing=k)$, which combines all three orthogonal contributions of Remark~\ref{rem:r3-fails} such as the within-neighborhood sampling term ($O(k^{-1/2})$), the between-level-set term $|r'(\ell_k^*)|\sqrt{v_k}$, and the within-level-set term $\omega(\ell_k^*)$. This two-type block kernel has community-dependent affinity, so condition~(R3) need not hold and the within-level-set term need not be negligible.
We therefore use the full $s_k$ rather than its between-level-set component alone and regard the normal calibration as a finite-sample approximation. Figure~\ref{fig:power_all_non1} shows that combining this full scale with normal critical values yields conservative tests at $n=5000$, especially after adjustment across degree classes.

For each of the $B = 1000$ replicated networks, we compute the same statistics $T_1$ and $\{T_k:k\in\mathcal K^+_{\mathrm{level}}\}$ as in Section~\ref{subsec:cl_power}. For a degree class $k$, the standardized statistic under the non-rank-one null is
\[
  Z_k
  = \frac{T_k-\hat m_0(k)}
       {\hat s_0(k)/\sqrt{n(k)}}.
\]
The rejection rate at shift $\Delta$ is estimated by replacing $\hat m_0(k)$ with $\hat m_0(k)+\Delta$ in the standardization, analogously to Section~\ref{subsec:cl_power}. Thus
\[
  \hat{\pi}_i(\Delta)
  = \frac{1}{B} \sum_{b=1}^B
    \mathbf{1}\!\left\{ T_i^{(b)} \in R_i(\hat m_0+\Delta,\hat s_0) \right\},
  \quad i = 1, 2, 3,
\]
evaluated over $\Delta\in\{-10,-9,\ldots,9,10\}$ at nominal level $\alpha=0.05$. By construction, $\hat\pi_i(0)$ estimates the Type~I error of Method~$i$ under the non-rank-one IRG null.
Figure~\ref{fig:power_all_non1} displays the rejection curves.
\begin{figure}[htbp]
  \centering
  \begin{subfigure}[b]{0.32\textwidth}
    \centering
    \includegraphics[width=\linewidth]{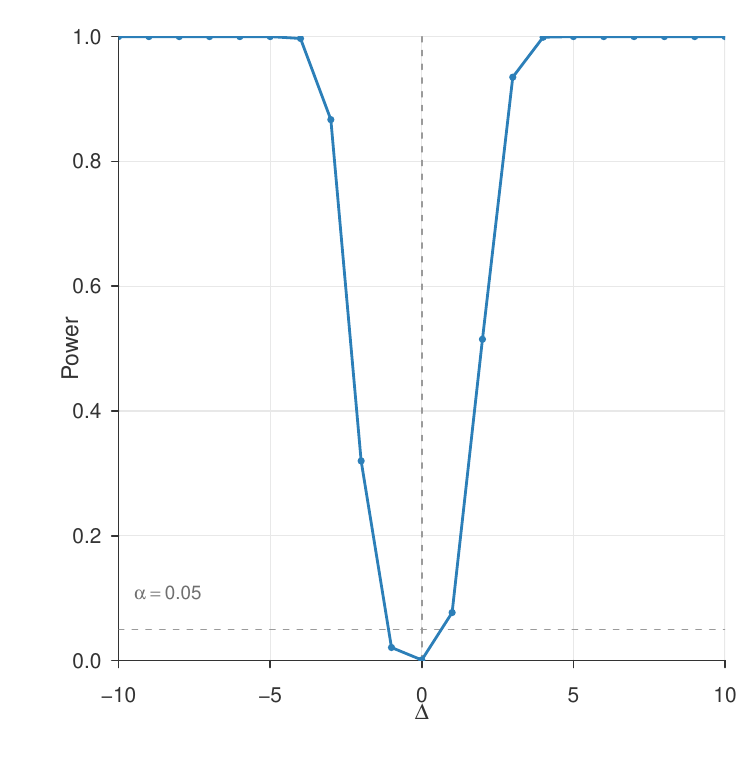}
    \caption{Method~1 (max-degree hub).}
    \label{fig:power_m1_non1}
  \end{subfigure}
  \hfill
  \begin{subfigure}[b]{0.32\textwidth}
    \centering
    \includegraphics[width=\linewidth]{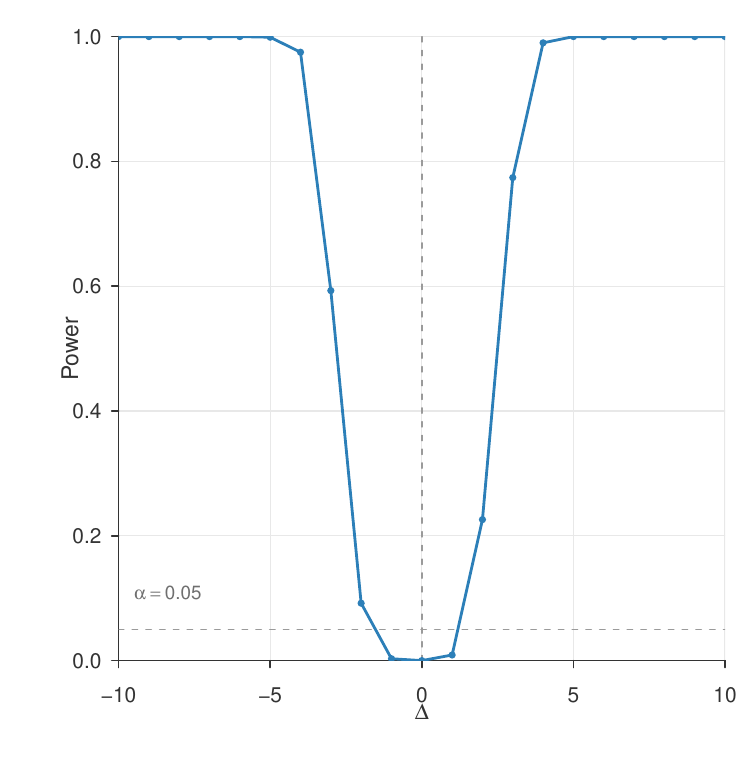}
    \caption{Method~2 (Bonferroni, top 5\%).}
    \label{fig:power_m2_non1}
  \end{subfigure}
  \hfill
  \begin{subfigure}[b]{0.32\textwidth}
    \centering
    \includegraphics[width=\linewidth]{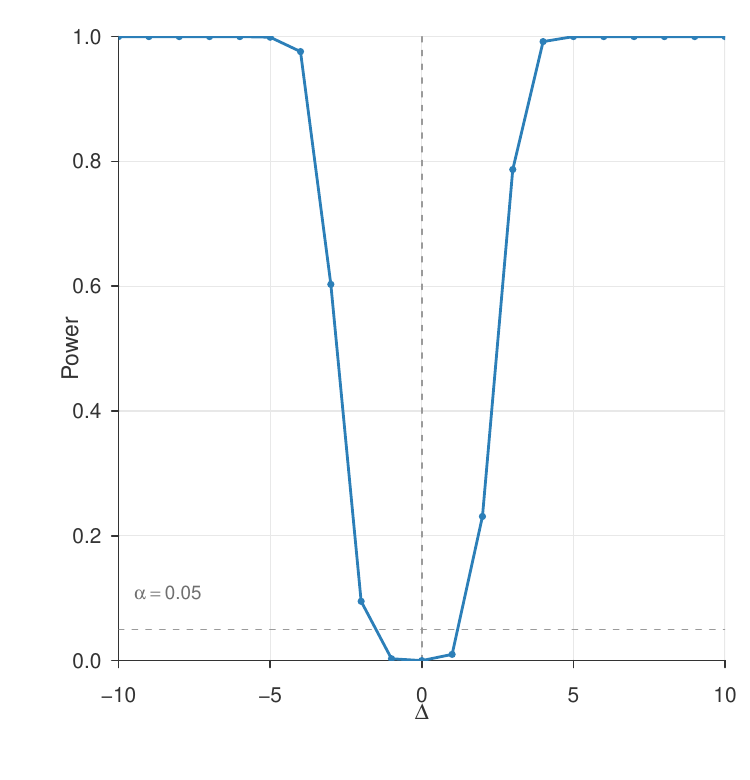}
    \caption{Method~3 (Simes, top 5\%).}
    \label{fig:power_m3_non1}
  \end{subfigure}
  \caption{Mean-shift rejection curves $\hat{\pi}_i(\Delta)$ for Methods~1-3 under the non-rank-one IRG null as a function of the imposed centering shift $\Delta$ ($n = 5000$, $B = 1000$, $\alpha = 0.05$). The dashed horizontal line marks the nominal level $\alpha = 0.05$.}
  \label{fig:power_all_non1}
\end{figure}
At $\Delta=0$, the rejection rates of Methods~1--3 are $0.001$, $0$, and $0$, respectively. Thus the finite-sample calibration is strongly conservative in this non-rank-one experiment. The rejection rate still increases as $|\Delta|$ grows, but the ordering differs from the rank-one case. At $\Delta=2$, Method~1 rejects with probability $0.515$, compared with $0.226$ and $0.231$ for Bonferroni and Simes. At $\Delta=-2$, the corresponding rates are $0.320$, $0.092$, and $0.095$. Here the cost of multiplicity and conservative class-level calibration exceeds the gain from pooling. The experiment therefore demonstrates strong control of false positives but also exposes a finite-sample sensitivity loss for the pooled procedures under a heterogeneous non-rank-one null.

\subsection{Sensitivity to a degree-distribution mismatch}
\label{subsec:tail_power}
The preceding two experiments trace sensitivity by restandardizing a fixed set of replications after shifting the plug-in null mean. We now use a network alternative whose degree distribution departs from the null by a controlled amount. Throughout this and the following subsections the Chung-Lu null uses Pareto weights with tail index $a_0 = 3.15$, so the third moment $\mathbb E[W^3]$ is finite and the null scale $\hat{\tilde\sigma}_\infty$ is
stably estimated.

We fix a scale-free null with Pareto weights of tail index $a_0$ and mean $\mu_W = 20$,
\[
  w_i = w_{\min}\,U_i^{-1/a_0},
  \qquad w_{\min} = \mu_W\,\frac{a_0-1}{a_0},
  \qquad U_i \sim \mathrm{Uniform}(0,1),
\]
drawn once and held fixed across replications. The null parameters are calibrated directly from the weight sequence, using the same rank-one centering $\mu_0$ and variance factor $\tilde\sigma_\infty^2$ defined in Section~\ref{sec:gof-notation}. Because the latent weights are observed in this experiment, $\hat\mu_0$ and $\hat{\tilde\sigma}_\infty^2$ are computed from the empirical moments of the fixed null weight sequence rather than from the degree-based plug-ins required for observed networks. This removes degree-moment estimation noise without changing the null quantities being estimated. The alternative is the same Chung-Lu model with all weights scaled by $1+\varepsilon$, so the mean degree increases by a fraction $\varepsilon$ while the tail index is unchanged. At $\varepsilon = 0$ the rejection rate estimates the size; for $\varepsilon > 0$ it estimates power against a pure scale departure. We take $a_0 = 3.15$, $n = 20000$, $B = 1000$.

\begin{figure}[htbp]
	\centering
	\includegraphics[width=\linewidth]{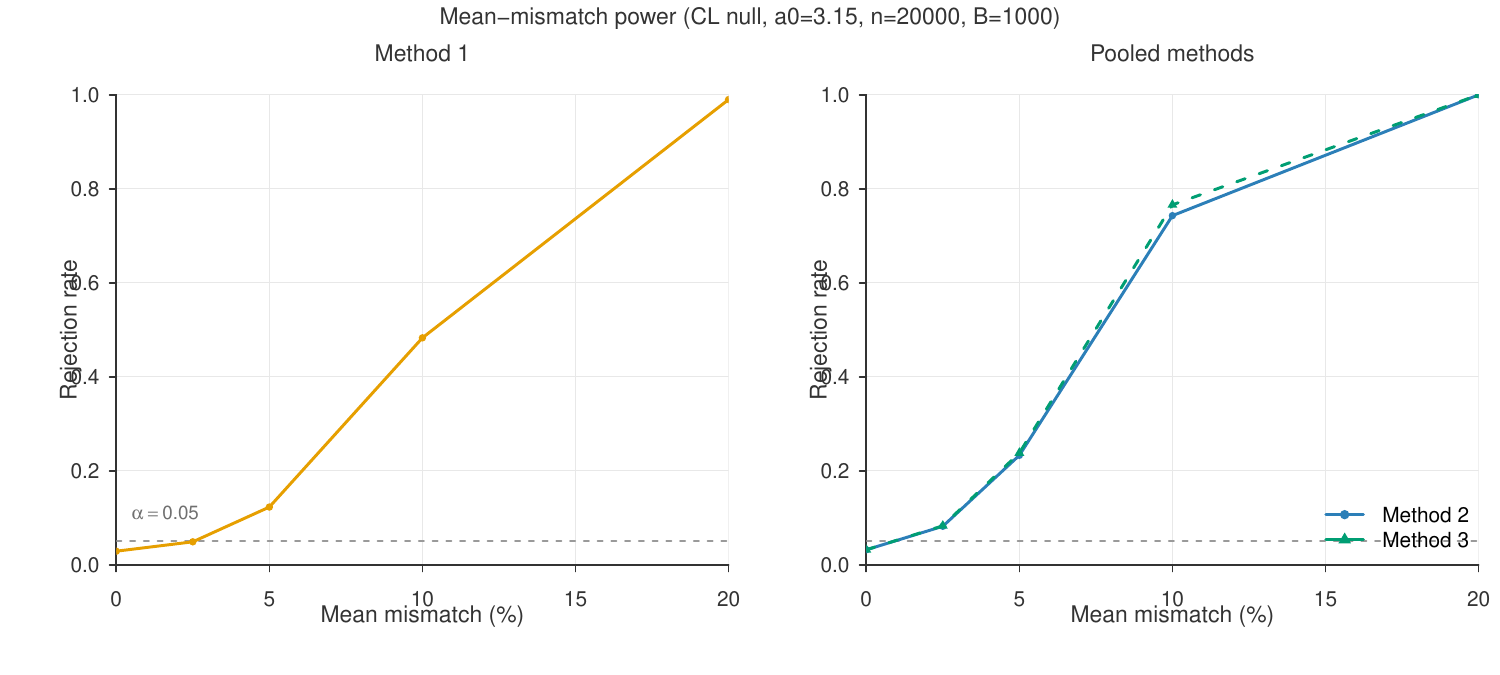}
	\caption{Empirical power under a controlled mean mismatch (Pareto null,
		$a_0 = 3.15$, $n = 20000$, $B = 1000$). Left: Method~1 (single max-degree hub), shown over the full mismatch range. Right: Methods~2 and~3 (pooled hubs), which rise much more steeply. The dashed line marks $\alpha = 0.05$.}
	\label{fig:meanmismatch}
\end{figure}
As seen from Figure~\ref{fig:meanmismatch}, at $\varepsilon=0$, the rejection rates are $0.029$, $0.032$, and $0.032$ for Methods~1--3, confirming mild conservativeness. A $2.5\%$ mean increase produces little power for Method~1, whose rejection rate is $0.049$, but raises the pooled rates to $0.082$ and $0.083$. At a $5\%$ increase, the rates are $0.123$, $0.233$, and $0.238$, and at a $10\%$ increase they reach $0.483$, $0.743$, and $0.766$. All three tests have essentially complete power at $20\%$. The near equality of Bonferroni and Simes indicates that pooling, rather than the choice between the two multiplicity corrections, drives the gain in this experiment.

\subsection{Level tests against preferential attachment}
\label{subsec:irg_pa_power}
We further evaluate the level tests against preferential attachment while matching the first-order degree behavior. For each PA parameter pair $(m,\delta)$ we construct a matched Chung-Lu null with Pareto weights of tail index
\[
  a_0 = 2 + \delta/m,
\]
which is equal to the PA survival-function exponent, i.e.\ $\mathbb{P}(D>k)\sim k^{-(2+\delta/m)}$. The mean weight $\mu_W = 2m$ is equal to the PA expected degree.
This matching makes the degree distributions comparable in tail exponent and mean, so rejection is not driven merely by a gross location or tail mismatch.

\begin{remark}[Why the degree tail cannot identify the mechanism]
\label{rem:tail-identifiability}
By construction, the matched Chung-Lu null and $\mathrm{PA}_n^{(m,\delta)}$ share the degree CCDF exponent $2+\delta/m$. Therefore a consistent tail-index estimator converges to the same value under both models and cannot, on its own, distinguish the mechanisms. This is the population-level reason a degree-based diagnostic is weak here, and it is why we use the second-order neighbor-degree slope, whose residual trend is $0$ under any rank-one IRG but $m+\delta>0$ under PA. 
Two caveats are worth stating.
First, the matching is asymptotic. At finite $n$, a naive tail-index estimate can separate the models because of pre-asymptotic bias. At the representative cell with target exponent $3.5$, a deep-tail maximum-likelihood estimate moves from $2.8$ toward $3.5$ as the threshold and $n$ increase, reaching $3.2$ at $n=20000$. This finite-sample gap reflects degree-distribution shape rather than cumulative advantage.
Second, this is why we frame the slope as a directional diagnostic rather than staging a degree-distribution power comparison, which would reward estimator bias rather than mechanism.
\end{remark}

The null hypothesis is
\[
  H_0: \text{the network is generated from the Chung--Lu model},
\]
whereas the alternative hypothesis is
\[
    H_a: \text{the network is generated from the $\text{PA}_n^{(m,\delta)}$ model}.
\]
We use the loop-free simple linear preferential attachment model with initial attractiveness $\delta>0$, consistently with Section~\ref{sec:locallimits}. Starting from a fixed simple seed graph, each new node attaches to $m$ distinct existing nodes with probabilities proportional to $\deg(v)+\delta$; self-loops and parallel edges are excluded. All rejection regions are constructed under the Chung-Lu null. For each test statistic $X(g)$, the two-sided rejection region is
\[
  R = \left\{ X :
    \left| \frac{X - \hat{\mu}_0}{\hat{\tilde\sigma}_\infty / \sqrt{k n(k)}} \right| > z_{1-\alpha/2}
  \right\},
\]
where $\hat{\mu}_0$ and $\hat{\tilde\sigma}_\infty / \sqrt{k n(k)}$ are the null mean and standard error estimated from the Chung--Lu degree sequence, as defined in Section~\ref{sec:methods}.

The power against the PA alternative is
\[
  \pi(g_{\mathrm{PA}}) = \mathbb{P}_{\mathrm{PA}}\!\left\{ X(g_{\mathrm{PA}}) \in R \right\},
\]
estimated by generating $B = 1000$ independent PA networks $g_{\mathrm{PA}}^{(1)}, \ldots, g_{\mathrm{PA}}^{(B)}$, each with $n = 20000$ nodes, and recording the proportion whose test statistic falls in $R$:
\[
  \widehat{\pi}(g_{\mathrm{PA}})
  = \frac{1}{B} \sum_{b=1}^{B}
    \mathbf{1}\!\left\{ X\!\left(g_{\mathrm{PA}}^{(b)}\right) \in R \right\}.
\]

Table~\ref{tab:allPA} reports results at the representative configuration $m=2$, $\delta/m=1.5$, $a_0=3.5$, and $\mu_W=4$. Under the matched Chung--Lu null, the empirical sizes of Methods~1--3 are $0.042$, $0.056$, and $0.056$. Their power against PA is $0.954$, $0.997$, and $0.997$, respectively, so the near-complete rejection by the pooled tests cannot be attributed to substantial size distortion. These methods detect a departure in the level of $\bar D_k$, whereas Method~4 in Section~\ref{subsec:slope_power} asks the more specific question of whether the residual increases with $\log k$.

\subsection{Log-degree slope test against preferential attachment}
\label{subsec:slope_power}

We apply Method~4 to the same per-cell matched Chung--Lu null and $\mathrm{PA}_n^{(m,\delta)}$ alternatives as in Section~\ref{subsec:irg_pa_power}, with $n=20000$ nodes. The null estimates $\hat\mu_0$, $\hat{\tilde\sigma}_\infty$ are computed from each network's own degree sequence. Because Method~4 is a \emph{slope} test, its high-degree set $\mathcal K^+_{\mathrm{slope}}$ uses the upper $50\%$ of distinct degrees rather than the top $5\%$ used by the level tests. The simulation is designed to estimate the PA-induced $\log k$ trend over a broad moderate-to-high range rather than only the most extreme degrees. The regression residuals $R_k=T_k-\hat\mu_0$ and weights $w_k=k\,n(k)/\hat{\tilde\sigma}_\infty^2$ are then formed across all $k\in\mathcal K^+_{\mathrm{slope}}$.

For each of the $B=1000$ Chung-Lu replications, we compute the slope statistic $Z_4$ and record whether $Z_4>z_{0.95}$ (one-sided test at $\alpha=0.05$). The empirical type-I error is
\[
  \hat\alpha_4
  = \frac{1}{B}\sum_{b=1}^B \mathbf{1}\!\left\{Z_4^{(b)}>z_{0.95}\right\}.
\]
Figure~\ref{fig:size_m4} shows a QQ-plot of the $B$ values of $Z_4$ against the standard normal, providing a visual check of the analytic null calibration in finite samples.

\begin{figure}[htbp]
  \centering
  \includegraphics[width=0.45\linewidth]{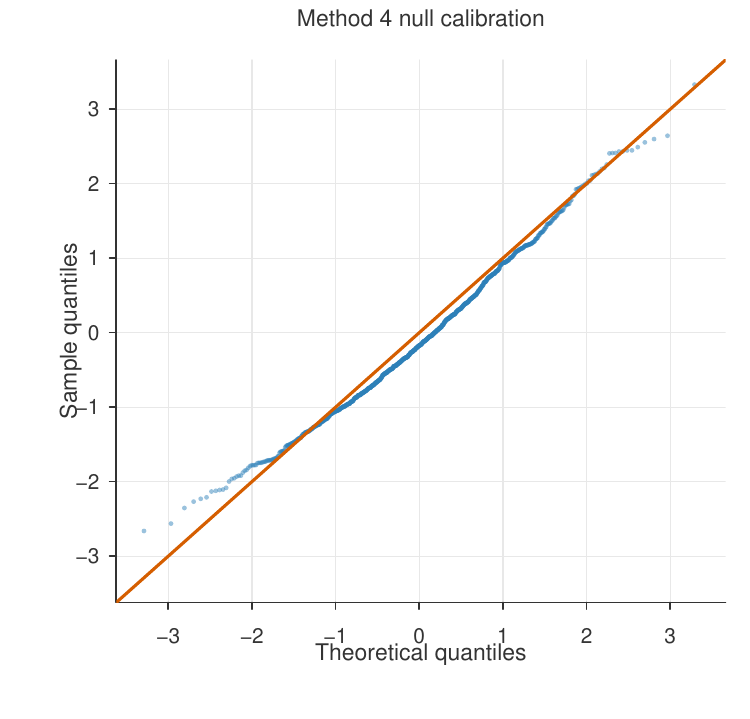}
  \caption{Normal QQ-plot of $Z_4$ under the matched Chung-Lu null ($a_0=3.5$, $\mu_W=4$, $n=20000$, $B=1000$). Agreement with the $45^\circ$ line indicates that the analytic $N(0,1)$ calibration is adequate at this sample size.}
  \label{fig:size_m4}
\end{figure}

At the representative matched-null configuration, $Z_4$ has empirical mean $-0.105$ and standard deviation $0.964$. Its upper-tail rejection rate is $0.043$. The slight left shift is visible in the QQ-plot, but the first two moments and the rejection rate support the working standard-normal calibration at $n=20000$.

For each of the $B=1000$ PA replications, we compute $Z_4$ using the same null estimates $(\hat\mu_0,\hat{\tilde\sigma}_\infty)$ fitted to the PA network's own degree sequence (plug-in under the null), and record whether $Z_4>z_{0.95}$.
The empirical power is $\hat\pi_4=B^{-1}\sum_b \mathbf{1}\{Z_4^{(b)}>z_{0.95}\}$. We vary $(m,\delta)$ over the grid $m\in\{2,4,6\}$ and $\delta/m\in\{1.15,1.25,1.5,2.0\}$. This grid lies in the finite-third-moment range for the matched tail exponent and therefore keeps the plug-in null scale stable across cells. Figure~\ref{fig:power_m4} plots $\hat\pi_4$ against $\delta/m$ for each $m$.

\begin{figure}[htbp]
  \centering
  \includegraphics[width=0.55\linewidth]{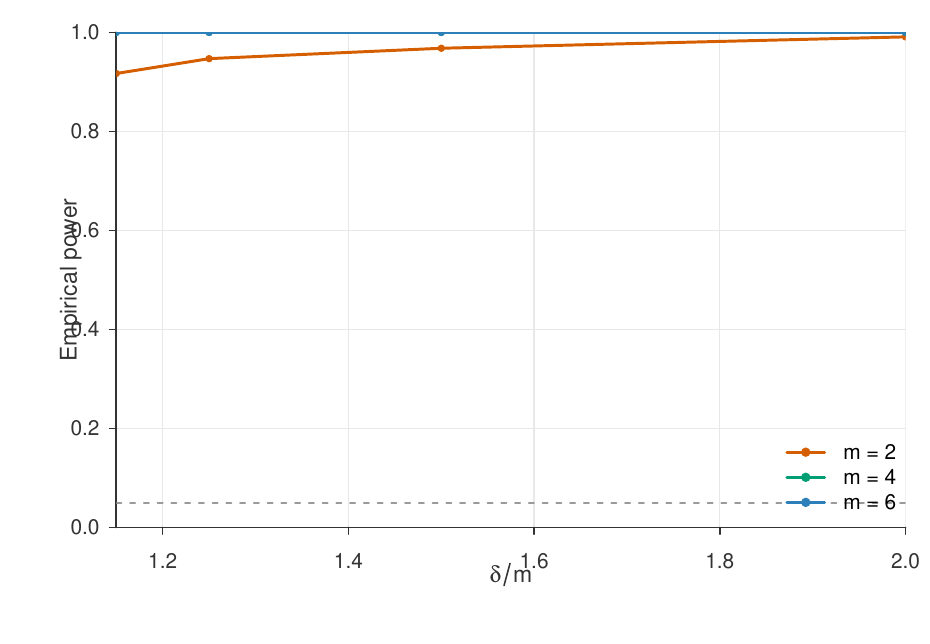}
  \caption{Empirical power $\hat\pi_4$ of Method~4 against
    $\mathrm{PA}_n^{(m,\delta)}$ as a function of $\delta/m$ for $m\in\{2,4,6\}$
    ($n=20000$, $B=1000$, $\alpha=0.05$, one-sided).
    The dashed line marks the nominal level.}
  \label{fig:power_m4}
\end{figure}

Table~\ref{tab:allPA} collects the empirical power of all four methods at the same representative configuration $(m=2,\delta/m=1.5)$ as in Section~\ref{subsec:irg_pa_power}.
Method~4 accumulates the $\log k$ signal across all high-degree classes through a single degree-of-freedom test rather than comparing levels class by class.

\begin{table}[htbp]
  \centering
  \caption{Empirical power against $\mathrm{PA}_n^{(2,3)}$ ($m=2$, $\delta/m=1.5$, matched CL null with $a_0=3.5$, $\mu_W=4$) for all four methods ($n=20000$, $B=1000$, $\alpha=0.05$). Methods~1--3 use two-sided tests; Method~4 uses a one-sided test.}
  \label{tab:allPA}
  \begin{tabular}{lcc}
    \toprule
    Method & Test type & Power $\hat\pi$ \\
    \midrule
    Method~1 (max-degree hub)      & two-sided & 0.954 \\
    Method~2 (Bonferroni, top 5\%) & two-sided & 0.997 \\
    Method~3 (Simes, top 5\%)      & two-sided & 0.997 \\
    Method~4 (log-degree slope)    & one-sided & 0.968 \\
    \bottomrule
  \end{tabular}
\end{table}

Across the 12 matched Chung-Lu cells, the empirical size of Method~4 ranges from $0.027$ to $0.060$. Its PA power increases from $0.917$ to $0.991$ as $\delta/m$ increases from $1.15$ to $2.0$ when $m=2$, and it equals $1.000$ in every cell with $m=4$ or $m=6$. At the representative configuration, its power is $0.968$, compared with $0.954$ for Method~1 and $0.997$ for both pooled level tests. Thus Method~4 is not uniformly more powerful than Methods~2 and~3 in these alternatives. Its contribution is a calibrated, one-sided assessment of the positive $\log k$ trend predicted by PA, while the level tests retain greater sensitivity to broad two-sided departures. Appendix~\ref{sec:appendix-sim-comparisons} compares this directional test with unweighted ANND and spectral baselines.

\section{Real Data Analysis}\label{sec:realdata}

The main applications contrast a high-school contact network, where class labels expose observed non-rank-one structure, with an arXiv coauthorship network, where cumulative advantage is plausible. The Reddit appendix adds large interaction graphs with negative residual slopes. In each case, Methods~1--3 test Chung--Lu hub-neighborhood levels and Method~4 estimates the residual direction over a reported upper-degree window. Appendix~\ref{subsec:real-tail-calibration} reports the tail-index diagnostics, calibration choices, and descriptive assortativity coefficients.

\subsection{High-school contact networks}
\label{subsec:highschool_examples}
We first analyze SocioPatterns high-school data collected with wearable proximity sensors \citep{fournet2014contact}. Each record marks a 20-second face-to-face contact, and each student has a class label. These labels distinguish dense within-class mixing from the cross-class contacts that bridge groups.

The raw proximity file contains $188508$ time-stamped contacts involving $327$ students, observed over $7375$ time points. We aggregate these temporal records
into an undirected simple graph on student pairs, so an edge means that two students were observed in contact at least once during the study. This all-contact graph has $5818$ distinct edges, mean degree $35.6$, and maximum degree $87$. To examine the role of class structure, we split the same edge list into two subgraphs. 
The \emph{within-class} graph contains $4016$ unique student pairs from the same class, while the \emph{cross-class} graph contains $1802$ pairs from different classes.
The companion friendship graph comes from the directed, unweighted
friendship survey administered in the same December 2013 Marseille high-school
study \citep{mastrandrea2015contact}.  A row $i\ j$ records that student $i$
reported a friendship with student $j$.  We remove self-nominations and
duplicate directed records, replace each unordered pair with one undirected
edge when either direction is reported, and collapse reciprocal nominations
to that same edge.  The resulting all-nomination projection contains $668$
directed nominations, $406$ unordered pairs among $134$ observed students,
and maximum degree $17$; $262$ pairs are reciprocal and $144$ are reported in
only one direction.  We apply the level tests before the slope diagnostic to
the three proximity projections and this friendship projection.

Table~\ref{tab:highschool-level} shows no level rejection for the all-contact graph. Both decomposed graphs reject, with overwhelming evidence within classes and a fitted-null bootstrap $p$-value of $0.001$ for every level test across classes. Aggregation can therefore hide local structure that becomes visible after separating contacts by class. The smaller friendship-nomination graph does not reject and is treated as exploratory because of its narrow degree range.

\begin{table}[htbp]
  \centering
  \small
  \caption{Methods~1-3 on the selected high-school graphs under the Chung-Lu
    null. The entries are two-sided $p$-values. The proximity decompositions
    reveal level departures that are hidden in the all-contact graph. The
    cross-class entries use $999$ fitted-null bootstrap graphs; the remaining
    entries use the analytic calibration.}
  \label{tab:highschool-level}
  \begin{tabular}{p{0.27\textwidth}ccccc}
    \toprule
    Network & $\hat\mu_0$ & $\hat{\tilde\sigma}_\infty$ & Method~1 & Method~2 & Method~3 \\
    \midrule
    Proximity, all contacts & 40.70 & 13.71 & $8.44\times10^{-1}$ & $6.83\times10^{-1}$ & $5.81\times10^{-1}$ \\
    Proximity, within class & 26.97 & 6.75 & $1.32\times10^{-9}$ & $<10^{-16}$ & $<10^{-16}$ \\
    Proximity, cross class & 17.93 & 11.49 & $1.00\times10^{-3}$ & $1.00\times10^{-3}$ & $1.00\times10^{-3}$ \\
    Friendship nominations & 8.01 & 3.59 & $2.28\times10^{-1}$ & $2.28\times10^{-1}$ & $2.28\times10^{-1}$ \\
    \bottomrule
  \end{tabular}
\end{table}

We then apply Method~4 to ask whether the departures are positive, negative, or flat after centering by the Chung--Lu plug-in mean. For comparability, all high-school graphs use the upper $25\%$ of distinct positive degrees in the slope regression. This gives 17, 10, and 10 slope classes for the all-contact, within-class, and cross-class proximity graphs, respectively. The friendship graph has only 5 slope classes, so we label its slope result exploratory.

\begin{table}[htbp]
  \centering
  \small
  \caption{Method~4 on the SocioPatterns high-school networks. The proximity
    and friendship graphs both use the upper $25\%$ of distinct degrees for the
    slope regression. Results are reported two-sided under the Chung--Lu null;
    the cross-class $p$-value uses $999$ fitted-null bootstrap graphs and the
    remaining $p$-values use the analytic calibration.}
  \label{tab:highschool-method4}
  \begin{tabular}{p{0.26\textwidth}rrrrrrl}
    \toprule
    Network & $n$ & Edges & $k_{\mathrm{slope}}$ & $M$ & $\hat\beta$ & $p$ & Decision \\
    \midrule
    Proximity, all contacts          & 327 & 5818 & 54 & 17 & $-5.50$ & $1.53\times10^{-2}$ & Reject CL (disassortative) \\
    Proximity, within class          & 327 & 4016 & 31 & 10 & $+13.65$ & $6.62\times10^{-13}$ & Reject CL (positive) \\
    Proximity, cross class           & 327 & 1802 & 32 & 10 & $-4.20$ & $1.61\times10^{-1}$ & Inconclusive \\
    Friendship nominations           & 134 & 406  & 13 & 5  & $+3.743$ & $3.739\times10^{-1}$ & Exploratory \\
    \bottomrule
  \end{tabular}
\end{table}

\begin{figure}[htbp]
	\centering
	\includegraphics[width=0.82\linewidth]{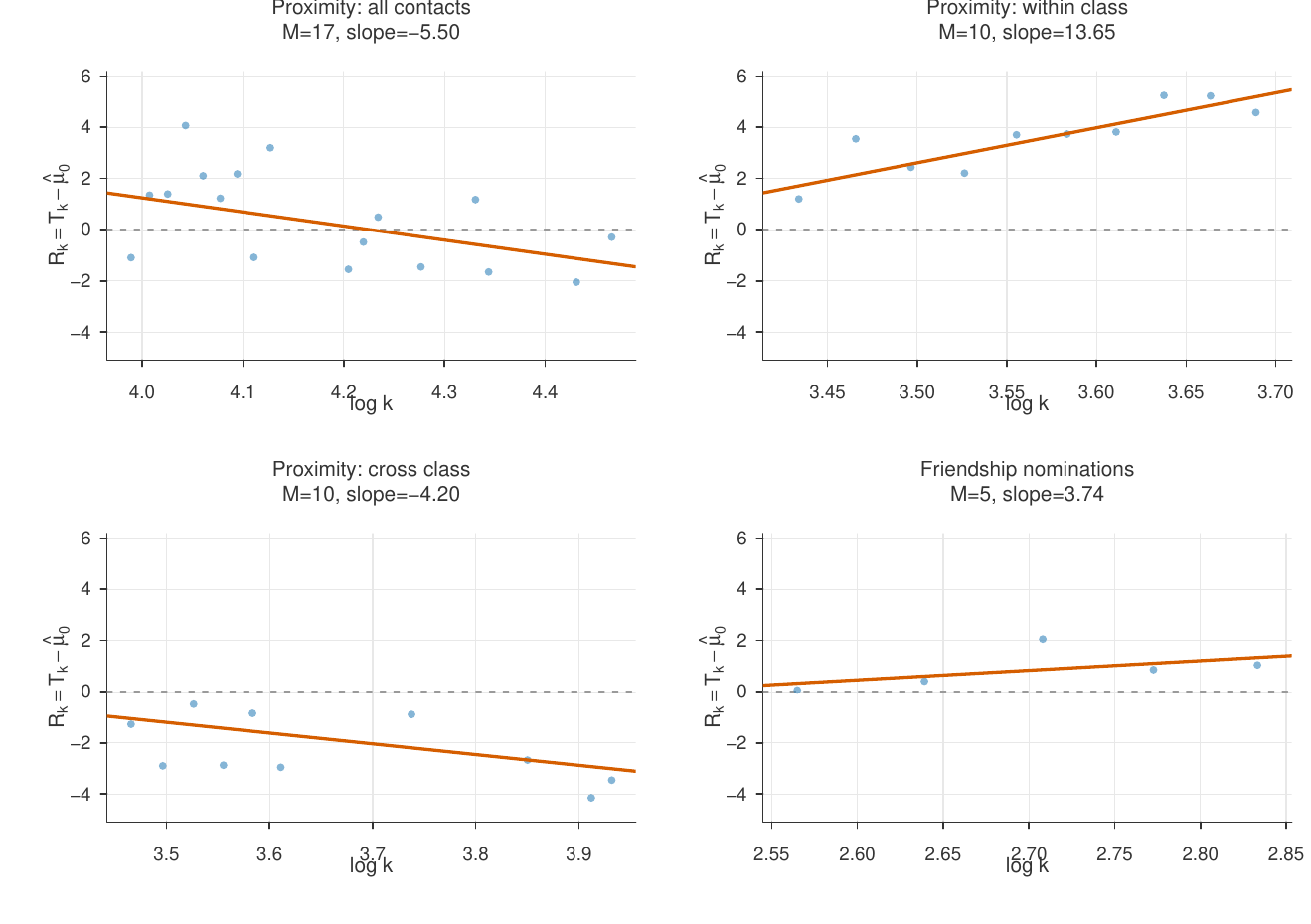}
	\caption{Method~4 residuals $R_k=T_k-\hat\mu_0$ against $\log k$ for the
		selected high-school graphs, with weighted least-squares fits (red). The
		within-class proximity graph shows a strong positive slope, whereas the
		full proximity graph is mildly negative and the cross-class graph is flat.}
	\label{fig:highschool-method4}
\end{figure}

Table~\ref{tab:highschool-method4} reports the Method~4 estimates, and Figure~\ref{fig:highschool-method4} displays the corresponding residual profiles. The directions differ sharply. The all-contact graph has a negative slope ($\hat\beta=-5.50$, $p=0.015$), whereas the within-class graph has a strong positive slope ($\hat\beta=13.65$, $p<10^{-12}$). The cross-class graph rejects in level but not in slope (bootstrap $p=0.161$). Hence the positive signal is specific to within-class contact rather than a property of the aggregate network.

Class labels also support a non-rank-one check. We fit the following degree-adjusted class-block null to each proximity projection.
\[
  \mathbb P(A_{ij}=1)
  =
  \min\{\gamma_{c_i c_j}d_i d_j,\,1\},
\]
Here $c_i$ and $d_i$ are the class label and observed degree, while $\gamma_{ab}$ matches the observed edge count between classes $a$ and $b$. We subtract the fitted block-adjusted neighbor mean and calibrate the Method~4 slope with $1000$ conditional bootstrap graphs. This check asks which trend remains after accounting for class mixing.

\begin{table}[htbp]
  \centering
  \small
  \caption{Class-block degree-adjusted Method~4 diagnostic for the high-school
    proximity graphs. The bootstrap interval is the central $95\%$ range of the
    slope under the fitted class-block null.}
  \label{tab:highschool-block}
  \begin{tabular}{p{0.31\textwidth}rrrrl}
    \toprule
    Network & $k_{\mathrm{slope}}$ & $M$ & $\hat\beta_{\mathrm{block}}$ & Null range & $p$ \\
    \midrule
    Proximity, all contacts & 54 & 17 & $+11.63$ & $[-6.65,\;7.13]$ & $1.00\times10^{-3}$ \\
    Proximity, within class & 31 & 10 & $+2.80$  & $[-4.21,\;4.82]$ & $2.20\times10^{-1}$ \\
    Proximity, cross class  & 32 & 10 & $-1.16$  & $[-6.97,\;7.48]$ & $7.11\times10^{-1}$ \\
    \bottomrule
  \end{tabular}
\end{table}

Table~\ref{tab:highschool-block} shows that after class adjustment, the strong within-class slope is no longer significant, whereas the all-contact graph shows a positive residual slope. The all-contact residual slope changes sign between the two nulls, from $\hat\beta=-5.50$ under the rank-one null to $\hat\beta_{\mathrm{block}}=+11.63$ under the class-block null. This is not a contradiction: the residual slope is defined relative to each null's predicted hub-neighbor level, and the two nulls differ. The rank-one null does not absorb the disassortative cross-class mixing, so its all-contact residual trend is negative; the richer class-block null does absorb that mixing, and the trend that remains after it is positive. The rank-one analysis thus locates failure of degree matching, and the class-block analysis shows how much of that failure is explained by observed group structure.

For transmission studies, the results distinguish globally connected students from concentrated within-class exposure. Because no epidemic dynamics are fitted, this distinction should be treated as a mechanism hypothesis rather than a policy estimate.

\subsection{arXiv coauthorship}
\label{subsec:arxiv_coauthorship_analysis}

We next analyze the LINQS arXiv MRDM 2005 author-paper data (\url{https://linqs.org/datasets/}). Vertices are author clusters, and two vertices share an unweighted edge when their authors appear on a common paper. Repeated collaborations collapse to one edge. This setting is substantively relevant because visible authors may attract new collaborators over time.

Related studies of statistical coauthorship networks examine centrality, communities, productivity, and collaboration trajectories \citep{ji2016coauthorship,ji2022cocitation}. Our analysis instead tests the degree-conditional neighborhood profile of an arXiv-based graph. It asks whether hub neighborhoods follow the level and slope implied by a fitted sparse null.

The raw file contains $58515$ author-paper rows, representing $29555$ paper clusters and $8967$ author clusters. Of these papers, $11481$ are single-author papers and hence do not create coauthorship edges. The largest paper has $10$ authors. After connecting all author pairs who share at least one paper and collapsing duplicate pairs, the resulting graph has $8967$ vertices and $20248$ edges. There are $772$ isolated authors, the mean degree is $4.52$, and the maximum degree is $63$.

The Chung-Lu plug-in moments for the coauthorship graph are $\hat\mu_0=12.05$ and $\hat{\tilde\sigma}_\infty=11.19$. Methods~1-3 strongly reject the Chung-Lu hub-neighbor level prediction. The maximum-degree author has $k^*=63$ and $Z=6.53$, while the top-$5\%$ degree-class threshold for Methods~2-3 is $56$; all three fitted-null bootstrap $p$-values equal $0.001$. Thus the level tests show that a rank-one degree-corrected null does not reproduce the hub neighborhoods of this collaboration graph.

Table~\ref{tab:arxiv-coauthorship} summarizes the four test results.

\begin{table}[htbp]
  \centering
  \small
  \caption{Methods~1-4 on the LINQS arXiv coauthorship graph. Vertices
    are author clusters, edges indicate at least one shared paper, and Method~4
    uses the upper $50\%$ of distinct positive degrees. The $p$-values use
    $999$ fitted-null bootstrap graphs.}
  \label{tab:arxiv-coauthorship}
  \begin{tabular}{p{0.18\textwidth}rrrrrl}
    \toprule
    Method & $k$ threshold & $M$ & $Z$ & $\hat\beta$ & $p$ & Decision \\
    \midrule
    Method~1 & 63 & 1  & $6.53$ & --     & $1.00\times10^{-3}$ & Reject CL \\
    Method~2 & 56 & 3  & --     & --     & $1.00\times10^{-3}$ & Reject CL \\
    Method~3 & 56 & 3  & --     & --     & $1.00\times10^{-3}$ & Reject CL \\
    Method~4 & 29 & 29 & $2.53$ & $2.41$ & $1.00\times10^{-3}$ & Reject CL (positive) \\
    \bottomrule
  \end{tabular}
\end{table}

Method~4 gives $\hat\beta=2.41$ with studentized-bootstrap $95\%$ confidence interval $[0.60,4.22]$ and $Z_4=2.53$ over the upper-$50\%$ degree window. Figure~\ref{fig:arxiv-coauthorship-method4} shows the corresponding residual profile and weighted fit. The positive residual slope agrees with the PA direction under the Chung-Lu benchmark. The positive global assortativity coefficient of $0.241$ is descriptively consistent with this conclusion but does not supply its null calibration.

\begin{figure}[htbp]
  \centering
  \includegraphics[width=0.70\linewidth]{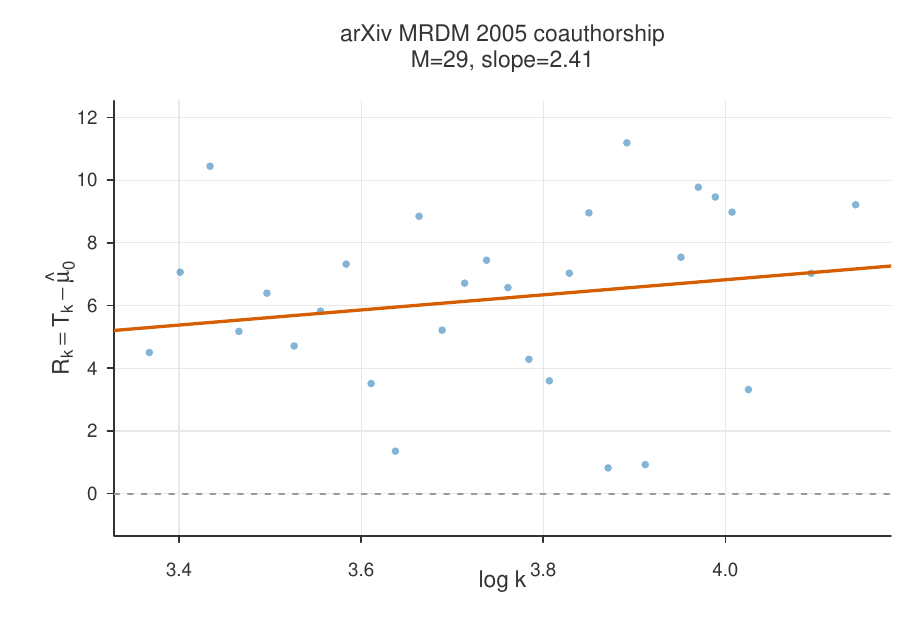}
  \caption{Method~4 residuals $R_k=T_k-\hat\mu_0$ against $\log k$ for the LINQS arXiv coauthorship graph, with the weighted least-squares fit (red).
    The fitted slope is significantly positive.}
  \label{fig:arxiv-coauthorship-method4}
\end{figure}

The Reddit interaction graphs in Appendix~\ref{sec:reddit_analysis} provide a large-network contrast to the arXiv result. Because their degree tails do not support the finite-third-moment analytic calibration, all reported Reddit $p$-values use the fitted-null bootstrap. The level tests reject Chung--Lu for \texttt{r/nba}, \texttt{r/CFB}, and \texttt{r/pics}; for \texttt{r/funny}, Methods~2--3 reject while the single max-degree test does not. Method~4 estimates negative residual slopes for \texttt{r/nba} ($\hat\beta=-53.30$, $p=0.001$), \texttt{r/CFB} ($\hat\beta=-32.80$, $p=0.001$), and \texttt{r/funny} ($\hat\beta=-3.37$, $p=0.009$), while \texttt{r/pics} is inconclusive ($\hat\beta=2.24$, $p=0.072$). Thus the same diagnostic that finds a positive coauthorship slope also detects disassortative hub neighborhoods in large interaction networks.

\section{Conclusion}
\label{sec:conclusion}

Taken together, our results show that matching the degree distribution does not ensure a credible model for hub neighborhoods. The local-limit calibration turns this discrepancy into two complementary questions concerning whether the predicted neighbor-degree levels fit and how the remaining error changes across hub degrees. Finite-sample experiments reveal a tradeoff. Pooling is effective for rank-one departures but can be costly when a heterogeneous non-rank-one null is conservatively calibrated, while the slope statistic contributes mechanism-relevant direction rather than universal power dominance. The empirical analyses also show that conclusions depend on network resolution and adjustment. Splitting school contacts exposes hidden lack of fit, class blocks explain only part of the resulting trend, and arXiv and Reddit exhibit opposite residual directions. These findings argue against inferring formation mechanism from degree tails or a single omnibus rejection and favor reporting the fitted level, residual slope, and null specification together.

Future work can extend the same idea in several directions. Directed and temporal versions would retain more information in citation, reply, and growth networks, making it possible to compare the fitted slope more directly with attachment-kernel parameters. Covariate-adjusted and latent-space nulls would allow the residualization step to incorporate observed groups or estimated types. Finally, resampling calibrations could make the method more robust for denser graphs, overlapping high-degree neighborhoods, and weighted networks.

\appendix

\section{Real-data calibration details}
\label{subsec:real-tail-calibration}

The analytic Chung--Lu calibration uses the condition
$\mathbb E[W^3]<\infty$.  We therefore assess the degree tail before choosing
between the analytic and bootstrap reference distributions.  We use the
survival-function convention
\[
  \mathbb P(D\geq x)=x^{-a}L(x),
\]
where $L$ is slowly varying.  Under a rank-one mixed-Poisson model with
regularly varying weights, the degree and weight distributions have the same
survival exponent.  Thus $a>3$ is the tail condition corresponding to
$\mathbb E[W^3]<\infty$.

For each graph, let $D_{(1)}\geq\cdots\geq D_{(n_+)}$ be the positive degrees.
We estimate $a$ with the Hill estimator
\[
  \widehat a
  =
  \left\{
    \frac{1}{r}\sum_{j=1}^{r}
    \log\!\left(\frac{D_{(j)}}{D_{(r+1)}}\right)
  \right\}^{-1}.
\]
To make the diagnostic reproducible rather than selecting a visually favorable
part of a Hill plot, we set $r=\lfloor\sqrt{n_+}\rfloor$.  Table
\ref{tab:real-tail-index} reports the resulting $k_{\min}=D_{(r+1)}$ and a
nonparametric vertex-bootstrap $95\%$ interval based on $1999$ resamples.  The
estimates are diagnostics of moment plausibility, not claims that a power law
is the unique tail model, especially for the smaller school graphs with short
degree ranges.

\begin{table}[htbp]
  \centering
  \scriptsize
  \caption{Degree-tail estimates and calibration used in the real-data
    analyses.  The index $a$ is the survival-function exponent, so the
    analytic finite-third-moment condition corresponds to $a>3$ under the
    regularly varying rank-one model.}
  \label{tab:real-tail-index}
  \begin{tabular}{p{0.32\textwidth}rrrrl}
    \toprule
    Network & $k_{\min}$ & $r$ & $\widehat a$ & $95\%$ CI & Calibration \\
    \midrule
    High school, all proximity contacts  & 56  & 18  & 5.70  & [4.58, 10.83] & Analytic \\
    High school, within-class contacts    & 36  & 18  & 20.34 & [13.48, 45.47] & Analytic \\
    High school, cross-class contacts     & 27  & 17  & 3.29  & [2.50, 6.24] & Bootstrap \\
    High school, friendship nominations  & 11  & 11  & 5.66  & [3.36, 10.54] & Analytic \\
    LINQS arXiv MRDM 2005 coauthorship   & 28  & 90  & 3.42  & [2.95, 4.09] & Bootstrap \\
    Reddit \texttt{r/nba}                & 563 & 177 & 2.71  & [2.39, 3.09] & Bootstrap \\
    Reddit \texttt{r/CFB}                & 271 & 134 & 2.65  & [2.28, 3.10] & Bootstrap \\
    Reddit \texttt{r/pics}               & 112 & 351 & 2.26  & [2.07, 2.52] & Bootstrap \\
    Reddit \texttt{r/funny}              & 72  & 386 & 2.26  & [2.04, 2.48] & Bootstrap \\
    \bottomrule
  \end{tabular}
\end{table}

We use the analytic normal calibration only when the entire tail-index
confidence interval lies above $3$.  If the interval includes $3$ or extends
below it, we use the fitted-null bootstrap.  This rule is fixed before
examining the neighbor-degree test outcomes.  It assigns bootstrap calibration
to the cross-class proximity graph, arXiv graph, and all four Reddit graphs,
rather than relying on the finite-third-moment normal approximation for those
cases.

\begin{table}[htbp]
  \centering
  \small
  \caption{Descriptive degree assortativity coefficients for the main-text
    real-data graphs. These global coefficients are included as a comparator,
    not as calibrated null tests.}
  \label{tab:assortativity}
  \begin{tabular}{llr}
    \toprule
    Dataset & Network & Degree assortativity \\
    \midrule
    High school & Proximity, all contacts & $+0.033$ \\
    High school & Proximity, within class & $+0.360$ \\
    High school & Proximity, cross class & $-0.087$ \\
    High school & Friendship nominations & $+0.287$ \\
    Coauthorship & LINQS arXiv, author collaboration & $+0.241$ \\
    \bottomrule
  \end{tabular}
\end{table}
Table~\ref{tab:assortativity} reports Newman's assortativity coefficient as a descriptive comparator. Because this global coefficient is neither calibrated to the fitted null nor separated into level and trend components, it does not answer the degree-conditional question tested here.

\section{Additional simulation comparisons}
\label{sec:appendix-sim-comparisons}

\subsection{Value of the calibration: the uncalibrated ANND slope}
\label{subsec:annd_baseline}
The statistic $T_k$ is the empirical average nearest-neighbor degree, a standard descriptor of degree dependence \citep{pastorsatorras2001dynamical,yao2018average}. The baseline regresses $T_k$ on $\log k$ by unweighted least squares over the same degree window. Method~4 instead uses the IRG variance $\tilde\sigma_\infty^2/\{k n(k)\}$. Its inverse-variance weights downweight noisy extreme classes, while $Z_4$ supplies an asymptotic $N(0,1)$ null and the PA limit $\hat\beta\to m+\delta$.

Both procedures control size.  The ANND baseline has estimated power $0.756$
in its $B=250$ comparison run, whereas the primary $B=1000$ Method~4
experiment in Table~\ref{tab:allPA} gives $0.968$.  We use that primary
Method~4 estimate here rather than the independent $B=250$ comparison-script
estimate $0.936$.

\begin{table}[htbp]
  \centering
  \caption{Calibrated slope test $Z_4$ versus the uncalibrated descriptive
    ANND-slope OLS $t$-test on the matched CL-versus-PA design at the
    representative configuration ($m=2$, $\delta/m=1.5$, $a_0=3.5$,
    $\mu_W=4$; $n=20000$, $\alpha=0.05$, one-sided).  The ANND row uses its
    $B=250$ comparison run; the Method~4 row repeats the primary $B=1000$
    estimate from Table~\ref{tab:allPA}.}
  \label{tab:annd-compare}
  \begin{tabular}{lcc}
    \toprule
    Slope test & Size under CL & Power vs PA \\
    \midrule
    Uncalibrated ANND-slope (OLS $t$) & 0.040 & 0.756 \\
    Calibrated $Z_4$ (Method~4; $B=1000$) & 0.043 & 0.968 \\
    \bottomrule
  \end{tabular}
\end{table}

\subsection{Comparison with a spectral goodness-of-fit test}
\label{subsec:spectral_compare}
A natural alternative is a \emph{spectral} goodness-of-fit test in the Lei-Bickel-Sarkar family \citep{lei2016goodness,bickel2016hypothesis}, which asks whether the residual of a fitted low-rank model carries an outlier eigenvalue. We adapt it to our rank-one (Chung-Lu) null. We fit $\hat P_{ij}=d_id_j/(2|E|)$, form the residual $R=A-\hat P$, and use the largest singular value $\sigma_1(R)$ as the statistic, computed matrix-free via Lanczos.
Because the adjacency spectrum of a sparse heavy-tailed graph is dominated by hubs ($\|A\|\sim\sqrt{d_{\max}}$), an unregularized residual mostly rediscovers the degree tail. We therefore also report a hub-regularized version that trims vertices of degree $>8\,\bar d$ before forming the residual, as recommended for the sparse regime \citep{le2017concentration}.
Since Tracy-Widom calibration is not valid at this sparsity, the spectral thresholds are Monte-Carlo calibrated from matched-CL replicates (size $=\alpha$ by construction). Table~\ref{tab:spectral-compare} reports the comparison at the representative configuration.

\begin{table}[htbp]
  \centering
  \caption{Spectral goodness-of-fit test ($\sigma_1$ of the degree-corrected
    residual, Monte-Carlo calibrated) versus the neighbor-degree slope $Z_4$
    on the matched CL-versus-PA design ($m=2$, $\delta/m=1.5$, $n=20000$,
    $\alpha=0.05$).  The spectral rows use $B=500$; the Method~4 row repeats
    the primary $B=1000$ estimate from Table~\ref{tab:allPA}.  Spectral size is
    $\alpha$ by construction; the informative size check is Method~4's
    analytic null.}
  \label{tab:spectral-compare}
  \begin{tabular}{lcc}
    \toprule
    Test & Size under CL & Power vs PA \\
    \midrule
    Spectral GoF, unregularized        & 0.05$^\dagger$ & 0.368 \\
    Spectral GoF, regularized (trim $8\bar d$) & 0.05$^\dagger$ & 0.988 \\
    Neighbor-degree slope $Z_4$ ($B=1000$) & 0.043 & 0.968 \\
    \bottomrule
  \end{tabular}\\[2pt]
  {\footnotesize $^\dagger$ Monte-Carlo calibrated, hence $\alpha$ by construction.}
\end{table}

The regularized spectral test has estimated power $0.988$ in its $B=500$ run,
slightly above the primary Method~4 estimate $0.968$ from $B=1000$, so raw
power does not favor the neighbor-degree statistic.  The exact binomial
$95\%$ interval for the latter is $[0.955,0.978]$; the two entries should still
be read as estimates from independent Monte Carlo runs.  The independent
$B=500$ spectral-comparison script gave $0.970$ for Method~4; we replace that
lower-precision rerun in the table with the primary $B=1000$ estimate so that
all three tables report the same canonical value.  The distinction is
interpretability and calibration. A spectral rejection is omnibus, whereas
the sign of $Z_4$ distinguishes the PA direction from disassortative mixing
and estimates the slope. Moreover, $Z_4$ has an analytic null in this regime,
while spectral performance changes from $0.368$ to $0.988$ with
regularization and requires Monte Carlo calibration.

\section{Proofs}
\label{sec:proofs}

\subsection{Proofs for Section~\ref{subsec:irg-results}}
\label{subsec:proof-irg}

\begin{proof}[Proof of Proposition~\ref{prop:irg-cond-mv}]
By Eq.\eqref{eq:neighbor-type}--\eqref{eq:offspring}, the pairs $(Y_i,\xi_i)$ are conditionally independent given $(X_\varnothing=x,D_\varnothing=k)$. For each $i$, applying the tower property:
\begin{align*}
  \mathbb{E}[D(i)\mid X_\varnothing=x,D_\varnothing=k]
  &= 1+\mathbb{E}[\xi_i\mid X_\varnothing=x,D_\varnothing=k]\\
  &= 1+\int_S\mathbb{E}[\xi_i\mid X_\varnothing=x,D_\varnothing=k,Y_i=y]\,Q_x(\mathrm{d}y)\\
  &= 1+\int_S\lambda(y)\,Q_x(\mathrm{d}y)=1+\rho(x),
\end{align*}
where the penultimate step uses $\mathbb{E}[\xi_i\mid x,k,Y_i=y]=\lambda(y)$ from Eq.\eqref{eq:offspring}. For the variance, apply the law of total variance conditioning on $Y_i$:
\begin{align*}
  \mathrm{Var}(D(i)\mid x,k)
  &= \mathbb{E}[\mathrm{Var}(D(i)\mid x,k,Y_i)\mid x,k]
   +\mathrm{Var}(\mathbb{E}[D(i)\mid x,k,Y_i]\mid x,k).
\end{align*}
From Eq.\eqref{eq:offspring}, $\mathrm{Var}(D(i)\mid x,k,Y_i=y)=\lambda(y)$ and $\mathbb{E}[D(i)\mid x,k,Y_i=y]=1+\lambda(y)$, so
\[
  \mathrm{Var}(D(i)\mid x,k)
  =\int_S\lambda(y)\,Q_x(\mathrm{d}y)+\mathrm{Var}_{Q_x}(\lambda(Y))
  =\rho(x)+\mathrm{Var}_{Q_x}(\lambda(Y))=\sigma^2(x). \qedhere
\]
\end{proof}

\begin{proof}[Proof of Theorem~\ref{thm:lln-irg}]
By Proposition~\ref{prop:irg-cond-mv} and the conditional independence of $D(1),\ldots,D(k)$ given $(X_\varnothing=x,D_\varnothing=k)$,
\[
  \mathrm{Var}(\bar D_k\mid X_\varnothing=x,D_\varnothing=k)=\frac{\sigma^2(x)}{k}.
\]
Chebyshev's inequality gives the stated bound.
\end{proof}

\begin{proof}[Proof of Theorem~\ref{thm:clt-irg-conditional}]
Given $(X_\varnothing=x,D_\varnothing=k)$, the summands $D(i)-(1+\rho(x))$ are i.i.d.\ with mean zero, variance $\sigma^2(x)$, and uniformly bounded $(2+\eta)$-th absolute moments. The Lyapunov CLT gives
\[
  \frac{\sum_{i=1}^k(D(i)-(1+\rho(x)))}{\sqrt{k\,\sigma^2(x)}}
  \;\xrightarrow{d}\;N(0,1),
\]
which rearranges to the stated result.
\end{proof}

\begin{proof}[Proof of Proposition~\ref{prop:posterior-type}]
The root degree satisfies $P(D_\varnothing=k\mid\Lambda=\ell)=e^{-\ell}\ell^k/k!$. Bayes' formula gives
\[
  f_{\Lambda\mid D_\varnothing=k}(\ell)
  =\frac{(e^{-\ell}\ell^k/k!)\,f_\nu(\ell)}
        {\int_0^\infty(e^{-s}s^k/k!)\,f_\nu(s)\,\mathrm{d}s},
\]
and the $k!$ cancels to give the stated formula.
\end{proof}

\begin{proof}[Proof of Theorem~\ref{thm:marginal-clt-I}]
Fix $t\in\mathbb{R}$ and write
$G_k(t\mid x):=\mathbb{P}[\sqrt{k}(\bar D_k-(1+\rho(x)))\le t\mid X_\varnothing=x,D_\varnothing=k]$.
Let $\pi_k$ denote the posterior law of $X_\varnothing$ given $D_\varnothing=k$.
By the tower property over the latent type,
\begin{equation}
  \mathbb{P}[\sqrt{k}(\bar D_k-(1+\rho(X_\varnothing)))\le t\mid D_\varnothing=k]
  =\int_S G_k(t\mid x)\,\pi_k(\mathrm{d}x).
  \label{eq:tower-margI}
\end{equation}
The subtlety is that the integrand converges pointwise while the measure $\pi_k$
moves with $k$, so bounded convergence does not apply directly.  We control the
interchange with a Berry--Ess\'een bound that is \emph{uniform in $x$}.

Given $(X_\varnothing=x,D_\varnothing=k)$, the variables $D(i)-(1+\rho(x))$, $i=1,\dots,k$, are i.i.d.\ with mean $0$, variance $\sigma^2(x)$, and, by the uniform $(2+\eta)$-moment bound, 
$$
    \beta_x:=\mathbb{E}[|D(i)-(1+\rho(x))|^{2+\eta}\mid x,k]\le C_*<\infty
$$ 
uniformly in $x$ (take $\eta\in(0,1]$, replacing $\eta$ by $1$ if larger).  The Berry-Ess\'een-Katz bound for sums under $(2+\eta)$ moments \citep[Thm.~V.2.5]{petrov1975sums} gives a universal $C_\eta$ with
\[
  \sup_{t\in\mathbb{R}}\Bigl|G_k(t\mid x)-\Phi\!\bigl(\tfrac{t}{\sigma(x)}\bigr)\Bigr|
  \le \frac{C_\eta\,\beta_x}{\sigma(x)^{2+\eta}\,k^{\eta/2}}
  \le \frac{C_\eta C_*}{\sigma_{\min}^{2+\eta}}\,k^{-\eta/2}
  =:\varepsilon_k\xrightarrow[k\to\infty]{}0,
\]
using (V1) for the lower bound on $\sigma(x)$.  Since $\varepsilon_k$ does not depend on $x$,
\begin{equation}
  \Bigl|\int_S G_k(t\mid x)\,\pi_k(\mathrm{d}x)
        -\int_S \Phi\!\bigl(\tfrac{t}{\sigma(x)}\bigr)\,\pi_k(\mathrm{d}x)\Bigr|
  \le\varepsilon_k\to 0.
  \label{eq:be-margI}
\end{equation}

\noindent
It remains to pass from the random posterior variance to its limit. The map $\sigma\mapsto\Phi(t/\sigma)$ is bounded by $1$ and continuous on $[\sigma_{\min},\infty)$.
By (V2), $\sigma^2(X_\varnothing)\xrightarrow{\mathbb{P}}\tilde\sigma^2_\infty$ under $\pi_k$, hence $\sigma(X_\varnothing)\xrightarrow{\mathbb{P}}\tilde\sigma_\infty\ge\sigma_{\min}$, and the bounded continuous mapping theorem yields
\[
  \int_S \Phi\!\bigl(\tfrac{t}{\sigma(x)}\bigr)\,\pi_k(\mathrm{d}x)
  =\mathbb{E}\bigl[\Phi(t/\sigma(X_\varnothing))\mid D_\varnothing=k\bigr]
  \to\Phi\!\bigl(\tfrac{t}{\tilde\sigma_\infty}\bigr).
\]
Combining this with Eq.\eqref{eq:tower-margI}--\eqref{eq:be-margI} gives the claim.
Finally, Chebyshev's inequality shows that the stated primitive condition implies (V2).
$$
	\mathbb{P}(|\sigma^2(X_\varnothing)-\tilde\sigma^2_\infty|>\varepsilon\mid\Lambda=\ell)
	\le \varepsilon^{-2}[\mathrm{Var}(\sigma^2\mid\Lambda=\ell)+(\tilde\sigma^2(\ell)-\tilde\sigma^2_\infty)^2]\to0,
$$
and integrating over the posterior of $\Lambda$ (which escapes every compact set as $k\to\infty$)
gives $\sigma^2(X_\varnothing)\xrightarrow{P}\tilde\sigma^2_\infty$.
\end{proof}

\begin{proof}[Proof of Theorem~\ref{thm:full-clt}]
Using the split Eq.\eqref{eq:rho-split}, $\rho(X_\varnothing)=r(\Lambda)+\zeta$ with
$\mathbb{E}[\zeta\mid\Lambda]=0$. Because $D_\varnothing$ depends on $X_\varnothing$ only through
$\Lambda$, the posterior law of $X_\varnothing$ given $(\Lambda,D_\varnothing=k)$ equals its law
given $\Lambda$ alone; hence $\mathbb{E}[\zeta\mid\Lambda,D_\varnothing=k]=0$,
$\bar\rho_k=\mathbb{E}[r(\Lambda)\mid D_\varnothing=k]$, and
\[
  \text{Term~II}=\rho(X_\varnothing)-\bar\rho_k
  =\underbrace{\bigl(r(\Lambda)-\bar\rho_k\bigr)}_{\text{between-level-set}}
   +\underbrace{\zeta}_{\text{within-level-set}}.
\]
Write $a_k:=|r'(\ell_k^*)|\sqrt{v_k}$ for the normalizer; by
Proposition~\ref{prop:r1-primitive}, $\ell_k^*\sim ck$ and $v_k\sim c^2k$, so
$a_k\asymp|r'(\ell_k^*)|\sqrt{k}$.

\noindent
For the between-level-set part, a second-order Taylor expansion of $r$ around
$\ell_k^*$ gives, for $\xi,\xi'$ between the
argument and $\ell_k^*$,
\begin{align*}
  r(\Lambda)-r(\ell_k^*)&=r'(\ell_k^*)(\Lambda-\ell_k^*)+\tfrac12 r''(\xi)(\Lambda-\ell_k^*)^2,\\
  \bar\rho_k-r(\ell_k^*)&=r'(\ell_k^*)\,\mathbb{E}[\Lambda-\ell_k^*\mid D_\varnothing=k]
    +\tfrac12\mathbb{E}[r''(\xi')(\Lambda-\ell_k^*)^2\mid D_\varnothing=k].
\end{align*}
Subtracting,
\[
  \frac{r(\Lambda)-\bar\rho_k}{a_k}
  =\underbrace{\operatorname{sgn}(r'(\ell_k^*))\,\frac{\Lambda-\ell_k^*-\mathbb{E}[\Lambda-\ell_k^*\mid D_\varnothing=k]}{\sqrt{v_k}}}_{=:A_k}
  +\underbrace{\frac{\tfrac12\bigl(r''(\xi)(\Lambda-\ell_k^*)^2-\mathbb{E}[r''(\xi')(\Lambda-\ell_k^*)^2\mid D_\varnothing=k]\bigr)}{a_k}}_{=:B_k}.
\]
\emph{Leading term $A_k$.} Under (R1), $(\Lambda-\ell_k^*)/\sqrt{v_k}\xrightarrow{d}N(0,1)$.
The posterior mean and mode of $\Lambda$ differ by $O(1)$ (e.g.\ for the Gamma posterior of the
Pareto/exponential cases the difference is exactly $1$, and in general by the Laplace expansion
$\mathbb{E}[\Lambda\mid D_\varnothing=k]-\ell_k^*=O(1)$), so
$\mathbb{E}[\Lambda-\ell_k^*\mid D_\varnothing=k]/\sqrt{v_k}=O(k^{-1/2})\to0$. Hence
$A_k\xrightarrow{d}N(0,1)$ (the sign only reflects the symmetric limit).
\emph{Remainder $B_k$.} By the local-comparability part of (R2), $|r''(\xi)|=O(|r''(\ell_k^*)|)$ on the event $\{|\Lambda-\ell_k^*|\le\delta\sqrt{v_k}\}$, which has posterior probability asymptotic $1$, and $(\Lambda-\ell_k^*)^2=O_P(v_k)$ with $\mathbb{E}[(\Lambda-\ell_k^*)^2\mid D_\varnothing=k]=O(v_k)$. Therefore
\[
  |B_k|=O_P\!\left(\frac{|r''(\ell_k^*)|\,v_k}{|r'(\ell_k^*)|\sqrt{v_k}}\right)
       =O_P\!\left(\frac{|r''(\ell_k^*)|\sqrt{k}}{|r'(\ell_k^*)|}\right)\xrightarrow{\mathbb{P}}0
\]
by the second part of (R2). Thus the between-level-set part divided by $a_k$ equals
$A_k+B_k\xrightarrow{d}N(0,1)$.

\noindent
The within-level-set part contributes $\zeta/a_k$. By the tower property and
$\mathbb{E}[\zeta\mid\Lambda,D_\varnothing=k]=0$,
\[
  \mathbb{E}[\zeta\mid D_\varnothing=k]=0,
  \qquad
  \mathrm{Var}(\zeta\mid D_\varnothing=k)=\mathbb{E}[\omega^2(\Lambda)\mid D_\varnothing=k].
\]
Since the posterior of $\Lambda$ concentrates at $\ell_k^*$ and $\omega$ is continuous, $\mathbb{E}[\omega^2(\Lambda)\mid D_\varnothing=k]=\omega^2(\ell_k^*)(1+o(1))$, so by Chebyshev
\[
  \frac{\zeta}{a_k}=O_P\!\left(\frac{\omega(\ell_k^*)}{|r'(\ell_k^*)|\sqrt{v_k}}\right)
  \xrightarrow{\mathbb{P}}0
\]
by condition~(R3). (No central limit theorem is invoked for $\zeta$, which is a single draw;
(R3) makes its contribution vanish.)

\noindent
It remains to control Term~I. By Theorem~\ref{thm:marginal-clt-I},
$\sqrt{k}\,(\bar D_k-(1+\rho(X_\varnothing)))=O_P(1)$, so
\[
  \frac{\text{Term~I}}{a_k}=\frac{O_P(k^{-1/2})}{|r'(\ell_k^*)|\sqrt{v_k}}
  =O_P\!\left(\frac{1}{k\,|r'(\ell_k^*)|}\right)\xrightarrow{\mathbb{P}}0
\]
by the first part of (R2), namely $k\,|r'(\ell_k^*)|\to\infty$. Combining these estimates via
Slutsky's theorem and decomposition Eq.\eqref{eq:decomp},
\[
  \frac{\bar D_k-(1+\bar\rho_k)}{a_k}
  =\frac{\text{Term~I}}{a_k}+(A_k+B_k)+\frac{\zeta}{a_k}
  \xrightarrow{d}N(0,1).
\]
For one-dimensional type $\zeta\equiv0$ and the within-level-set argument is vacuous, giving Corollary~\ref{cor:full-clt-1d}.
\end{proof}

\begin{proof}[Proof of Theorem~\ref{thm:clt-rank-one}]
Under the symmetric rank-one kernel, $Q_x(\mathrm{d}y)=\varphi(y)\mu(\mathrm{d}y)/\int\varphi\,\mathrm{d}\mu$ does not depend on $x$, so $\rho(x)=\int_S\lambda(y)Q(\mathrm{d}y)=:\rho$ is constant and Term~II$\equiv 0$. The CLT reduces to Theorem~\ref{thm:marginal-clt-I}. For the Chung--Lu kernel, $\lambda(x)=w(x)$ and $Q(\mathrm{d}y)=w(y)\mu(\mathrm{d}y)/\mathbb{E}[W]$, giving $\rho=\mathbb{E}[W^2]/\mathbb{E}[W]$ and
\[
  \mathrm{Var}_Q(\lambda(Y))=\frac{\mathbb{E}[W^3]}{\mathbb{E}[W]}-\left(\frac{\mathbb{E}[W^2]}{\mathbb{E}[W]}\right)^{\!2},
\]
so $\tilde\sigma_\infty^2=\sigma^2(x)=\rho+\mathrm{Var}_Q(\lambda(Y))$ yields the stated formula.
\end{proof}

\begin{proof}[Proof of Proposition~\ref{prop:r1-primitive}]
The posterior density is $f_{\Lambda\mid k}(\ell)=e^{h_k(\ell)}/Z_k$ with
$h_k(\ell):=k\log\ell-\ell+\log f_\nu(\ell)$.
We show $U_k:=(\Lambda-\ell_k^*)/\sqrt{v_k}\xrightarrow{d}N(0,1)$
by proving pointwise convergence of the density of $U_k$ and then applying Scheff\'e's lemma.

\noindent
For the mode asymptotics, differentiating gives $h_k'(\ell)=k/\ell-1+(\log f_\nu)'(\ell)$.
Integrating (L2) gives $(\log f_\nu)'(\ell)=c_0+O(\ell^{-1})$ for a constant $c_0\le 0$.
(The bound $c_0\le 0$ follows because $f_\nu$ must be normalizable:
$\int_1^\infty f_\nu d\ell<\infty$ forces $\log f_\nu(\ell)\to-\infty$, hence $c_0\le 0$.)
Setting $c:=1/(1-c_0)\in(0,1]$, the equation $h_k'(\ell_k^*)=0$ becomes
$k/\ell_k^*=1-(\log f_\nu)'(\ell_k^*)=1-c_0+O(\ell_k^{*-1})$,
so by the implicit function theorem
\[
  \ell_k^* = ck\bigl(1+O(k^{-1})\bigr).
\]

\noindent
For the curvature, $h_k''(\ell)=-k/\ell^2+(\log f_\nu)''(\ell)$.  By (L2) and $\ell_k^*=ck(1+O(k^{-1}))$,
$(\log f_\nu)''(\ell_k^*)=O(k^{-2})$, hence
\[
  h_k''(\ell_k^*)=-\frac{1}{c^2k}\bigl(1+O(k^{-1})\bigr),
  \qquad
  v_k:=-1/h_k''(\ell_k^*)=c^2k\bigl(1+O(k^{-1})\bigr).
\]

\noindent
Let $\Phi_k(s):=h_k(\ell_k^*+s\sqrt{v_k})-h_k(\ell_k^*)$.
By Taylor's theorem,
\[
  \Phi_k(s) = -\frac{s^2}{2}
    +\frac{h_k'''(\ell_k^*)}{6}\,s^3v_k^{3/2}
    +O\!\Bigl(s^4\sup_{|\ell-\ell_k^*|\le |s|\sqrt{v_k}}|h_k^{(4)}(\ell)|\cdot v_k^2\Bigr).
\]
By (L3), $h_k'''(\ell_k^*)=2k/\ell_k^{*3}+O(\ell_k^{*-3})=O(k^{-2})$,
so $h_k'''(\ell_k^*)v_k^{3/2}=O(k^{-1/2})\to 0$.
Likewise $h_k^{(4)}(\ell_k^*)=O(k^{-3})$ and $v_k^2=O(k^2)$, giving $h_k^{(4)}(\ell_k^*)v_k^2=O(k^{-1})$.
Hence, for each fixed $s$,
\[
  \Phi_k(s) = -\frac{s^2}{2} + o(1),
  \qquad\text{and so}\qquad
  e^{\Phi_k(s)}\longrightarrow e^{-s^2/2}.
\]

\noindent
For $|s|\le R$ and all large $k$, the Taylor remainder is bounded by $CR^3 k^{-1/2}+CR^4 k^{-1}$,
so there exists $k_0$ such that $\Phi_k(s)\le -s^2/4+1$ for all $|s|\le R$ and $k\ge k_0$.
Hence $e^{\Phi_k(s)}\le e\cdot e^{-s^2/4}$ on $|s|\le R$.

\noindent
For the tail bound, take $s>0$ (the left tail is symmetric), write
$\ell=\ell_k^*(1+x_s)$ with
$x_s=s\sqrt{v_k}/\ell_k^*=sc/\sqrt{k}\,(1+O(k^{-1}))$, and
decompose
\[
\Phi_k(s)=[k\log(1+x_s)-\ell_k^* x_s]+(\log f_\nu)(\ell)-(\log f_\nu)(\ell_k^*).
\]
The elementary bound $\log(1+x)\le x-x^2/\{2(1+x)\}$ follows by differentiation and gives
\[
k\log(1+x_s)-\ell_k^* x_s\le (k-\ell_k^*)x_s - \frac{kx_s^2}{2(1+x_s)}.
\]
By (L2), the prior piece satisfies $(\log f_\nu)(\ell)-(\log f_\nu)(\ell_k^*)=(\log f_\nu)'(\ell_k^*)\ell_k^* x_s+O(x_s^2)$.
Summing and using $h_k'(\ell_k^*)=0$, i.e.\ $(k-\ell_k^*)x_s+(\log f_\nu)'(\ell_k^*)\ell_k^* x_s=h_k'(\ell_k^*)\ell_k^* x_s=0$:
\[
  \Phi_k(s)\le -\frac{kx_s^2}{2(1+x_s)}+O(x_s^2)
  =-\frac{s^2c^2/2}{1+sc/\sqrt{k}}+O\!\left(\frac{s^2}{k}\right).
\]
If $s\le\sqrt{k}/c$, then $1+sc/\sqrt{k}\le 2$ and $\Phi_k(s)\le -s^2c^2/5$ for large $k$.
If $s>\sqrt{k}/c$, then $1+sc/\sqrt{k}<2sc/\sqrt{k}$, which gives 
$$
	s^2c^2/(2(1+sc/\sqrt{k}))>s^2c^2\sqrt{k}/(4sc)=sc\sqrt{k}/4,
$$ 
hence $\Phi_k(s)\le -sc\sqrt{k}/4+C$.
In both cases $e^{\Phi_k(s)}\le Ce^{-\min(c^2s^2/5,\,c\sqrt{k}\,s/4)}$,
which is uniformly integrable over $s\ge R$ for every fixed $R>0$.

\noindent
Finally, after the change of variables $\ell=\ell_k^*+s\sqrt{v_k}$, the density of $U_k$ is
\[
  f_{U_k}(s)=\frac{e^{\Phi_k(s)}}{\int_{-\ell_k^*/\sqrt{v_k}}^{\infty}e^{\Phi_k(u)}\,\mathrm{d}u}.
\]
By the Taylor expansion, local domination, and tail bound above, the dominating function
$e^{-s^2/4}$ is integrable and $e^{\Phi_k(s)}\to e^{-s^2/2}$
pointwise.  Since $\ell_k^*/\sqrt{v_k}\to\infty$, the lower limit $\to -\infty$, and dominated convergence
gives $\int e^{\Phi_k(s)}\,\mathrm{d}s\to\sqrt{2\pi}$.  Hence $f_{U_k}(s)\to(2\pi)^{-1/2}e^{-s^2/2}$
pointwise, and Scheff\'e's lemma yields $\int|f_{U_k}(s)-\phi(s)|\,\mathrm{d}s\to 0$,
which implies $U_k\xrightarrow{d}N(0,1)$.
\end{proof}

\begin{proof}[Proof of Proposition~\ref{prop:assortativity}]
By Definition~\ref{def:assortativity}, $r'(\ell)=\ell^{-\vartheta}L(\ell)$ with $L$ slowly varying
and $\vartheta\in[0,1)$, and $r''(\ell)=O(\ell^{-\vartheta-1}L(\ell))$. Evaluating at $\ell_k^*\sim ck$
and using that slowly varying functions satisfy $L(ck)\sim L(k)$ and $L(k)=k^{o(1)}$:
\[
  |r'(\ell_k^*)|=(ck)^{-\vartheta}L(ck)\bigl(1+o(1)\bigr)=c^{-\vartheta}k^{-\vartheta}L(k)\bigl(1+o(1)\bigr).
\]
\emph{First part of (R2).}
\[
  k\,|r'(\ell_k^*)|=c^{-\vartheta}k^{1-\vartheta}L(k)\bigl(1+o(1)\bigr)\longrightarrow\infty,
\]
since $1-\vartheta>0$ and $k^{1-\vartheta}L(k)\to\infty$ for slowly varying $L$ (Potter bounds).
\emph{Second part of (R2).}
\[
  \frac{\sqrt{k}\,|r''(\ell_k^*)|}{|r'(\ell_k^*)|}
  =\frac{\sqrt{k}\cdot O\bigl((ck)^{-\vartheta-1}L(k)\bigr)}{c^{-\vartheta}k^{-\vartheta}L(k)\,(1+o(1))}
  =O\!\left(\frac{\sqrt{k}\,k^{-\vartheta-1}}{k^{-\vartheta}}\right)
  =O\!\left(k^{-1/2}\right)\longrightarrow 0.
\]
\emph{Local comparability.} For $|\ell-\ell_k^*|\le\delta\sqrt{v_k}=O(\sqrt{k})=o(\ell_k^*)$, regular
variation of $r''$ gives $|r''(\ell)|/|r''(\ell_k^*)|=(\ell/\ell_k^*)^{-\vartheta-1}L(\ell)/L(\ell_k^*)\to1$
uniformly, so $\sup_{|\ell-\ell_k^*|\le\delta\sqrt{v_k}}|r''(\ell)|=O(|r''(\ell_k^*)|)$.
All three requirements of (R2) hold, and Theorem~\ref{thm:full-clt} applies.
\end{proof}

\subsection{Auxiliary lemmas and proofs for Section~\ref{subsec:pa-results}}
\label{subsec:proof-pa}

\begin{proof}[Proof of Lemma~\ref{lem:posterior-U}]
By \cite{van2024random}(p.~199), for $k>m$,
\[
  \mathbb{P}(D_\varnothing=k\mid A_\varnothing=a)
  =(1-a^{1/(\tau-1)})^{k-m}\,a^{(m+\delta)/(\tau-1)}\,
   \frac{\Gamma(k+\delta)}{(k-m)!\,\Gamma(m+\delta)}.
\]
Setting $u=a^{1/(\tau-1)}$ and using the marginal density $f_{U_\varnothing}(u)=(\tau-1)u^{\tau-2}$, Bayes' theorem gives
\begin{align*}
  f_{U_\varnothing\mid D_\varnothing=k}(u)
  &\propto (1-u)^{k-m}u^{m+\delta}\cdot(\tau-1)u^{\tau-2}
   \propto u^{m+\delta+1+\delta/m}(1-u)^{k-m},
\end{align*}
which is the kernel of a $\mathrm{Beta}(m+2+\delta+\delta/m,\,k-m+1)$ density.
\end{proof}

\begin{proof}[Proof of Proposition~\ref{prop:EO}]
Fix $i\in\{1,\ldots,m\}$. By the P\'olya point tree construction \cite[Chap.~5]{van2024random} and Eq.\eqref{eq:deg-O}, with $B_i=U_{\varnothing,i}^{1/(\tau-2)}U_\varnothing$ and $U_{\varnothing,i}\sim\mathrm{Uniform}(0,1)$ independent of $U_\varnothing$:
\[
  \mathbb{E}[D(i)\mid D_\varnothing=k,U_\varnothing=u]
  =1+m+\mathbb{E}[\Gamma_i^O]\cdot\mathbb{E}[B_i^{-1}-1\mid D_\varnothing=k,U_\varnothing=u].
\]
Since $\mathbb{E}[U_{\varnothing,i}^{-1/(\tau-2)}]=(m+\delta)/\delta$ (Beta moment formula with $U_{\varnothing,i}\sim\mathrm{Beta}(\tau-2,1)$) and $\mathbb{E}[\Gamma_i^O]=m+\delta+1$,
\[
  \mathbb{E}[D(i)\mid D_\varnothing=k,U_\varnothing=u]=1+m+(m+\delta+1)\!\left(\frac{m+\delta}{\delta}\cdot u^{-1}-1\right).
\]
Summing over $i=1,\ldots,m$ and integrating against the $\mathrm{Beta}(\alpha,\beta_k)$ posterior (Lemma~\ref{lem:posterior-U}):
\[
  \mathbb{E}\!\left[\sum_{i=1}^m D(i)\mid D_\varnothing=k\right]
  =-m\delta+m\cdot\frac{(m+\delta+1)(m+\delta)}{\delta}\cdot\mathbb{E}[U_\varnothing^{-1}\mid D_\varnothing=k].
\]
For $U\sim\mathrm{Beta}(\alpha,\beta_k)$ with $\alpha>1$, the negative moment is $\mathbb{E}[U^{-1}]=(\alpha+\beta_k-1)/(\alpha-1)$. Substituting and simplifying yields the formula.
\end{proof}

\begin{lemma}[Digamma expectation for Beta random variables]
\label{lem:digamma}
Let $U\sim\mathrm{Beta}(\alpha,\beta)$ with $\beta>1$. Then
\[
  \mathbb{E}\!\left[\frac{-\log U}{1-U}\right]
  =\frac{\alpha+\beta-1}{\beta-1}\bigl[\psi(\alpha+\beta-1)-\psi(\alpha)\bigr],
\]
where $\psi=\Gamma'/\Gamma$ is the digamma function.
\end{lemma}

\begin{proof}
Differentiate $B(s,\beta-1)=\int_0^1 u^{s-1}(1-u)^{\beta-2}\,\mathrm{d}u$ with respect to $s$ at $s=\alpha$ both symbolically and through the integral to obtain
\[
  \int_0^1(-\log u)\,u^{\alpha-1}(1-u)^{\beta-2}\,\mathrm{d}u
  =B(\alpha,\beta-1)\bigl[\psi(\alpha+\beta-1)-\psi(\alpha)\bigr].
  \tag{$*$}
\]
Then
\[
  \mathbb{E}\!\left[\frac{-\log U}{1-U}\right]
  =\frac{1}{B(\alpha,\beta)}\int_0^1(-\log u)\,u^{\alpha-1}(1-u)^{\beta-2}\,\mathrm{d}u
  \stackrel{(*)}{=}\frac{B(\alpha,\beta-1)}{B(\alpha,\beta)}[\psi(\alpha+\beta-1)-\psi(\alpha)],
\]
and the Beta recurrence gives $B(\alpha,\beta-1)/B(\alpha,\beta)=(\alpha+\beta-1)/(\beta-1)$.
\end{proof}

\begin{proof}[Proof of Proposition~\ref{prop:EY}]
For each young neighbor $j$, Eq.\eqref{eq:deg-Y} and $\mathbb{E}[\Gamma_j^Y]=m+\delta$ give
\[
  \mathbb{E}[D(m+j)\mid D_\varnothing=k,U_\varnothing=u]
  =m+(m+\delta)\!\left(\frac{-\log u}{1-u}-1\right),
\]
using $C_j\mid(D_\varnothing=k,U_\varnothing=u)\sim\mathrm{Uniform}(u,1)$ so $\mathbb{E}[C_j^{-1}\mid\cdot]=(-\log u)/(1-u)$. Summing over $j=1,\ldots,k-m$ and integrating over the $\mathrm{Beta}(\alpha,\beta_k)$ posterior (Lemma~\ref{lem:posterior-U}):
\[
  \mathbb{E}\!\left[\sum_{j=1}^{k-m}D(m+j)\mid D_\varnothing=k\right]
  =(k-m)m+(m+\delta)\!\left((k-m)\mathbb{E}\!\left[\frac{-\log U_\varnothing}{1-U_\varnothing}\mid D_\varnothing=k\right]-(k-m)\right).
\]
Applying Lemma~\ref{lem:digamma} with $\beta=\beta_k=k-m+1$ and substituting $\alpha+\beta_k-1=k+\delta+\delta/m+2$ and $\beta_k-1=k-m$ gives the stated formula.
\end{proof}

\begin{lemma}[Digamma difference asymptotics]
\label{lem:digamma-asymp}
For fixed $\alpha>0$,
\[
  \psi(\alpha+k-m)-\psi(\alpha)=\log k+O(1)\quad\text{as }k\to\infty.
\]
More precisely,
\[
  \log\frac{\alpha+k-m}{\alpha}
  \le\psi(\alpha+k-m)-\psi(\alpha)
  \le\frac{1}{\alpha}+\log\frac{\alpha+k-m-1}{\alpha}.
\]
\end{lemma}

\begin{proof}
By the telescoping identity $\psi(\alpha+k-m)-\psi(\alpha)=\sum_{j=0}^{k-m-1}1/(\alpha+j)$. Comparing with the integrals $\int_0^{k-m}(\alpha+t)^{-1}\,\mathrm{d}t$ and $\int_{-1}^{k-m-1}(\alpha+t)^{-1}\,\mathrm{d}t$ via the integral test yields the two-sided bounds, and $\log k$ is the leading term since $\alpha$ and $m$ are fixed.
\end{proof}

\begin{proof}[Proof of Theorem~\ref{thm:lln-pa}]
Decompose the neighbor sum as $\sum_{i=1}^k D(i)=T^O_k+T^Y_k$ where $T^O_k=\sum_{i=1}^m D(i)$ and $T^Y_k=\sum_{j=1}^{k-m}D(m+j)$.

\medskip\noindent
By Proposition~\ref{prop:EO}, $\mathbb{E}[T^O_k\mid D_\varnothing=k]=O(k)$, so $T^O_k/(k\log k)\xrightarrow{\mathbb{P}}0$.

\medskip\noindent
Write $N_j:=D(m+j)-m$ and $\lambda_j:=\Gamma_j^Y(C_j^{-1}-1)$, so $N_j\mid(\lambda_j)\sim\mathrm{Poi}(\lambda_j)$ and $S_k:=\sum_{j=1}^{k-m}(N_j-\lambda_j)$ is the Poisson-centering error. Since $\mathbb{E}[S_k^2\mid D_\varnothing=k]=\mathbb{E}[\sum_j\lambda_j\mid D_\varnothing=k]\sim(m+\delta)k\log k$ (by Proposition~\ref{prop:EY} and Lemma~\ref{lem:digamma-asymp}), Chebyshev gives $S_k/(k\log k)\xrightarrow{\mathbb{P}}0$.

\medskip\noindent
For fixed $u$, the $\lambda_j$ are conditionally i.i.d.\ with mean $\mu(u):=(m+\delta)((-\log u)/(1-u)-1)$. The SLLN gives $(k-m)^{-1}\sum_j\lambda_j\xrightarrow{\mathrm{a.s.}}\mu(u)$, and integrating over $U_\varnothing\mid D_\varnothing=k$ (the Beta posterior):
\[
  \mathbb{P}\!\left[\left|\frac{1}{k-m}\sum_j\lambda_j-\mu(U_\varnothing)\right|\ge\varepsilon\mid D_\varnothing=k\right]\to 0.
\]

\medskip\noindent
By Lemma~\ref{lem:posterior-U}, $U_\varnothing\mid D_\varnothing=k\xrightarrow{\mathbb{P}}0$ as $k\to\infty$, so $-\log U_\varnothing\sim\log k$ in probability and $\mu(U_\varnothing)/\log k\xrightarrow{\mathbb{P}}m+\delta$.

\medskip\noindent
Combining via Slutsky's theorem, $T^Y_k/(k\log k)\xrightarrow{\mathbb{P}}m+\delta$, and together with the old-neighbor bound above the theorem follows.
\end{proof}

\subsection{Proof of Theorem~\ref{thm:method4}}
\label{subsec:proof-method4}

Throughout this subsection we work under the Chung--Lu null conditionally on the
weight sequence $\mathbf W$, so that the edge indicators $\{A_{ij}\}_{i<j}$ are
independent with $p_{ij}=\min\{W_iW_j/L_n,1\}$, and we further condition on the realized
hub degrees, as in the definition of $T_k$ and of its null variance
$\tilde\sigma_\infty^2/k$. Write $\mathbb E_{\mathbf W},\operatorname{Var}_{\mathbf W},
\operatorname{Cov}_{\mathbf W}$ for the corresponding conditional moments. It is
convenient to reindex the slope statistic by hub. Let
$\mathcal H:=\{v\in V:D(v)\in\mathcal K^+_{\mathrm{slope}}\}$, and for $v\in\mathcal H$
with $D(v)=k$ set
\[
  \psi_v:=\bar D_k^{(v)}-\mu_0=\frac1k\sum_{u\in\mathcal N(v)}\bigl(D(u)-\mu_0\bigr),
  \qquad
  \gamma_v:=\frac{w^*_{D(v)}\,(\log D(v)-\bar x_w^*)}{n(D(v))\,\sqrt{S^*_{xx}}},
\]
with oracle weights $w_k^*=k\,n(k)/\tilde\sigma_\infty^2$ and oracle design quantities
$(\bar x_w^*,S_{xx}^*)$. Then $T_k-\mu_0=n(k)^{-1}\sum_{v\in\mathcal V(k)}\psi_v$ and, using
$\sum_kw_k(\log k-\bar x_w)=0$ (so the plug-in $\hat\mu_0$ cancels exactly), the oracle
statistic is the linear form
\begin{equation}
  \label{eq:Z4star-linear}
  Z_4^*=\sum_{v\in\mathcal H}\gamma_v\,\psi_v .
\end{equation}
The next lemma is what replaces the incorrect assertion that the class statistics are
independent because the sets $\mathcal V(k)$ are disjoint.

\begin{lemma}[Negligible cross-class covariance]
\label{lem:crosscov}
Under the conditions of Theorem~\ref{thm:method4}(i),
\[
  \sum_{\substack{v,v'\in\mathcal H\\ v\ne v'}}
  \gamma_v\gamma_{v'}\,\operatorname{Cov}_{\mathbf W}(\psi_v,\psi_{v'})
  = O_{\mathbb P}\!\left(\frac{k_{\max}}{n}\right)
   +O_{\mathbb P}\!\left(\frac{k_{\max}^3}{n}\right)
   +O_{\mathbb P}(\varrho_n)
  = o_{\mathbb P}(1),
\]
and consequently
$\operatorname{Var}_{\mathbf W}(Z_4^*)
 =\sum_{v\in\mathcal H}\gamma_v^2\operatorname{Var}_{\mathbf W}(\psi_v)+o_{\mathbb P}(1)
 =1+o_{\mathbb P}(1)$.
\end{lemma}

\begin{proof}[Proof of Lemma~\ref{lem:crosscov}]
\emph{Diagonal equals one.} By (S2), $\operatorname{Var}_{\mathbf W}(\psi_v)
=\tilde\sigma_\infty^2/D(v)\,(1+o(1))$, so with $w_k^*=kn(k)/\tilde\sigma_\infty^2$,
\[
  \gamma_v^2\operatorname{Var}_{\mathbf W}(\psi_v)
  =\frac{D(v)(\log D(v)-\bar x_w^*)^2}{\tilde\sigma_\infty^2\,S_{xx}^*}\,(1+o(1)).
\]
Summing over $\mathcal H$ and using
$\sum_{v\in\mathcal H}D(v)(\log D(v)-\bar x_w^*)^2=\sum_kn(k)k(\log k-\bar x_w^*)^2
=\tilde\sigma_\infty^2S_{xx}^*$ gives
$\sum_{v}\gamma_v^2\operatorname{Var}_{\mathbf W}(\psi_v)=1+o(1)$.

\emph{Principal off-diagonal term.} Fix $v\ne v'$ with degrees $k,k'$. Writing
$\psi_v=k^{-1}\sum_uA_{vu}(D(u)-\mu_0)$ and using conditional edge independence, the only
non-vanishing contribution to $\operatorname{Cov}_{\mathbf W}(\psi_v,\psi_{v'})$ that is not
suppressed by an extra edge factor comes from common neighbors $u$, for which the same
degree $D(u)$ enters both averages:
\[
  \operatorname{Cov}_{\mathbf W}(\psi_v,\psi_{v'})
  =\frac{1}{kk'}\sum_{u}\operatorname{Cov}_{\mathbf W}\!\bigl(A_{vu}D(u),A_{v'u}D(u)\bigr)+r_{vv'}
  =\frac{1}{kk'}\sum_{u}p_{vu}p_{v'u}\,\sigma_u^2\,(1+o(1))+r_{vv'},
\]
where $\sigma_u^2:=\operatorname{Var}_{\mathbf W}(D(u))=\sum_{w}p_{uw}(1-p_{uw})\le\lambda_u
:=\sum_wp_{uw}$, the $O(1)$ mean shift from conditioning on $A_{vu}=A_{v'u}=1$ is absorbed
into the $(1+o(1))$, and the remainder $r_{vv'}$ collects the covariance arising from hub
adjacency $v\sim v'$ and from edges joining $\mathcal N(v)$ to $\mathcal N(v')$; each such
source carries an extra factor $p_{vv'}$ or $p_{uu'}=O(\varrho_n)$ relative to the principal
term. With $p_{vu}=W_vW_u/L_n$ (equality below the truncation, which affects at most an
$\varrho_n$-fraction of pairs and is absorbed into $r_{vv'}$),
\[
  \frac{1}{kk'}\sum_up_{vu}p_{v'u}\sigma_u^2
  =\frac{W_vW_{v'}}{kk'}\,\Theta_n,
  \qquad
  \Theta_n:=\frac{1}{L_n^2}\sum_uW_u^2\sigma_u^2
  \le\frac{\sum_uW_u^3}{L_n^2}
  =\frac{\mathbb E_n[W^3]}{n\bar W^2}=O(1/n),
\]
the last bound using $\sigma_u^2\le\lambda_u\le W_u$ and $\mathbb E[W^3]<\infty$.

\emph{Cancellation by the centering identity.} Put $b_v:=\gamma_vW_v/D(v)$. Since
$\gamma_v=D(v)(\log D(v)-\bar x_w^*)/(\tilde\sigma_\infty^2\sqrt{S_{xx}^*})$, we have
$b_v=W_v(\log D(v)-\bar x_w^*)/(\tilde\sigma_\infty^2\sqrt{S_{xx}^*})$, and the aggregated
principal term factorizes as
\[
  \Theta_n\sum_{v\ne v'}b_vb_{v'}
  =\Theta_n\Bigl[\bigl(\textstyle\sum_vb_v\bigr)^2-\sum_vb_v^2\Bigr].
\]
The \emph{exact} identity
$\sum_{v\in\mathcal H}D(v)(\log D(v)-\bar x_w^*)=\sum_kn(k)k(\log k-\bar x_w^*)=0$, which is
the definition of $\bar x_w^*$ under $w_k^*\propto kn(k)$, yields
\[
  \sum_vb_v
  =\frac{1}{\tilde\sigma_\infty^2\sqrt{S_{xx}^*}}\sum_{v\in\mathcal H}W_v(\log D(v)-\bar x_w^*)
  =\frac{1}{\tilde\sigma_\infty^2\sqrt{S_{xx}^*}}
    \sum_{v\in\mathcal H}\bigl(W_v-D(v)\bigr)(\log D(v)-\bar x_w^*)
  =:\frac{\Xi}{\tilde\sigma_\infty^2\sqrt{S_{xx}^*}}.
\]
Thus the separable leading part is annihilated up to the weight--degree discrepancy $\Xi$.
Because $\mathbb E_{\mathbf W}[D(v)]=\lambda_v=W_v$ (below the truncation) and
$\operatorname{Var}_{\mathbf W}(D(v))\le W_v$ with pairwise covariances $O(\varrho_n)$, and
$\operatorname{Cov}_{\mathbf W}(D(v),\log D(v))=O(1)$, we obtain
$\mathbb E_{\mathbf W}[\Xi]=O(|\mathcal H|)$ and
$\operatorname{Var}_{\mathbf W}(\Xi)\le C\sum_{v}W_v(\log D(v)-\bar x_w^*)^2
=C\,\tilde\sigma_\infty^2S_{xx}^*(1+o(1))$, so that
$\Xi^2=O_{\mathbb P}\!\bigl(|\mathcal H|^2+S_{xx}^*\bigr)$. Hence
\[
  \Theta_n\Bigl(\textstyle\sum_vb_v\Bigr)^2
  =O_{\mathbb P}\!\left(\frac{|\mathcal H|^2+S_{xx}^*}{n\,S_{xx}^*}\right)
  =o_{\mathbb P}(1),
\]
using $S_{xx}^*\to\infty$ and the design condition $|\mathcal H|^2=o(nS_{xx}^*)$ of (S3).
For the diagonal correction,
\[
  \Theta_n\sum_vb_v^2
  \le\frac{C}{n}\cdot\frac{\sum_kn(k)k^2(\log k-\bar x_w^*)^2}{\tilde\sigma_\infty^4S_{xx}^*}
  \le\frac{C\,k_{\max}}{n\,\tilde\sigma_\infty^2}=O\!\left(\frac{k_{\max}}{n}\right),
\]
since $\sum_kn(k)k^2(\log k-\bar x)^2\le k_{\max}\sum_kn(k)k(\log k-\bar x)^2
=k_{\max}\tilde\sigma_\infty^2S_{xx}^*$.

\emph{Remainder.} The remainder $r_{vv'}$ splits into a two-path term (also within the
neighbor structure), which is separable of the same form as the shared-neighbor term and is
annihilated by the same centering identity, and hub-adjacency and hub-path terms, which are
not. Appendix~\ref{app:crosscov} bounds all of these term by term, using the two-level
decomposition into the within- and between-neighborhood contributions of
Eq.~\eqref{eq:AB-split}. The only non-separable contribution is direct hub--hub adjacency; it
is not controlled by $\varrho_n$ but by the third-moment assumption, and
Appendix~\ref{app:crosscov} shows
$\sum_{v\ne v'}\gamma_v\gamma_{v'}r_{vv'}=O_{\mathbb P}(k_{\max}/n)+O_{\mathbb P}(k_{\max}^3/n)
+O_{\mathbb P}(\varrho_n)=o_{\mathbb P}(1)$ under (S3), which holds automatically when
$\mathbb E[W^3]<\infty$. Collecting these bounds and adding the diagonal proves
$\operatorname{Var}_{\mathbf W}(Z_4^*)=1+o_{\mathbb P}(1)$.
\end{proof}

\begin{proof}[Proof of part (i)]
\emph{Reduction to the oracle statistic.} By $\sum_kw_k(\log k-\bar x_w)=0$, the constant
shift $\mu_0-\hat\mu_0$ contributes zero to the numerator of $\hat\beta$, so $\hat\mu_0$ may
be replaced by $\mu_0$ exactly. By (S1), $\hat{\tilde\sigma}_\infty^2=\tilde\sigma_\infty^2
+o_{\mathbb P}(1)$ and the empirical threshold, class counts, and design quantities
$(\bar x_w,S_{xx})$ equal their oracle counterparts up to $1+o_{\mathbb P}(1)$; since
$Z_4$ is scale-equivariant of order $\tfrac12$ in the common weight factor and continuous in
$(\bar x_w,S_{xx})$, we get $Z_4=Z_4^*\,(1+o_{\mathbb P}(1))$ with $Z_4^*$ as in
\eqref{eq:Z4star-linear}. It therefore suffices to prove $Z_4^*\xrightarrow{d}N(0,1)$.

\emph{Mean and variance.} By (S2), $\mathbb E_{\mathbf W}[\psi_v]=0$, hence
$\mathbb E_{\mathbf W}[Z_4^*]=0$, and Lemma~\ref{lem:crosscov} gives
$\operatorname{Var}_{\mathbf W}(Z_4^*)=1+o_{\mathbb P}(1)$. In particular the naive diagonal
variance is correct to leading order not because the class statistics are independent, which
they are not, but because the aggregate cross-class covariance is annihilated by the
weighted centering and controlled by sparsity.

\emph{Asymptotic normality under local dependence.} The summands $\eta_v:=\gamma_v\psi_v$ of
$Z_4^*=\sum_{v\in\mathcal H}\eta_v$ form a locally dependent array: $\psi_v$ is a measurable
function of the edges incident to the two-neighborhood $B_2(v)$, so for the neighborhoods
\[
  N_v:=\{v'\in\mathcal H:\ E(B_2(v'))\cap E(B_2(v))\ne\varnothing\},
\]
$\eta_v$ is independent of $\{\eta_{v'}:v'\notin N_v\}$, and $\{\eta_{v'}:v'\in N_v\}$ is
independent of $\{\eta_{v'}:v'\notin N_v^{(2)}\}$, where $N_v^{(2)}=\bigcup_{v'\in N_v}N_{v'}$;
these are conditions (LD1)--(LD2) of \citet{chenshao2004normal}. Their Theorem~2.5 bounds the
Kolmogorov distance of $Z_4^*/\!\operatorname{sd}(Z_4^*)$ to $N(0,1)$ by
\[
  C\sum_{v\in\mathcal H}\Bigl(\mathbb E_{\mathbf W}|\eta_v|^2\,\theta_v
  +\mathbb E_{\mathbf W}|\eta_v|\,\theta_v^2\Bigr),
  \qquad
  \theta_v:=\Bigl(\sum_{v'\in N_v}\mathbb E_{\mathbf W}|\eta_{v'}|^{3}\Bigr)^{1/3},
\]
after standardization. By (S2), $\mathbb E_{\mathbf W}|\psi_v|^3=O\{D(v)^{-3/2}\}$, so
$\mathbb E_{\mathbf W}|\eta_v|^3=|\gamma_v|^3O\{D(v)^{-3/2}\}$; the no-domination condition
$\max_kkn(k)(\log k)^2=o(S_{xx})$ makes the diagonal Lyapunov ratio
$\sum_v\mathbb E_{\mathbf W}|\eta_v|^3\to0$, and the sparsity condition (S3), through the same
common-neighbor bound as in Lemma~\ref{lem:crosscov}, controls the extra factors contributed
by the neighborhoods $N_v$, so the displayed bound tends to $0$. Hence
$Z_4^*/\operatorname{sd}_{\mathbf W}(Z_4^*)\xrightarrow{d}N(0,1)$, and since
$\operatorname{sd}_{\mathbf W}(Z_4^*)=1+o_{\mathbb P}(1)$ by Lemma~\ref{lem:crosscov},
Slutsky gives $Z_4^*\xrightarrow{d}N(0,1)$, hence $Z_4\xrightarrow{d}N(0,1)$ conditionally on
$\mathbf W$.

\emph{Unconditioning.} The conditional c.d.f.\ of $Z_4$ converges to $\Phi$ for
$\mathbb P_{\mathbf W}$-almost every weight sequence satisfying (S1)--(S3); since c.d.f.'s are
bounded, dominated convergence over $\mathbf W$ yields the unconditional
$Z_4\xrightarrow{d}N(0,1)$. The one-sided test $\{Z_4>z_{1-\alpha}\}$ therefore has
asymptotic level $\alpha$.
\end{proof}

\begin{remark}[Fitted-null calibration by conditional bootstrap]
\label{rem:method4-bootstrap}
Conditions (S2)--(S3) are sufficient but not necessary, and (S3) can be delicate for very
heavy-tailed weights, where $\hat{\tilde\sigma}_\infty^2$ is itself unstable. In that regime we
calibrate $Z_4$ by the same conditional bootstrap used for the non-rank-one analysis in
Section~\ref{subsec:highschool_examples}: draw $B$ graphs from the fitted Chung--Lu null with
the estimated weights, recompute $Z_4$ on each, and compare the observed value to the resulting
reference distribution. The resampled graphs target the finite-sample joint law of the class
statistics under the fitted null, including the cross-class dependence analyzed in
Lemma~\ref{lem:crosscov}. We therefore use this model-based calibration whenever the analytic
$N(0,1)$ approximation is in doubt.
\end{remark}

\begin{proof}[Proof of part (ii)]
Using $\sum_k w_k(\log k-\bar x_w)=0$, we split
\begin{align*}
  \hat\beta
  &= \frac{\sum_k w_k(\log k-\bar x_w)\,R_k}{S_{xx}}
  = (m+\delta) + \frac{\sum_k w_k(\log k-\bar x_w)\,\epsilon_k}{S_{xx}},
\end{align*}
where $\epsilon_k:=T_k-(m+\delta)\log k-\hat\mu_0$.  Since
$\hat\mu_0=O_P(1)$ and $T_k=(m+\delta)\log k+o_P(\log k)$ for every
$k\in\mathcal{K}^+_{\mathrm{slope}}$ by Theorem~\ref{thm:lln-pa}
(noting that $k\geq k_{\min}\to\infty$), we have $\epsilon_k=o_P(\log k)$.

By the Cauchy--Schwarz inequality,
\[
  \left|\sum_k w_k(\log k-\bar x_w)\,\epsilon_k\right|
  \;\leq\;
  \sqrt{S_{xx}}\;\sqrt{\sum_k w_k\,\epsilon_k^2}.
\]
The term $\sum_k w_k\epsilon_k^2/S_{xx}$ is a weighted mean of $o_P((\log k)^2)$
quantities; since $w_k(\log k)^2/S_{xx}\to 0$ for each $k$ (which follows from
$S_{xx}\to\infty$ and boundedness of the weight-proportion), a Chebyshev
argument gives $\sum_k w_k\epsilon_k^2/S_{xx}\xrightarrow{P}0$.  Therefore
$\hat\beta\xrightarrow{P}m+\delta$.

\noindent
Since $\hat\beta\xrightarrow{P}m+\delta>0$ and $S_{xx}\to\infty$,
\[
  Z_4 = \hat\beta\sqrt{S_{xx}} \;\geq\; \tfrac{m+\delta}{2}\,\sqrt{S_{xx}}
  \;\xrightarrow{P}\;+\infty,
\]
which gives $\mathbb{P}_{\mathrm{PA}}(Z_4>z_{1-\alpha})\to 1$.
\end{proof}

\subsection{Cross-class covariance: term-by-term remainder bounds}
\label{app:crosscov}

This subsection proves the remainder bounds used in Lemma~\ref{lem:crosscov}, controlling
$\operatorname{Cov}_{\mathbf W}(\psi_v,\psi_{v'})$ for $v\ne v'\in\mathcal H$ term by term.
We work under the Chung--Lu null conditional on $\mathbf W$, treat the hub set $\mathcal H$
and the design coefficients as fixed (condition (S1)), and use $p_{ij}=W_iW_j/L_n$.
Truncation at one and conditioning on the realized hub degrees perturb each $p_{ij}$ by a
factor $1+O(\varrho_n)$ \citep[Chap.~6]{van2024random}, which is absorbed into the
remainder throughout. Recall $\lambda_u=\sum_wp_{uw}=W_u(1+O(\varrho_n))$.

\paragraph{Two-level decomposition.}
Let $\mathcal G:=\sigma\bigl(\mathbf W,\{A_{vw}:v\in\mathcal H,\,w\in V\}\bigr)$ be generated
by the weights and all hub-incident edges. Given $\mathcal G$, every neighbor set
$\mathcal N(v)$ and denominator $D(v)=k$ are fixed, and for a non-hub $u$ the degree splits
as $D(u)=D^{\mathcal H}(u)+D^c(u)$ with $D^{\mathcal H}(u)=\sum_{v''\in\mathcal H}A_{uv''}$
$\mathcal G$-measurable and $D^c(u)=\sum_{w\notin\mathcal H}A_{uw}$ independent of
$\mathcal G$; a hub neighbor has $D^c\equiv0$ given $\mathcal G$. Then
\begin{equation}
\label{eq:AB-split}
  \operatorname{Cov}_{\mathbf W}(\psi_v,\psi_{v'})
  =\underbrace{\mathbb E_{\mathbf W}\!\bigl[\operatorname{Cov}(\psi_v,\psi_{v'}\mid\mathcal G)\bigr]}_{(A)}
  +\underbrace{\operatorname{Cov}_{\mathbf W}\!\bigl(\mathbb E[\psi_v\mid\mathcal G],\mathbb E[\psi_{v'}\mid\mathcal G]\bigr)}_{(B)}.
\end{equation}

\paragraph{Term (A): shared-neighbor and two-path (separable, annihilated by centering).}
Because the $\{D^c(u)\}$ are sums of independent non-hub edges that share only the single
edge $A_{uu'}$ for non-hub $u\ne u'$, the within-$\mathcal G$ covariance is \emph{exact}:
\[
  \operatorname{Cov}(\psi_v,\psi_{v'}\mid\mathcal G)
  =\frac{1}{kk'}\Bigl[
    \sum_{u\in\mathcal N(v)\cap\mathcal N(v')\setminus\mathcal H}\sigma_u^{c2}
    +\!\!\sum_{\substack{u\in\mathcal N(v)\setminus\mathcal H,\ u'\in\mathcal N(v')\setminus\mathcal H\\ u\ne u'}}\!\!p_{uu'}(1-p_{uu'})\Bigr],
\]
with $\sigma_u^{c2}=\sum_{w\notin\mathcal H}p_{uw}(1-p_{uw})\le\lambda_u$; the first sum is
the shared-neighbor term, the second the two-path term. Taking $\mathbb E_{\mathbf W}$, so
that $\mathbf 1\{u\in\mathcal N(v)\}\mapsto p_{vu}$,
\[
  (A)=\frac{W_vW_{v'}}{kk'}\,(\Theta_n+\Theta'_n)\,(1+O(\varrho_n)),
  \qquad
  \Theta_n=\frac{\sum_uW_u^2\sigma_u^{c2}}{L_n^2},\quad
  \Theta'_n=\frac{(\sum_uW_u^2)^2}{L_n^3},
\]
both $O(1/n)$ under $\mathbb E[W^3]<\infty$ (respectively $\mathbb E[W^2]<\infty$). Both are
separable in $(W_v/k,\,W_{v'}/k')$, hence so is their sum, and the argument in the proof of
Lemma~\ref{lem:crosscov} applies verbatim with $\Theta_n$ replaced by $\Theta_n+\Theta'_n$:
the exact identity $\sum_vD(v)(\log D(v)-\bar x_w^*)=0$ annihilates the separable leading
part, leaving
\begin{equation}
\label{eq:A-bound}
  \sum_{v\ne v'}\gamma_v\gamma_{v'}(A)=O_{\mathbb P}\!\Bigl(\tfrac{k_{\max}}{n}\Bigr)+o_{\mathbb P}(1).
\end{equation}

\paragraph{Term (B): hub paths (separable) and hub adjacency (non-separable).}
Here $\mathbb E[\psi_v\mid\mathcal G]=k^{-1}\sum_uA_{vu}h_u$ with
$h_u:=D^{\mathcal H}(u)+\bar D^c_u-\mu_0$ and $\bar D^c_u:=\sum_{w\notin\mathcal H}p_{uw}$, a
function of hub-incident edges. Its covariance with the primed version decomposes over
shared hub-incident edges into three types.
\begin{itemize}
\item[(B1)] \emph{Shared-neighbor through $D^{\mathcal H}$}: a common neighbor $u$ whose
  hub-degree $D^{\mathcal H}(u)$ enters both means. Since
  $\operatorname{Var}(D^{\mathcal H}(u))\le\lambda^{\mathcal H}_u:=\sum_{v''\in\mathcal H}p_{uv''}
  =W_u\,L_H/L_n$ with $L_H=\sum_{v''\in\mathcal H}W_{v''}$, this term is at most
  $L_H/L_n\le1$ times the shared-neighbor principal term of (A); it is separable in
  $(W_v/k,W_{v'}/k')$ and is annihilated by the same centering identity, contributing
  $O_{\mathbb P}(k_{\max}/n)+o_{\mathbb P}(1)$.
\item[(B2)] \emph{Hub paths}: a neighbor $u'$ of $v'$ that is itself adjacent to the hub $v$
  (edge $A_{u'v}$ shared between $A_{vu'}$ in the first mean and $D^{\mathcal H}(u')$ in the
  second), and symmetric configurations. Each such term carries an extra factor
  $p_{u'v}=O(\varrho_n)$ relative to (B1), so their aggregate is $O_{\mathbb P}(\varrho_n)$.
\item[(B3)] \emph{Direct hub adjacency}: the single shared edge $A_{vv'}$. This is the only
  contribution that is \emph{not} separable and \emph{not} annihilated by centering, and it
  requires the third-moment assumption. We bound it next.
\end{itemize}

\paragraph{Bounding (B3).} Using $D(v)=k$, $D(v')=k'$ ($\mathcal G$-measurable),
\[
  (\mathrm{B3})_{vv'}
  =\frac{(k-\mu_0)(k'-\mu_0)}{kk'}\,p_{vv'}(1-p_{vv'})
  =\frac{(k-\mu_0)(k'-\mu_0)}{kk'}\cdot\frac{W_vW_{v'}}{L_n}\,(1+O(\varrho_n)).
\]
Set $a_v:=\gamma_v(k_v-\mu_0)W_v/k_v$, so that
$\sum_{v\ne v'}\gamma_v\gamma_{v'}(\mathrm{B3})_{vv'}
=\tfrac1{L_n}\bigl[(\sum_va_v)^2-\sum_va_v^2\bigr](1+O(\varrho_n))$. We bound both pieces.

For the diagonal piece, recall $\sum_v\gamma_v^2\operatorname{Var}_{\mathbf W}(\psi_v)
=\tilde\sigma_\infty^2\sum_v\gamma_v^2/k_v=1$, i.e.\ $\sum_v\gamma_v^2/k_v
=1/\tilde\sigma_\infty^2$. With $|k_v-\mu_0|\le k_v$ and $W_v\le k_v(1+o(1))$,
$a_v^2\le\gamma_v^2k_v^2(1+o(1))$, so
\begin{equation}
\label{eq:B3-diag}
  \frac1{L_n}\sum_va_v^2
  \le\frac{1+o(1)}{L_n}\sum_v\Bigl(\tfrac{\gamma_v^2}{k_v}\Bigr)k_v^3
  \le\frac{(1+o(1))\,k_{\max}^3}{L_n\,\tilde\sigma_\infty^2}
  =O\!\Bigl(\tfrac{k_{\max}^3}{n}\Bigr).
\end{equation}
For the square of the sum, write
$\sum_va_v=\tfrac{1+o(1)}{\tilde\sigma_\infty^2\sqrt{S_{xx}^*}}\sum_k n(k)k(k-\mu_0)(\log k-\bar x_w^*)$,
and apply the Cauchy--Schwarz inequality in the $w_k^*=kn(k)/\tilde\sigma_\infty^2$ inner
product, whose $(\log k-\bar x_w^*)$-norm is $\sqrt{S_{xx}^*}$:
\[
  \Bigl|\sum_kw_k^*(k-\mu_0)(\log k-\bar x_w^*)\Bigr|
  \le\Bigl(\sum_kw_k^*(k-\mu_0)^2\Bigr)^{1/2}\sqrt{S_{xx}^*}.
\]
Hence, using $(k-\mu_0)^2\le k^2$ for $k\ge\mu_0$ (the finitely many classes with $k<\mu_0$
add $O(1)$) and $w_k^*=kn(k)/\tilde\sigma_\infty^2$,
\begin{equation}
\label{eq:B3-sq}
  \frac1{L_n}\Bigl(\sum_va_v\Bigr)^2
  \le\frac{(1+o(1))\sum_kw_k^*(k-\mu_0)^2}{L_n\,\tilde\sigma_\infty^2}
  \le\frac{\sum_{k\ge k_{\mathrm{slope}}(n)}n(k)k^3}{L_n\,\tilde\sigma_\infty^4}
  =\frac{\mathbb E_n\!\bigl[D^3\mathbf 1\{D\ge k_{\mathrm{slope}}(n)\}\bigr]}{\bar D\,\tilde\sigma_\infty^4}(1+o(1)).
\end{equation}
Because $\mathbb E[W^3]<\infty$ implies $\mathbb E[D^3]<\infty$ with $D^3$ uniformly
integrable, and $k_{\mathrm{slope}}(n)\to\infty$ by (S3), the tail expectation in
\eqref{eq:B3-sq} tends to $0$. Combining \eqref{eq:B3-diag}--\eqref{eq:B3-sq},
\begin{equation}
\label{eq:B3-bound}
  \sum_{v\ne v'}\gamma_v\gamma_{v'}(\mathrm{B3})_{vv'}
  =O_{\mathbb P}\!\Bigl(\tfrac{k_{\max}^3}{n}\Bigr)+o_{\mathbb P}(1)=o_{\mathbb P}(1).
\end{equation}

\paragraph{Conclusion.} Adding \eqref{eq:A-bound}, the (B1)--(B2) bounds, and
\eqref{eq:B3-bound},
\[
  \sum_{v\ne v'}\gamma_v\gamma_{v'}\operatorname{Cov}_{\mathbf W}(\psi_v,\psi_{v'})
  =O_{\mathbb P}\!\Bigl(\tfrac{k_{\max}}{n}\Bigr)
  +O_{\mathbb P}\!\Bigl(\tfrac{k_{\max}^3}{n}\Bigr)
  +O_{\mathbb P}(\varrho_n)+o_{\mathbb P}(1)
  =o_{\mathbb P}(1)
\]
under (S3). The only non-separable contribution is direct hub--hub adjacency (B3), and it is
controlled precisely by the finite-third-moment assumption $\mathbb E[W^3]<\infty$ already
maintained throughout, together with the slope window being a genuine upper tail. This is the
bound asserted in Lemma~\ref{lem:crosscov}. $\qed$

\section{Reddit interaction networks}
\label{sec:reddit_analysis}
We analyze monthly user interaction networks from \texttt{reddit.com}, a platform where users form topic-based discussion communities called subreddits. The dataset covers 2046 subreddits in 2014 \citep{hamilton2017loyalty} and is available from the Stanford Network Analysis Project (SNAP, \url{http://snap.stanford.edu/data/web-RedditNetworks.html}).

We focus on four subreddits that span a range of community sizes and behavioral contexts:
\begin{itemize}
	\item \textbf{NBA} (\texttt{r/nba}): discussions on NBA games, teams, and players;
	\item \textbf{CFB} (\texttt{r/CFB}): college football discussions;
	\item \textbf{Pics} (\texttt{r/pics}): user-submitted images and casual commentary;
	\item \textbf{Funny} (\texttt{r/funny}): humor-focused image and link sharing.
\end{itemize}
The July and August sports graphs contain $31522$ and $18139$ vertices. The June \texttt{r/pics} and \texttt{r/funny} graphs contain $123970$ and $149806$ vertices, with $400616$ and $312780$ edges. We project the directed replies to undirected simple graphs by dropping direction, collapsing repeated pairs, and removing self-loops. The analysis therefore concerns interaction incidence rather than frequency or conversational flow.

We apply Methods~1--3 to the four selected monthly Reddit interaction networks, testing the null hypothesis
\[
  H_0:\quad \mathbb{E}\{\bar{D}_k\mid D_\varnothing=k\}=\hat{\mu}_0
\]
that the observed local neighbor-degree structure is consistent with a
Chung--Lu IRG after plug-in estimation of the null parameters.  For each
network, $\hat{\mu}_0$ and $\hat{\tilde\sigma}_\infty$ are computed from its
empirical degree sequence.

The Reddit level window uses the empirical $95$th percentile of the
positive-degree \emph{vertices}, with all ties retained.  The resulting
thresholds for \texttt{r/nba}, \texttt{r/CFB}, \texttt{r/pics}, and
\texttt{r/funny} are $128$, $77$, $22$, and $14$.  They retain respectively
$1587$, $911$, $6538$, and $8167$ vertices and contain $490$, $277$, $248$,
and $191$ distinct degree classes.  Thus the reported $M$ is the number of
retained distinct classes, not $5\%$ of the number of distinct degrees.

Because the Reddit tails do not support relying on the
finite-third-moment approximation, we calibrate all four methods by a
fitted-null bootstrap.  We set the fitted Chung--Lu expected degrees to
$\hat w_i=D_i$ and draw $B_{\mathrm{boot}}=999$ independent graphs on
the observed vertex set.  For each bootstrap graph we refit the nuisance
parameters, reapply the same vertex-percentile level window and
distinct-degree slope window, and recompute the complete Method~1--4
pipeline.  For a scalar discrepancy $Q_j$ from Method $j$, the reported
bootstrap $p$-value is
\[
  p_j^{\mathrm{boot}}
  =
  \frac{1+\sum_{b=1}^{B_{\mathrm{boot}}}
  \mathbf 1\{Q_j^{*(b)}\geq Q_j^{\mathrm{obs}}\}}
  {B_{\mathrm{boot}}+1}.
\]
For the two-sided tests we use $Q_1=|Z_1|$,
$Q_2=-\log p_2^{\mathrm{ana}}$, $Q_3=-\log p_3^{\mathrm{ana}}$, and
$Q_4=|Z_4|$; the directional Method~4 test instead compares the signed
$Z_4$ statistics in the specified tail.  Tables~\ref{tab:reddit-nba}--%
\ref{tab:reddit-funny} report the resulting fitted-null bootstrap
$p$-values at $\alpha=0.05$.  For Method~4, let $c^*_{0.95}$ be the empirical
$0.95$ quantile of $|Z_4^*|$ and
$\widehat{\mathrm{se}}(\hat\beta)=|\hat\beta/Z_4|$; the reported
studentized-bootstrap interval is
$\hat\beta\pm c^*_{0.95}\widehat{\mathrm{se}}(\hat\beta)$.

With bootstrap calibration, Methods~1--3 all reject for \texttt{r/nba},
\texttt{r/CFB}, and \texttt{r/pics}.  For \texttt{r/funny}, the
maximum-degree test does not reject ($p=0.622$), whereas both pooled tests
reject ($p=0.006$), showing that the departure is distributed across the hub
window rather than concentrated at its single largest degree.
\begin{table}[htbp]
  \centering
  \caption{Test outcome for \texttt{r/nba} under the Chung-Lu null
    ($\hat{\mu}_0=279.51$, $\hat{\tilde\sigma}_\infty=338.72$,
    $\alpha=0.05$).  The pooled methods use the upper $5\%$ of
    positive-degree vertices, with $M$ denoting retained distinct classes;
    $p$-values are fitted-null bootstrap values.}
  \label{tab:reddit-nba}
  \begin{tabular}{p{0.32\textwidth}p{0.25\textwidth}cc}
    \toprule
    Method & Degree classes used & $p$-value & Decision \\
    \midrule
    Method~1 (maximum-degree hub) & $k^*=1845$, $n(k^*)=1$ & $1.00\times10^{-3}$ & Reject CL \\
    Method~2 (Bonferroni) & $k_{\mathrm{level}}(n)=128$, $M=490$ & $1.00\times10^{-3}$ & Reject CL \\
    Method~3 (Simes) & $k_{\mathrm{level}}(n)=128$, $M=490$ & $1.00\times10^{-3}$ & Reject CL \\
    \bottomrule
  \end{tabular}
\end{table}

\begin{table}[htbp]
  \centering
  \caption{Test outcome for \texttt{r/CFB} under the Chung-Lu null
    ($\hat{\mu}_0=151.52$, $\hat{\tilde\sigma}_\infty=195.30$,
    $\alpha=0.05$).  The pooled methods use the upper $5\%$ of
    positive-degree vertices; $p$-values are fitted-null bootstrap values.}
  \label{tab:reddit-cfb}
  \begin{tabular}{p{0.32\textwidth}p{0.25\textwidth}cc}
    \toprule
    Method & Degree classes used & $p$-value & Decision \\
    \midrule
    Method~1 (maximum-degree hub) & $k^*=1155$, $n(k^*)=1$ & $1.00\times10^{-3}$ & Reject CL \\
    Method~2 (Bonferroni) & $k_{\mathrm{level}}(n)=77$, $M=277$ & $1.00\times10^{-3}$ & Reject CL \\
    Method~3 (Simes) & $k_{\mathrm{level}}(n)=77$, $M=277$ & $1.00\times10^{-3}$ & Reject CL \\
    \bottomrule
  \end{tabular}
\end{table}

\begin{table}[htbp]
  \centering
  \caption{Test outcome for \texttt{r/pics} under the Chung--Lu null
    ($\hat{\mu}_0=41.80$, $\hat{\tilde\sigma}_\infty=85.35$,
    $\alpha=0.05$), full network.  The pooled methods use the upper $5\%$
    of positive-degree vertices; $p$-values are fitted-null bootstrap values.}
  \label{tab:reddit-pics}
  \begin{tabular}{p{0.32\textwidth}p{0.25\textwidth}cc}
    \toprule
    Method & Degree classes used & $p$-value & Decision \\
    \midrule
    Method~1 (maximum-degree hub) & $k^*=843$, $n(k^*)=1$ & $1.00\times10^{-3}$ & Reject CL \\
    Method~2 (Bonferroni) & $k_{\mathrm{level}}(n)=22$, $M=248$ & $1.00\times10^{-3}$ & Reject CL \\
    Method~3 (Simes) & $k_{\mathrm{level}}(n)=22$, $M=248$ & $1.00\times10^{-3}$ & Reject CL \\
    \bottomrule
  \end{tabular}
\end{table}

\begin{table}[htbp]
  \centering
  \caption{Test outcome for \texttt{r/funny} under the Chung-Lu null
    ($\hat{\mu}_0=25.29$, $\hat{\tilde\sigma}_\infty=53.11$,
    $\alpha=0.05$), full network.  The pooled methods use the upper $5\%$
    of positive-degree vertices; $p$-values are fitted-null bootstrap values.}
  \label{tab:reddit-funny}
  \begin{tabular}{p{0.32\textwidth}p{0.25\textwidth}cc}
    \toprule
    Method & Degree classes used & $p$-value & Decision \\
    \midrule
    Method~1 (maximum-degree hub) & $k^*=571$, $n(k^*)=1$ & $6.22\times10^{-1}$ & Do not reject \\
    Method~2 (Bonferroni) & $k_{\mathrm{level}}(n)=14$, $M=191$ & $6.00\times10^{-3}$ & Reject CL \\
    Method~3 (Simes) & $k_{\mathrm{level}}(n)=14$, $M=191$ & $6.00\times10^{-3}$ & Reject CL \\
    \bottomrule
  \end{tabular}
\end{table}

Method~4 uses the upper $25\%$ of distinct degrees and is reported two-sided. Its sign determines whether the level departure is positive, negative, or flat over hub degrees.
Table~\ref{tab:reddit-method4} reports the Method~4 estimates, and
Figure~\ref{fig:method4-real} displays the four residual profiles.  The
slopes are unchanged by the calibration switch, but their intervals,
$p$-values, and decisions are obtained from the fitted-null bootstrap.

\begin{table}[htbp]
  \centering
  \caption{Method~4 on the four Reddit networks under the Chung--Lu null,
    reported two-sided.  Intervals, $p$-values, and decisions use the
    fitted-null bootstrap; a negative slope indicates disassortative mixing
    and a positive slope the preferential-attachment direction
    ($\alpha=0.05$, full networks).}
  \label{tab:reddit-method4}
  \begin{tabular}{lccrcrl}
    \toprule
    Network & $k_{\mathrm{slope}}(n)$ & $M$ & $\hat\beta$ & $95\%$ boot.\ CI & $p_{\mathrm{boot}}$ & Decision \\
    \midrule
    \texttt{r/nba}   & 546 & 155 & $-53.30$ & [$-97.24$, $-9.35$] & $1.00\times10^{-3}$ & Reject CL (disassortative) \\
    \texttt{r/CFB}   & 297 & 89  & $-32.80$ & [$-56.95$, $-8.65$] & $1.00\times10^{-3}$ & Reject CL (disassortative) \\
    \texttt{r/pics}  & 224 & 68  & $+2.24$  & [$-0.18$, $4.66$] & $7.20\times10^{-2}$ & Inconclusive \\
    \texttt{r/funny} & 166 & 51  & $-3.37$  & [$-5.83$, $-0.91$] & $9.00\times10^{-3}$ & Reject CL (disassortative) \\
    \bottomrule
  \end{tabular}
\end{table}

\begin{figure}[htbp]
  \centering
  \includegraphics[width=0.85\linewidth]{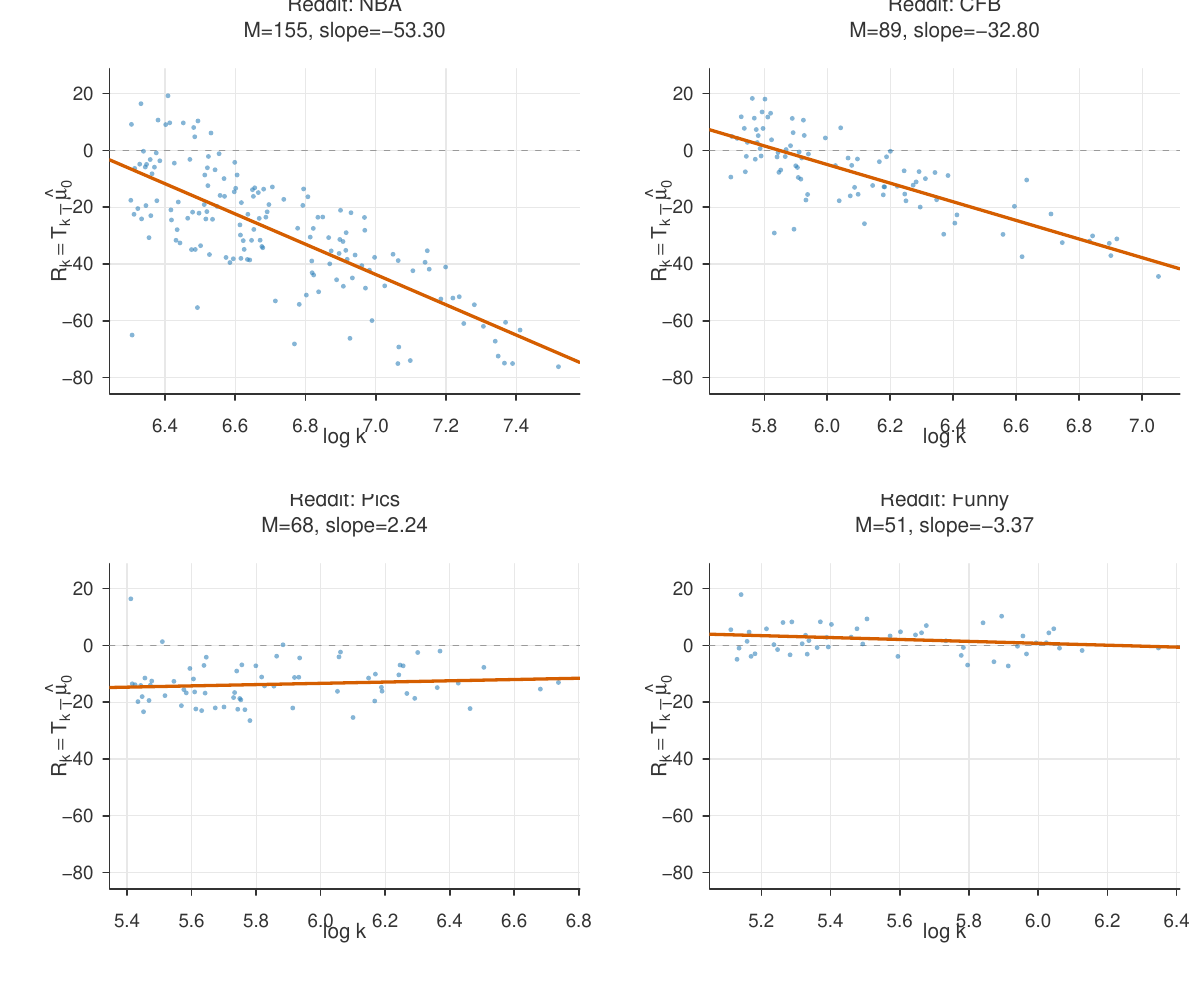}
  \caption{Standardized residuals $R_k=T_k-\hat\mu_0$ against $\log k$ for the
    four networks, with the weighted least-squares fit (red). \texttt{r/nba},
    \texttt{r/CFB}, and \texttt{r/funny} show negative slopes under the common
    upper-$25\%$ high-degree window; \texttt{r/pics} is inconclusive. None shows
    the positive slope of preferential attachment.}
  \label{fig:method4-real}
\end{figure}

The bootstrap-calibrated results therefore preserve the negative residual
direction for \texttt{r/nba}, \texttt{r/CFB}, and \texttt{r/funny}; the
\texttt{r/pics} slope remains inconclusive.  The level tests additionally
reject \texttt{r/pics}, so a level failure need not imply a monotone
high-degree trend.  None of the four networks provides significant evidence
for the positive residual slope predicted by preferential attachment under
the Chung--Lu benchmark.

\bibliographystyle{chicago}
\bibliography{ref.bib}

@article{cirkovic2024modeling,
  title={Modeling random networks with heterogeneous reciprocity},
  author={Cirkovic, Daniel and Wang, Tiandong},
  journal={Journal of Machine Learning Research},
  volume={25},
  number={10},
  pages={1--40},
  year={2024}
}

@incollection{benjamini2011recurrence,
  title={Recurrence of distributional limits of finite planar graphs},
  author={Benjamini, Itai and Schramm, Oded},
  booktitle={Selected Works of Oded Schramm},
  pages={533--545},
  year={2011},
  publisher={Springer}
}

@incollection{aldous2004objective,
  title={The objective method: probabilistic combinatorial optimization and local weak convergence},
  author={Aldous, David and Steele, J Michael},
  booktitle={Probability on discrete structures},
  pages={1--72},
  year={2004},
  publisher={Springer}
}

@book{van2024random,
  title={Random graphs and complex networks},
  author={Van Der Hofstad, Remco},
  volume={2},
  year={2024},
  publisher={Cambridge University Press}
}

@inproceedings{hamilton2017loyalty,
  title={Loyalty in online communities},
  author={Hamilton, William and Zhang, Justine and Danescu-Niculescu-Mizil, Cristian and Jurafsky, Dan and Leskovec, Jure},
  booktitle={Proceedings of the International AAAI conference on web and social media},
  volume={11},
  pages={540--543},
  year={2017}
}

@article{fournet2014contact,
  title={Contact patterns among high school students},
  author={Fournet, Julie and Barrat, Alain},
  journal={PLOS ONE},
  volume={9},
  number={9},
  pages={e107878},
  year={2014},
  publisher={Public Library of Science}
}

@article{mastrandrea2015contact,
  title={Contact patterns in a high school: A comparison between data collected using wearable sensors, contact diaries and friendship surveys},
  author={Mastrandrea, Rossana and Fournet, Julie and Barrat, Alain},
  journal={PLOS ONE},
  volume={10},
  number={9},
  pages={e0136497},
  year={2015},
  doi={10.1371/journal.pone.0136497},
  publisher={Public Library of Science}
}

@article{bollobas2007phase,
  title={The phase transition in inhomogeneous random graphs},
  author={Bollob{\'a}s, B{\'e}la and Janson, Svante and Riordan, Oliver},
  journal={Random Structures \& Algorithms},
  volume={31},
  number={1},
  pages={3--122},
  year={2007},
  publisher={Wiley}
}

@article{chung2002connected,
  title={Connected components in random graphs with given expected degree sequences},
  author={Chung, Fan and Lu, Linyuan},
  journal={Annals of Combinatorics},
  volume={6},
  number={2},
  pages={125--145},
  year={2002},
  publisher={Springer}
}

@article{barabasi1999emergence,
  title={Emergence of scaling in random networks},
  author={Barab{\'a}si, Albert-L{\'a}szl{\'o} and Albert, R{\'e}ka},
  journal={Science},
  volume={286},
  number={5439},
  pages={509--512},
  year={1999},
  publisher={American Association for the Advancement of Science}
}

@article{price1976general,
  title={A general theory of bibliometric and other cumulative advantage processes},
  author={Price, Derek J. de Solla},
  journal={Journal of the American Society for Information Science},
  volume={27},
  number={5},
  pages={292--306},
  year={1976},
  publisher={Wiley}
}

@article{bollobas2001degree,
  title={The degree sequence of a scale-free random graph process},
  author={Bollob{\'a}s, B{\'e}la and Riordan, Oliver and Spencer, Joel and Tusn{\'a}dy, G{\'a}bor},
  journal={Random Structures \& Algorithms},
  volume={18},
  number={3},
  pages={279--290},
  year={2001},
  publisher={Wiley}
}

@article{krapivsky2000connectivity,
  title={Connectivity of growing random networks},
  author={Krapivsky, P. L. and Redner, Sidney and Leyvraz, Francois},
  journal={Physical Review Letters},
  volume={85},
  number={21},
  pages={4629--4632},
  year={2000},
  publisher={APS}
}

@article{jeong2003measuring,
  title={Measuring preferential attachment in evolving networks},
  author={Jeong, Hawoong and Neda, Zoltan and Barab{\'a}si, Albert-L{\'a}szl{\'o}},
  journal={Europhysics Letters},
  volume={61},
  number={4},
  pages={567--572},
  year={2003},
  publisher={IOP Publishing}
}

@article{simes1986improved,
  title={An improved {B}onferroni procedure for multiple tests of significance},
  author={Simes, R John},
  journal={Biometrika},
  volume={73},
  number={3},
  pages={751--754},
  year={1986},
  publisher={Oxford University Press}
}

@article{clauset2009power,
  title={Power-law distributions in empirical data},
  author={Clauset, Aaron and Shalizi, Cosma Rohilla and Newman, Mark E J},
  journal={SIAM Review},
  volume={51},
  number={4},
  pages={661--703},
  year={2009},
  publisher={SIAM}
}

@book{petrov1975sums,
  title={Sums of Independent Random Variables},
  author={Petrov, Valentin V},
  year={1975},
  publisher={Springer-Verlag},
  address={Berlin}
}

@article{newman2002assortative,
  title={Assortative mixing in networks},
  author={Newman, Mark E J},
  journal={Physical Review Letters},
  volume={89},
  number={20},
  pages={208701},
  year={2002},
  publisher={APS}
}

@article{newman2003mixing,
  title={Mixing patterns in networks},
  author={Newman, Mark E J},
  journal={Physical Review E},
  volume={67},
  number={2},
  pages={026126},
  year={2003},
  publisher={APS}
}

@article{pastorsatorras2001dynamical,
  title={Dynamical and correlation properties of the {I}nternet},
  author={Pastor-Satorras, Romualdo and V{\'a}zquez, Alexei and Vespignani, Alessandro},
  journal={Physical Review Letters},
  volume={87},
  number={25},
  pages={258701},
  year={2001},
  publisher={APS}
}

@article{lei2016goodness,
  title={A goodness-of-fit test for stochastic block models},
  author={Lei, Jing},
  journal={The Annals of Statistics},
  volume={44},
  number={1},
  pages={401--424},
  year={2016},
  publisher={Institute of Mathematical Statistics}
}

@article{bickel2016hypothesis,
  title={Hypothesis testing for automated community detection in networks},
  author={Bickel, Peter J and Sarkar, Purnamrita},
  journal={Journal of the Royal Statistical Society Series B},
  volume={78},
  number={1},
  pages={253--273},
  year={2016},
  publisher={Wiley}
}

@article{le2017concentration,
  title={Concentration and regularization of random graphs},
  author={Le, Can M and Levina, Elizaveta and Vershynin, Roman},
  journal={Random Structures \& Algorithms},
  volume={51},
  number={3},
  pages={538--561},
  year={2017},
  publisher={Wiley}
}

@article{ji2016coauthorship,
  title={Coauthorship and citation networks for statisticians},
  author={Ji, Pengsheng and Jin, Jiashun},
  journal={The Annals of Applied Statistics},
  volume={10},
  number={4},
  pages={1779--1812},
  year={2016},
  doi={10.1214/15-AOAS896}
}

@article{ji2022cocitation,
  title={Co-citation and co-authorship networks of statisticians},
  author={Ji, Pengsheng and Jin, Jiashun and Ke, Zheng Tracy and Li, Wanshan},
  journal={Journal of Business \& Economic Statistics},
  volume={40},
  number={2},
  pages={469--485},
  year={2022},
  doi={10.1080/07350015.2021.1978469}
}

@article{newman2001random,
  title={Random graphs with arbitrary degree distributions and their applications},
  author={Newman, Mark E. J. and Strogatz, Steven H. and Watts, Duncan J.},
  journal={Physical Review E},
  volume={64},
  number={2},
  pages={026118},
  year={2001},
  doi={10.1103/PhysRevE.64.026118}
}

@article{norros2006conditionally,
  title={On a conditionally {P}oissonian graph process},
  author={Norros, Ilkka and Reittu, Hannu},
  journal={Advances in Applied Probability},
  volume={38},
  number={1},
  pages={59--75},
  year={2006},
  doi={10.1239/aap/1143936140}
}

@article{hunter2008goodness,
  title={Goodness of fit of social network models},
  author={Hunter, David R. and Goodreau, Steven M. and Handcock, Mark S.},
  journal={Journal of the American Statistical Association},
  volume={103},
  number={481},
  pages={248--258},
  year={2008},
  doi={10.1198/016214507000000446}
}

@article{yao2018average,
  title={Average nearest neighbor degrees in scale-free networks},
  author={Yao, Dong and van der Hoorn, Pim and Litvak, Nelly},
  journal={Internet Mathematics},
  volume={1},
  number={1},
  pages={1--38},
  year={2018},
  doi={10.24166/im.02.2018}
}

@article{barrat2005rate,
  title={Rate equation approach for correlations in growing network models},
  author={Barrat, Alain and Pastor-Satorras, Romualdo},
  journal={Physical Review E},
  volume={71},
  number={3},
  pages={036127},
  year={2005},
  doi={10.1103/PhysRevE.71.036127}
}

@article{krot2017assortativity,
  title={Assortativity in generalized preferential attachment models},
  author={Krot, Alexander and Ostroumova Prokhorenkova, Liudmila},
  journal={Internet Mathematics},
  volume={1},
  number={1},
  pages={1--16},
  year={2017},
  doi={10.24166/im.15.2017}
}

@article{broido2019scale,
  title={Scale-free networks are rare},
  author={Broido, Anna D. and Clauset, Aaron},
  journal={Nature Communications},
  volume={10},
  pages={1017},
  year={2019},
  doi={10.1038/s41467-019-08746-5}
}

@article{wan2017fitting,
  title={Fitting the linear preferential attachment model},
  author={Wan, Phyllis and Wang, Tiandong and Davis, Richard A. and Resnick, Sidney I.},
  journal={Electronic Journal of Statistics},
  volume={11},
  number={2},
  pages={3738--3780},
  year={2017},
  doi={10.1214/17-EJS1327}
}

@article{wang2019hill,
  title={Consistency of {H}ill estimators in a linear preferential attachment model},
  author={Wang, Tiandong and Resnick, Sidney I.},
  journal={Extremes},
  volume={22},
  number={1},
  pages={1--28},
  year={2019},
  doi={10.1007/s10687-018-0335-7}
}

@article{chenshao2004normal,
  title={Normal approximation under local dependence},
  author={Chen, Louis H. Y. and Shao, Qi-Man},
  journal={The Annals of Probability},
  volume={32},
  number={3},
  pages={1985--2028},
  year={2004},
  doi={10.1214/009117904000000450}
}

@article{sarkar1998probability,
  title={Some probability inequalities for ordered {MTP}$_2$ random variables: a proof of the {S}imes conjecture},
  author={Sarkar, Sanat K.},
  journal={The Annals of Statistics},
  volume={26},
  number={2},
  pages={494--504},
  year={1998},
  doi={10.1214/aos/1028144846}
}

@article{benjamini2001control,
  title={The control of the false discovery rate in multiple testing under dependency},
  author={Benjamini, Yoav and Yekutieli, Daniel},
  journal={The Annals of Statistics},
  volume={29},
  number={4},
  pages={1165--1188},
  year={2001},
  doi={10.1214/aos/1013699998}
}

\end{document}